\pdfoutput=1
\documentclass[a4paper,12pt]{amsart}
\usepackage{amsmath,amssymb,amsxtra,amscd,mathrsfs}
\usepackage[margin=3cm]{geometry} 
\usepackage{tikz} 
\usepackage{color} 
\usepackage{booktabs}
\usepackage{longtable}  
\usepackage{calc} 
\usepackage{hyperref} 
\usepackage[all]{xy}

\numberwithin{equation}{section} 

\newtheorem{theorem}{Theorem}[section] 
\newtheorem{lemma}[theorem]{Lemma} 
\newtheorem{proposition}[theorem]{Proposition} 
\newtheorem{corollary}[theorem]{Corollary} 
\newtheorem{conjecture}[theorem]{Conjecture}

\theoremstyle{definition}
\newtheorem{definition}[theorem]{Definition} 
\newtheorem{notation}[theorem]{Notation} 
\newtheorem{remark}[theorem]{Remark} 
\newtheorem{example}[theorem]{Example} 

\usepackage{xspace} 

\newcommand{\C}{\mathbb{C}}
\newcommand{\Z}{\mathbb{Z}} 
\newcommand{\N}{\mathbb{N}}
\newcommand{\Q}{\mathbb{Q}} 
\newcommand{\R}{\mathbb{R}}
\newcommand{\bS}{\mathbb{S}} 
\newcommand{\hbS}{\widehat{\bS}}
\newcommand{\tbS}{\widetilde{\bS}} 
\newcommand{\PP}{\mathbb{P}} 

\newcommand{\cO}{\mathcal{O}} 
\newcommand{\cS}{\mathcal{S}} 
\newcommand{\hcS}{\widehat{\cS}}
\newcommand{\cE}{\mathcal{E}} 
\newcommand{\cN}{\mathcal{N}}
\newcommand{\cL}{\mathcal{L}} 
\newcommand{\tcL}{\widetilde{\cL}}
\newcommand{\cH}{\mathcal{H}} 
\newcommand{\cI}{\mathcal{I}} 
\newcommand{\cM}{\mathcal{M}}

\newcommand{\cJ}{\mathcal{J}} 
\newcommand{\tcJ}{\widetilde{\cJ}}
\newcommand{\hcJ}{\widehat{\cJ}} 
\newcommand{\cF}{\mathcal{F}} 

\newcommand{\hD}{{\widehat{D}}}
\newcommand{\hvarphi}{\hat{\varphi}}
\newcommand{\homega}{\hat{\omega}} 
\newcommand{\hT}{{\widehat{T}}} 
\newcommand{\htau}{\hat{\tau}}
 
\newcommand{\hS}{{\widehat{S}}}
\newcommand{\hjmath}{{\hat{\jmath}}}
\newcommand{\hpi}{\hat{\pi}}
\newcommand{\hp}{{\hat{p}}} 
\newcommand{\halpha}{\hat{\alpha}} 
\newcommand{\hGamma}{\widehat{\Gamma}} 

\newcommand{\ovbeta}{\overline{\beta}}

\newcommand{\tX}{{\widetilde{X}}}
\newcommand{\td}{{\tilde{d}}}
\newcommand{\tQ}{{\widetilde{Q}}}
\newcommand{\te}{{\tilde{e}}} 
\newcommand{\tDelta}{{\widetilde{\Delta}}} 
\newcommand{\ttau}{{\tilde{\tau}}}
\newcommand{\tv}{{\tilde{v}}}

\newcommand{\bs}{\mathbf{s}} 
\newcommand{\bc}{\mathbf{c}} 
\newcommand{\bx}{\mathbf{x}} 
\newcommand{\bt}{\mathbf{t}}
\newcommand{\bbf}{\mathbf{f}}

\newcommand{\bg}{\mathbf{g}} 
\newcommand{\btau}{\boldsymbol{\tau}} 
\newcommand{\btheta}{\boldsymbol{\theta}}

\newcommand{\sfF}{\mathsf{F}}
\newcommand{\sfH}{\mathsf{H}} 

\newcommand{\frakm}{\mathfrak{m}} 
\newcommand{\frq}{\mathfrak{q}}  
\newcommand{\frs}{\mathfrak{s}} 
\newcommand{\frr}{\mathfrak{r}} 
\newcommand{\frM}{\mathfrak{M}}

\newcommand{\scrF}{\mathscr{F}}

\newcommand{\Bl}{\operatorname{Bl}} 
\newcommand{\NE}{\operatorname{NE}} 
\newcommand{\ovNE}{\operatorname{\overline{NE}}}
\newcommand{\NEN}{\NE_{\N}}
\newcommand{\Int}{\operatorname{Int}} 
\newcommand{\Spec}{\operatorname{Spec}} 
\newcommand{\Spf}{\operatorname{Spf}}
\newcommand{\id}{\operatorname{id}} 
\newcommand{\ev}{\operatorname{ev}} 
\newcommand{\pt}{\operatorname{pt}} 
\newcommand{\Hom}{\operatorname{Hom}} 
\newcommand{\End}{\operatorname{End}} 
\newcommand{\rank}{\operatorname{rank}}
\newcommand{\ch}{\operatorname{ch}} 
\newcommand{\Ker}{\operatorname{Ker}}

\newcommand{\QDM}{\operatorname{QDM}} 
 
\newcommand{\Pic}{\operatorname{Pic}} 
\newcommand{\Tot}{\operatorname{Tot}} 
\newcommand{\pr}{\operatorname{pr}} 
\newcommand{\FT}{\operatorname{FT}}
\newcommand{\Res}{\operatorname{Res}} 
 
\newcommand{\LT}{\operatorname{LT}}
\newcommand{\ad}{\operatorname{ad}} 
\newcommand{\QH}{\operatorname{QH}} 

\newcommand{\iu}{\mathtt{i}}
\newcommand{\ttt}{\mathtt{t}} 
\newcommand{\st}{{\rm s}}
\newcommand{\QW}{\mathcal{Q}}

\newcommand{\bPsi}{\boldsymbol{\Psi}}

\newcommand{\bigboxplus}{
  \mathop{
    \vphantom{\bigoplus} 
    \mathchoice
      {\vcenter{\hbox{\resizebox{\widthof{$\displaystyle\bigoplus$}}{!}{$\boxplus$}}}}
      {\vcenter{\hbox{\resizebox{\widthof{$\bigoplus$}}{!}{$\boxplus$}}}}
      {\vcenter{\hbox{\resizebox{\widthof{$\scriptstyle\oplus$}}{!}{$\boxplus$}}}}
      {\vcenter{\hbox{\resizebox{\widthof{$\scriptscriptstyle\oplus$}}{!}{$\boxplus$}}}}
  }\displaylimits 
}

\def\corr#1{\left\langle#1 \right\rangle} 
\def\parfrac#1#2{\frac{\partial #1}{\partial #2}} 

\begin{document} 

\title{Quantum cohomology of blowups} 
\author{Hiroshi Iritani}
\email{iritani@math.kyoto-u.ac.jp} 
\address{Department of Mathematics, Graduate School of Science, 
Kyoto University, Kitashirakawa-Oiwake-cho, Sakyo-ku, 
Kyoto, 606-8502, Japan}


\begin{abstract} 
We prove a decomposition theorem of the quantum cohomology $D$-module of the blowup $\tX=\Bl_Z X$ of a smooth projective variety $X$ along a smooth subvariety $Z$. The main tools we use are shift operators and Fourier analysis for equivariant quantum cohomology. 
\end{abstract} 
\maketitle 


\section{Introduction} 
Let $X$ be a smooth projective variety over $\C$ and let $Z\subset X$ be a smooth subvariety of codimension $r\ge 2$. It is well-known that the cohomology of the blowup $\tX$ of $X$ along $Z$ decomposes additively as 
\[
H^*(\tX) \cong H^*(X) \oplus H^*(Z)^{\oplus (r-1)}.  
\]
The goal of the present paper is to prove an analogous (but multiplicative) decomposition for quantum cohomology rings (or $F$-manifolds); the decomposition also lifts to a formal decomposition of quantum cohomology $D$-modules. 
\subsection{Decomposition theorem} 
The \emph{quantum cohomology} of a smooth projective variety $X$ is a formal family of supercommutative product structures $\star_\tau$ on $H^*(X)\otimes \C[\![Q]\!]$ parametrized by $\tau \in H^*(X)$ (see \S\ref{subsec:qcoh_qconn}). 
The coefficient ring $\C[\![Q]\!]$ is the completed monoid ring of the set $\NEN(X)\subset H_2(X,\Z)$ of effective curves, called the \emph{Novikov ring}. 
The \emph{quantum $D$-module} $\QDM(X)$ is the module $H^*(X)\otimes \C[z][\![Q,\tau]\!]$ equipped with the following mutually supercommuting operators from $\QDM(X)$ to $z^{-1}\QDM(X)$, called the \emph{quantum connection}:  
\begin{align*} 
\nabla_{\tau^i} &= \partial_{\tau^i} + z^{-1} (\phi_i \star_\tau), \\
\nabla_{z\partial_z} & = z \partial_z - z^{-1} (E_X\star_\tau) + \mu_X, \\ 
\nabla_{\xi Q\partial_Q} & = \xi Q \partial_Q + z^{-1} (\xi \star_\tau), 
\end{align*} 
where $\{\tau^i\}$ denotes coordinates on $H^*(X)$ dual to a basis $\{\phi_i\} \subset H^*(X)$, $E_X$ is the Euler vector field, $\mu_X$ is the grading operator and $\xi \in H^2(X)$ (see \S\ref{subsec:qcoh_qconn} for the details). We view $\nabla$ as a flat connection on the trivial $H^*(X)$-bundle over the $(z,Q,\tau)$-space. 
The variable $z$ here arises as the equivariant parameter of the $\C^\times$-action rotating the domain of maps $\PP^1 \to X$. The quantum $D$-module is also equipped with the pairing $P_X(f,g) = \int_X f(-z) \cup g(z)$ which is compatible with the quantum connection $\nabla$. 

In order to compare the quantum $D$-modules of $X$, $\tX$, $Z$, we embed their Novikov rings into a common ring. 
Let $Q$, $\tQ$, $Q_Z$ denote the Novikov variables of $X$, $\tX$ and $Z$ respectively. 
We introduce another Novikov variable $\frq$ and extend the Novikov rings of $X, \tX, Z$ to $\C(\!(\frq^{-1/\frs})\!)[\![Q]\!]$ as follows: 
\begin{equation}
\label{eq:extension_of_Novikov_rings}
\begin{aligned}
\C[\![Q]\!] & \hookrightarrow \C(\!(\frq^{-1/\frs})\!)[\![Q]\!]  \qquad & & \text{in an obvious way}\\ 
\C[\![\tQ]\!] & \hookrightarrow \C(\!(\frq^{-1/\frs})\!)[\![Q]\!] & & \text{$\tQ^\td \mapsto Q^{\varphi_*\td} \frq^{-[D]\cdot \td}$} \\ 
\C[\![Q_Z]\!] & \to \C(\!(\frq^{-1/\frs})\!)[\![Q]\!] & & \text{$Q_Z^d \mapsto Q^{\imath_*d} \frq^{-\rho_Z \cdot d/(r-1)}$}
\end{aligned} 
\end{equation}
where $\frs$ equals $r-1$ or $2(r-1)$ depending on whether $r$ is even or odd (see \eqref{eq:frs}), $\rho_Z = c_1(\cN_{Z/X})$ and $\imath\colon Z\to X$ is the inclusion map. 
The variable $\frq$ represents the class of a line in the exceptional divisor $D \subset \tX$ that maps to a point in $Z$. 
Let $\QDM(X)^{\rm la}$, $\QDM(\tX)^{\rm la}$, $\QDM(Z)^{\rm la}$ denote the base changes\footnote{The superscript `la'  means formal \emph{Laurent} forms.} of the quantum $D$-modules of $X$, $\tX$, $Z$ to $\C[z](\!(\frq^{-1/\frs})\!)[\![Q]\!]$ via the above ring extension \eqref{eq:extension_of_Novikov_rings}. 

\begin{theorem}[see Theorem \ref{thm:decomposition_fd} for the details] 
\label{thm:decomposition_introd} 
There exist a formal invertible change of variables $H^*(\tX) \to H^*(X) \oplus H^*(Z)^{\oplus (r-1)}$,  $\ttau \mapsto \left(\tau(\ttau), \{\varsigma_j(\ttau)\}_{0\le j\le r-2}\right)$ defined over $\C(\!(\frq^{-\frac{1}{r-1}})\!)[\![Q]\!]$ and an isomorphism 
\[
\Psi \colon \QDM(\tX)^{\rm la} \to \tau^*\QDM(X)^{\rm la} \oplus \bigoplus_{j=0}^{r-2} \varsigma_j^* \QDM(Z)^{\rm la}  
\]
that commutes with the quantum connection. Moreover, $\Psi$ intertwines the pairing $P_\tX$ with $P_X \oplus P_Z^{\oplus (r-1)}$. 
\end{theorem}

\begin{corollary} 
\label{cor:F-manifolds} 
The map $\ttau \mapsto (\tau(\ttau), \{\varsigma_j(\ttau)\}_{0\le j\le r-2})$ defines an isomorphism of quantum cohomology $F$-manifolds over $\C(\!(\frq^{-\frac{1}{r-1}})\!)[\![Q]\!]$, i.e.~the differential of this map defines a ring isomorphism $
(H^*(\tX), \star_\ttau) \cong (H^*(X), \star_{\tau(\ttau)}) \oplus 
\bigoplus_{j=0}^{r-2} (H^*(Z), \star_{\varsigma_j(\ttau)})$. It also preserves the Euler vector fields. 
\end{corollary} 

\begin{remark} 
Throughout the paper, we work with completions in the category of graded rings or modules. The power series rings such as $\C[\![Q]\!]$, $\C[z](\!(\frq^{-1/\frs})\!)[\![Q]\!]$ should be understood in the graded sense. For example,  $\C(\!(\frq^{-\frac{1}{r-1}})\!)[\![Q]\!]$ is the same as $\C[\frq^{\pm \frac{1}{r-1}}][\![Q]\!]$ since $\frq$ has positive degree. 
See \S\ref{subsec:formal_power_series}. 
\end{remark} 

\begin{remark} 
The pullbacks $\tau^*\QDM(X)^{\rm la}$, $\varsigma_j^*\QDM(Z)^{\rm la}$ are defined to be the modules $H^*(X)\otimes \C[z](\!(\frq^{-1/\frs})\!)[\![Q,\ttau]\!]$, $H^*(Z) \otimes \C[z](\!(\frq^{-1/\frs})\!)[\![Q,\ttau]\!]$, respectively, equipped with the pulled-back quantum connections. See \eqref{eq:qconn_QDMtXla}, \eqref{eq:qconn_QDMXla}, \eqref{eq:qconn_QDMZla} and the discussion there. 
\end{remark} 

\begin{remark} 
The decomposition $\Psi$ in Theorem \ref{thm:decomposition_introd} is defined over $\C[z](\!(\frq^{-1/\frs})\!)[\![Q,\ttau]\!]$. The base ring can also be written as $\C[\frq^{\pm 1/\frs}][\![Q,\ttau]\!][\![z]\!]$ because of our convention on the graded completion (where $\deg \frq = 2(r-1)$, $\deg z = 2$). Although quantum cohomology is expected to have convergent structure constants in general, we do not expect that the decomposition $\Psi$ is analytic in all the variables $\frq, Q, \ttau, z$. Instead, we expect that $\Psi$ can be analytified to an isomorphism over the ring $\cO^{\rm an}[\![z]\!]$ of formal power series in $z$ with coefficients in the sheaf $\cO^{\rm an}$ of analytic functions in $\frq, Q,\ttau$. 
In fact, we showed \cite{Iritani:discrepant} for toric blowups that the Stokes structure of $\QDM(\tX)$ arising from the irregular singularities at $z=0$ does not admit an orthogonal decomposition, but a \emph{semiorthogonal decomposition} induced by Orlov's one \cite{Orlov:proj} for derived categories of coherent sheaves 
\begin{equation} 
\label{eq:Orlov_SOD} 
D^b_{\rm coh}(\tX) \cong \Bigl\langle D^b_{\rm coh}(X), D^b_{\rm coh}(Z)_0, \cdots, D^b_{\rm coh}(Z)_{r-2} \Bigr \rangle 
\end{equation} 
via the $\hGamma$-integral structure, where $D^b_{\rm coh}(Z)_0,\dots,D^b_{\rm coh}(Z)_{r-2}$ are copies of $D^b_{\rm coh}(Z)$. We refer the reader to \cite{Sanda-Shamoto}, \cite[\S 8]{Iritani:discrepant}, \cite{Kontsevich:Miami2020}, \cite{Iritani:Nottingham}, \cite[\S 3.3]{Iritani:ICM} for related conjectures in the analytic setting. 
See also \cite{Halpern-Leistner:nc_mmp} for far-reaching conjectures in this direction. 
\end{remark} 

\begin{remark} 
There are many blowup formulae for Gromov-Witten invariants, see e.g.~\cite{Hu:blowup, Gathmann:blowup, Hu-Li-Ruan:birational_cobordism, Lai:blowup, Manolache:virtual_pullback, Fan:Chern_projective}, but 
our results do not rely on these formulas. The geometric ingredients of the proof are the virtual localization formula \cite{Kontsevich:enumeration, Graber-Pandharipande} and the quantum Riemann-Roch theorem \cite{Coates-Givental}.  
\end{remark} 

\begin{remark}[related works]  
(1) Bayer \cite[Lemma 3.4.2]{Bayer:semisimple} showed that the quantum cohomology ring of a blowup of $X$ at a point decomposes into the quantum cohomology of $X$ and $\dim X-1$ many idempotent factors at certain quantum parameters. 
Recently, Milanov-Xia \cite{Milanov-Xia} showed that the exceptional objects $\cO_D(k)$ in $D^b_{\rm coh}(\tX)$ correspond to these idempotent factors under the $\hGamma$-integral structure (in the language of the second structure connection). 

(2) It is expected \cite{Kontsevich:ICM} that quantum cohomology arises as the Hochschild cohomology of the Fukaya category. Sanda \cite{Sanda:thesis, Sanda:computation_QC} and Venugopalan-Woodward-Xu \cite{VWX:Fukaya_blowup} deduced a decomposition of quantum cohomology from that of Fukaya categories for certain blowups. 

(3) Many people have studied quantum cohomology of toric blowups or flips. Gonz\'alez-Wooward \cite{Gonzalez-Woodward:tmmp} described a decomposition of quantum cohomology rings under a running of the toric minimal model programme. Acosta-Shoemaker \cite{Acosta-Shoemaker:blowup_LG} found a relationship between the $I$-functions of toric flips via asymptotic expansion. 
Lee-Lin-Wang \cite{Lee-Lin-Wang:flips_I} and the author \cite{Iritani:discrepant} proved a formal decomposition of quantum $D$-modules under toric flips; in the latter work, it was shown that an analytic lift of the formal decomposition corresponds to the Orlov decomposition \eqref{eq:Orlov_SOD} for weighted blowups of weak-Fano toric orbifolds along toric suborbifolds. Lee-Lin-Wang \cite{Lee-Lin-Wang:StringMath15, Wang:Gakushuin2022} also announced certain blowup formulas beyond the toric case. 

(4) Coates-Lutz-Shafi \cite{Coates-Luts-Shafi:blowup} obtained the $I$-function of a blowup along a complete intersection defined by convex line bundles. 

(5) Kontsevich has given several talks \cite{Kontsevich:HSE2019, Kontsevich:Miami2020, Kontsevich:Simons2021}  on joint project with Ludmil Katzarkov, Tony Pantev and Tony Yue Yu, where they conjectured a formal decomposition of quantum connections for blowups (similar to Theorem \ref{thm:decomposition_introd}) and discovered remarkable applications to birational geometry. 
They, along with \cite{HYZZ:framing}, also highlighted the unique reconstructibility of genus-zero Gromov-Witten invariants of $\tX$ from those of $X$ and $Z$, based on this decomposition theorem. We describe the initial condition for this reconstruction in \S\ref{subsec:initial_conditions}. 
\end{remark}

\subsection{Equivariant quantum cohomology and Fourier transformation} 
\label{subsec:Fourier_introd}
We prove the decomposition theorem (Theorem \ref{thm:decomposition_introd}) using \emph{Fourier analysis} on the equivariant quantum cohomology of the master space $W = \Bl_{Z\times \{0\}} (X\times \PP^1)$, the blowup of $X\times \PP^1$ along $Z\times \{0\}$. 
Both $X$ and its blowup $\tX$ arise as the Geometric Invariant Theory (GIT) quotient of $W$ by the  $\C^\times$-action on the $\PP^1$-factor (see \S\ref{subsec:GIT}). 
Teleman \cite{Teleman:gauge_mirror} conjectured a relationship between equivariant quantum cohomology and quantum cohomology of the GIT quotients in terms of the Seidel representation. 
Following his conjecture, we connect the quantum $D$-modules of $X$ and $\tX$ with the $\C^\times$-equivariant quantum $D$-module of $W$ via Fourier transformation. 

We explain an expected Fourier duality of quantum $D$-modules derived from Teleman's conjecture. Suppose that a smooth projective variety $X$ is equipped with an algebraic $T$-action, where $T\cong (\C^\times)^{\rank T}$. 
Then the cocharacter lattice $\Hom(\C^\times,T) \cong H_2^T(\pt,\Z)$ acts on the equivariant quantum $D$-module $\QDM_T(X)$ by the so-called shift operators of equivariant parameters \cite{Okounkov-Pandharipande:Hilbert,Maulik-Okounkov:qcoh_qgroup, BMO:Springer, Iritani:shift,LJones:shift_symp, GMP:nil-Hecke}: it is a lift of the Seidel representation \cite{Seidel:pi1} on quantum cohomology to $D$-modules. 
In this paper we combine the shift operators with the Novikov variables and define the action of the group $N_1^T(X)$ of equivariant curve classes, which fits into the exact sequence (see \eqref{eq:equiv_homology_seq}) 
\[
\begin{CD} 
0 @>>> N_1(X) @>>> N_1^T(X) @>>> H_2^T(\pt,\Z) @>>> 0, 
\end{CD} 
\]
where $N_1(X) \subset H_2(X,\Z)$ is the group of classes of algebraic curves. For each $\beta \in N_1^T(X)$, we define an operator $\hbS^\beta(\tau) \colon \QDM_T(X) \to \QDM_T(X)[Q^{-1}]$ satisfying the commutation relation (see Definition \ref{def:shift}): 
\begin{equation} 
\label{eq:commutation_relation} 
[\lambda, \hbS^\beta(\tau)] = z (\lambda\cdot \ovbeta) \hbS^\beta(\tau) 
\end{equation} 
with $\lambda \in H^2_T(\pt)$, where $\ovbeta\in H_2^T(\pt)$ is the image of $\beta\in N_1^T(X)$.  The operator $\hbS^\beta(\tau)$ equals the Novikov variable $Q^\beta$ if $\beta \in N_1(X)$. 
The relation \eqref{eq:commutation_relation} can be interpreted in two different ways: we can either view $\hbS^\beta(\tau)$ as a \emph{shift operator} of equivariant parameters $\lambda \to \lambda - z (\lambda \cdot \ovbeta)$, or view $\lambda$ as a \emph{differential operator} in the $\hbS$-variables (like $z \hbS \partial_{\hbS}$). These viewpoints are related by Fourier transformation.  
Teleman's conjecture suggests that $\QDM_T(X)[Q^{-1}]$ viewed as a module over $\C[N_1^T(X)]$ via the Seidel representation corresponds to $\QDM(X/\!/T)$ in such a way that the action of $\C[N_1^T(X)]$ corresponds to the Novikov variables for $X/\!/T$ and the action of equivariant parameters $\lambda$ corresponds to the quantum connection for $X/\!/T$. We note that $\C[N_1^T(X)]$ is an extension of $\C[N_1(X/\!/T)]$ via the dual Kirwan map $\kappa^* \colon N_1(X/\!/T) \to N_1^T(X)$. 

More precisely, we need to take into consideration the cone of curves. We have a monoid $\NEN^T(X) \subset N_1^T(X)$ that preserves the lattice $\QDM_T(X)$ under the Seidel representation (Proposition \ref{prop:module_over_extended_shift}). It turns out \cite{Iritani-Sanda:reduction} that the cone generated by $\NEN^T(X)$ is dual to the so-called \emph{$T$-ample cone} $C_T(X) \subset N^1_T(X)_\R$, the cone spanned by the equivariant first Chern classes $c_1^T(L)$ of $T$-equivariant ample line bundles $L\to X$ whose stable loci are nonempty. 
The cone $C_T(X)$ is divided into chambers depending on GIT quotients \cite{Dolgachev-Hu:VGIT, Thaddeus:GIT, Ressayre:GIT}. Let $Y= X/\!/T$ be a smooth\footnote{This means that there are no strictly semistable points and that $T$ acts freely on the stable locus.} GIT quotient and write  $C_Y\subset C_T(X)$ for the open subcone (chamber) corresponding to $Y$. Roughly speaking, we expect that restricting and completing $\QDM_T(X)$ along the affine chart $\Spec (\C[C_Y^\vee \cap N_1^T(X)]) \subset \Spec  (\C[\NEN^T(X)])$ gives rise to $\QDM(Y)$, where $C_Y^\vee \cap N_1^T(X)$ denotes the set of $\beta \in N_1^T(X)$ whose image in $N_1^T(X)_\R$ lies in the dual cone $C_Y^\vee$ of $C_Y$. We can state a precise conjecture in terms of the $J$-function (a solution of the quantum $D$-module) and the Givental cone as follows: 

\begin{conjecture}
\label{conj:reduction} 
Let $Y = X/\!/T$ be a smooth GIT quotient of a smooth projective $T$-variety $X$. 
Let $J_X(\tau)$ be the $T$-equivariant $J$-function \eqref{eq:J-function} of $X$ and let $\hS$ denote a formal variable of the group ring $\C[N_1^T(X)]$. 
Then the sum 
\begin{equation} 
\label{eq:I-function} 
I = z \sum_{[\beta] \in N_1^T(X)/N_1(X)} \kappa_Y\left(\hcS^{-\beta} J_X(\tau)\right) \hS^\beta
\end{equation} 
is supported on $C_Y^\vee \cap N_1^T(X)$ as an $\hS$-power series and gives a point on the Givental cone of $Y$ defined over the extension $\C[\![C_Y^\vee \cap N_1^T(X)]\!]$ of the Novikov ring of $Y$. Here $\hcS^\beta$ is the shift operator  on the Givental space $($see Definition $\ref{def:shift_Givental}$$)$ and $\kappa_Y \colon H^*_T(X) \to H^*(Y)$ is the Kirwan map. 
\end{conjecture} 

The formulation of this conjecture has been worked out in joint work with Fumihiko Sanda: we plan to give a more detailed account in \cite{Iritani-Sanda:reduction}. 

\begin{remark} 
Each summand $\kappa_Y(\hcS^{-\beta} J_X(\tau)) \hS^\beta$ in \eqref{eq:I-function} depends only on the class $[\beta]$ in  $N_1^T(X)/N_1(X) \cong H_T^2(\pt,\Z)$. 
The well-definedness of $\kappa_Y(\hcS^{-\beta} J_X(\tau))$ follows from the fact that $\hcS^{-\beta}$ preserves the tangent spaces of the Givental cone up to localization in Novikov variables (see the proof of Proposition \ref{prop:support_conjecture}). Note that the Novikov ring $\C[\![\NEN(Y)]\!]$ of $Y$ is embedded into $\C[\![C_Y^\vee \cap N_1^T(X)]\!]$ via the dual Kirwan map $\kappa_Y^* \colon \NEN(Y) \to C_Y^\vee \cap N_1^T(X)$. 
\end{remark} 

\begin{remark} 
The formula \eqref{eq:I-function} gives a discrete (partial) Fourier transform of $J_X(\tau)$: it transforms functions in $(\lambda,Q)$ into functions in $\hS=(S,Q)$ and intertwines the shift operator $\hcS^\beta$ with multiplication by the monomial $\hS^\beta$ and the equivariant parameter $\lambda$ with a differential operator in $\hS$, see \eqref{eq:properties_discrete_Fourier}. 
\end{remark} 

\begin{example} 
Givental's mirror theorem \cite{Givental:toric_mirrorthm} for toric varieties $\C^n/\!/T$ can be viewed as an example of this conjecture (although $\C^n$ is non-compact). 
We recall that Givental \cite{Givental:homological} heuristically obtained his $I$-function as a Fourier transform of the Floer fundamental cycle. 
\end{example} 

\begin{example} 
Let $V\to X$ be a vector bundle such that $V^\vee$ is generated by global sections and consider the fibrewise scalar $\C^\times$-action on $V$. In joint work \cite{Iritani-Koto:projective_bundle} with Koto, we showed that an analogue of the Conjecture \ref{conj:reduction} holds for the GIT quotient $V/\!/\C^\times = \PP(V)$ of the non-compact space $V$. 
\end{example} 

In this paper, we need the following special case of Conjecture \ref{conj:reduction}. 

\begin{theorem}[Corollary \ref{cor:Fourier_transform_JW_GIT}, Appendix \ref{append:divisor_reduction}] 
\label{thm:divisor_reduction} 
Let $X$ be a smooth projective variety with a $\C^\times$-action. 
Let $Y\subset X$ be a $\C^\times$-fixed divisor whose normal bundle has $\C^\times$-weight $1$ or $-1$. 
Then the GIT quotient $X/\!/\C^\times =Y$ satisfies Conjecture $\ref{conj:reduction}$. 
\end{theorem}

\subsection{Global K\"ahler moduli space for blowups}

\begin{figure}[t]
  \centering    
  \begin{tikzpicture} 
\shadedraw [opacity =0.5, left color =green, right color =white] (2.5,-1.3) -- (-0.5,0.5) -- (3,0.5); 
\shadedraw [opacity=0.5, left color =yellow, right color =white]  (3,0.5) -- (-0.5,0.5) -- (2.5,2.3); 
\draw (3,0.5) -- (-0.5,0.5) -- (2.5,-1.3); 
\draw (-0.5,0.5) -- (2.5,2.3); 
\draw (2.3,-0.3) node {\small $C_\tX$}; 
\draw (2.3,1.3) node {\small $C_X$}; 
\draw (1.5,-1.7) node {\small GIT fan};

\draw (6,2) ..controls (6.8,1.65) and (7.5,1.4) .. (8,1.3) .. controls (7.9,0.7) and (7.9,0.3) .. (8,-0.3) .. controls (7.5,-0.4) and (6.7,-0.5) .. (6,-1);  
\draw [very thick] (8,1.3) .. controls (7.9,0.7) and (7.9,0.3) .. (8,-0.3); 
\shadedraw [opacity=0.2, right color=black, left color=white] (6,2) ..controls (6.8,1.65) and (7.5,1.4) .. (8,1.3) .. controls (7.9,0.7) and (7.9,0.3) .. (8,-0.3)  .. controls (7.5,-0.4) and (6.7,-0.5) .. (6,-1);  

\draw (7,0.5) node {$\frM$}; 
\draw [very thick, ->] (7.925,0.5)--(7.925,0.5); 
\draw (8.2,0.4) node {$\frq$}; 
  
  \filldraw (8,1.3) circle [radius=0.05]; 
  \filldraw (8,-0.3) circle [radius =0.05]; 
  
\draw (8.7,1.4) node {\footnotesize $X$-cusp}; 
\draw (8.7,-0.4) node {\footnotesize $\tX$-cusp}; 

\draw (7.7,-1.76) node {\small toric variety (global K\"ahler moduli)}; 
\end{tikzpicture} 
\caption{The GIT fan on the $T$-ample cone $C_T(W)$ and the associated ``toric variety'' $\frM = \frM_X \cup \frM_{\tX}$. See also Figure \ref{fig:Mori_cones}. The variable $\frq$, which is written also as $y S^{-1}$ in the main body of the text, comes from a shift operator.} 
\label{fig:global} 
\end{figure}  

As the above examples suggest, \emph{the equivariant quantum $D$-modules play the role of mirrors for the GIT quotients}. Applying Theorem \ref{thm:divisor_reduction} to the master space $W = \Bl_{Z\times \{0\}} (X\times \PP^1)$, we find that Conjecture \ref{conj:reduction} holds for the two GIT quotients $X, \tX$ of $W$. Let $T=\C^\times$. The $T$-ample cone $C_T(W)$ contains two open chambers $C_X$, $C_\tX$ and we obtain a global ``K\"ahler moduli space'' by gluing $\frM_X = \Spec \C[C_{X,\N}^\vee]$ and $\frM_\tX = \Spec \C[C_{\tX,\N}^\vee]$ (see Figure \ref{fig:global}; see also Remark \ref{rem:global}). The quantum $D$-modules of $X$ and $\tX$ live on the charts $\frM_X$ and $\frM_\tX$ respectively and arise from $\QDM_T(W)$ by Fourier transformation. More precisely, we have the following description for $\QDM(\tX)$: 

\begin{theorem}[see Theorem \ref{thm:Fourier_isom} for precise statements] 
\label{thm:Fourier_duality_introd}
Let $\QDM_T(W)_\tX\sphat$ be a completion of the restriction of $\QDM_T(W)$ to the affine chart $\frM_\tX$. There exist a ``mirror map'' $\ttau \colon H^*_T(W) \to H^*(\tX)$ and an isomorphism  $\QDM_T(W)_\tX\sphat \cong \ttau^*\QDM(\tX)^{\rm ext}$ induced by Fourier transformation. 
\end{theorem} 

The validity of Conjecture \ref{conj:reduction} for $X = W/\!/T$ shows that we have a projection   $\QDM_T(W)_\tX\sphat \to \tau^* \QDM(X)^{\rm La}$ induced by Fourier transformation (Proposition \ref{prop:FTX_completion}). 
On the other hand, the $T$-fixed component $Z$ of $W$ produces the Mellin-Barnes-type solutions for the Fourier transform of $\QDM_T(W)$:  
\[
\int e^{-\lambda \log \frq/z} G_Z(\lambda) J_W(\theta)|_Z d\lambda \qquad \text{taking values in $H^*(Z)$}
\]
where $G_Z(\lambda)$ is the product \eqref{eq:G_F} of $\Gamma$-functions arising from the inverse $\hGamma$-class of the normal bundle $\cN_{Z/W}$; $G_Z(\lambda)$ gives the quantum Riemann-Roch operator of Coates-Givental \cite{Coates-Givental} via asymptotic expansion. 
By the stationary phase approximation of this Fourier integral, we obtain $(r-1)$ many projections  $\QDM_T(W)_\tX\sphat \to \varsigma_j^*\QDM(Z)^{\rm La}$ (Corollary \ref{cor:FT_Zj}). Similar techniques were used in \cite{Iritani-Koto:projective_bundle}. The decomposition theorem follows by combining these projections. 

\begin{remark} We can view the $(r-1)$ copies of $\QDM(Z)$ as wall-crossing terms for quantum $D$-modules. The method in this paper could be applied to more general GIT variations or symplectic birational cobordisms \cite{Guillemin-Sternberg:birational, Hu-Li-Ruan:birational_cobordism}. 
It would be also interesting to understand the relationship between our results/conjectures and the works  \cite{Woodward:qKirwan, Gonzalez-Woodward:wall-crossing,GSW:gauged_maps} on gauged Gromov-Witten theory and quantum Kirwan maps. 
\end{remark} 

\subsection{Plan of the paper} 
In \S\ref{sec:preliminaries}, we recall the definition of quantum $D$-modules, Givental cone, twisted Gromov-Witten invariants and shift operators. In particular, we introduce the \emph{extended} shift action of $N_1^T(X)$ that combines the shift operators and the Novikov variables. 
In \S\ref{sec:geom_blowup}, we collect basic geometric facts about the master space $W$ of blowups. 
The content of \S\ref{sec:Fourier} is the technical core of the paper: we show that discrete/continuous Fourier  transforms of the equivariant $J$-function lie in the Givental cone of the GIT quotients or $T$-fixed loci. In \S\ref{sec:decomposition}, we introduce a completion of the equivariant quantum $D$-module of $W$ and prove the decomposition theorem.

\subsection*{Acknowledgements} 
I thank Sergey Galkin, Yuki Koto, Fumihiko Sanda, Yuuji Tanaka and Constantin Teleman for helpful discussions at various stages of this work. 
I thank Ludmil Katzarkov, Maxim Kontsevich, Tony Pantev and Tony Yue Yu for their interest in this work and for helpful discussions on the unique reconstructibility of the decomposition. 
I thank Weiqiang He, Xiaowen Hu, Hua-Zhong Ke, Changzheng Li, Chris Woodward and Longting Wu for valuable comments on the paper. 
I warmly thank anonymous referees for their many helpful comments, which significantly improves the presentation of the paper. 
This research is supported by JSPS grant 16H06335, 20K03582, 21H04994 and 23H01073. 



\begin{longtable}{ll} 
\caption{List of Notation} \\
\toprule 
$\iu$ & imaginary unit. \\ 
$\N$ & the set $\Z_{\ge 0}$ of non-negative integers. \\ 
$T$ & algebraic torus $(\C^\times)^{\rank T}$; $\rank T =1$ from \S\ref{sec:geom_blowup} onwards. \\ 
\midrule 
$\NEN(X)$ & the monoid of effective curves, see \S\ref{subsec:formal_power_series}. \\ 
$\C[\![Q]\!]$ & the Novikov ring $\C[\![\NEN(X)]\!]$, see \S\ref{subsec:formal_power_series}. \\ 
$M_X(\tau)$ & fundamental solution of the quantum connection \eqref{eq:fundsol}. \\ 
$J_X(\tau)$ & the $J$-function $M_X(\tau) 1$ \eqref{eq:J-function}. \\
$\cJ_X(\bt)$ & the big $J$-function \eqref{eq:point_on_the_cone}. \\ 
$N_1^T(X)$ & the group of equivariant curve classes in $H_2^T(X,\Z)$, see \S\ref{subsec:equiv_curve_classes}. \\ 
$\NEN^T(X)$ & the monoid of equivariant curve classes in $N_1^T(X)$, see Proposition \ref{prop:module_over_extended_shift}. \\
$\hbS^\beta$ & the shift operator for $\beta\in N_1^T(X)$, see \S \ref{subsec:shift}. \\
$\hcS^\beta$ & the shift operator on the Givental space for $\beta \in N_1^T(X)$, see \S \ref{subsec:shift}. \\ 
$\sigma_F(k)$ & class in $N_1^T(X)$ associated with a fixed component $F\subset X$ and \\ 
& $k\in \Hom(\C^\times,T)$, see \S \ref{subsec:shift}. \\ 
$\sigma_{\rm max}(k)$ & the maximal section class in $N_1^T(X)$ for $k\in \Hom(\C^\times,T)$, see \S \ref{subsec:shift}. \\ 
$\cH_X$ & the Givental space \eqref{eq:Givental_space} of $X$. \\
$\cH_X^{\rm rat}$ & the rational Givental space \eqref{eq:rational_Givental_space} of $X$. \\ 
$\cL_X$ & the Givental cone of $X$, see \S\ref{subsec:Givental_cone}. \\ 
$\Delta_{(V,\bc)}$ & the quantum Riemann-Roch operator \eqref{eq:QRR_operator}. \\
\midrule 
$\tX$ & blowup of $X$ along $Z$. \\ 
$W$ & the master space $\Bl_{Z\times\{0\}}(X\times \PP^1)$. \\
$Q,x,y,S$ & variables for $\C[N_1^T(W)] \cong \C[N_1(X)\oplus \Z^3]$ with $\deg x = 4$, $\deg y =2r$, \\ 
& $\deg S = 2$, see \S \ref{subsec:N1_W}.  \\ 
$\QW$ & the Novikov variable of $W$, denoting $(Q,x,y)$ collectively;  \\ 
& for $\delta \in N_1(W)$, $\QW^{\delta} = Q^{\pr_{1*}\hvarphi_*\delta} x^{[X]\cdot \delta} y^{-[\hD]\cdot \delta}$, see \S \ref{subsec:N1_W}. \\ 
$\hS$ & the variable for $\C[N_1^T(W)]$, denoting $(\QW,S)$ collectively, see \S\ref{subsec:QDM(W)}. \\ 
$S$ & the Seidel variable $\hS^{(0,0,0,1)}$ of $W$, see \S\ref{subsec:N1_W}. \\ 
$\cS$ & the shift operator $\hcS^{(0,0,0,1)}$ on $\cH^W_{\rm rat}$, see \S \ref{subsec:QDM(W)}. \\ 
$\tQ$ & the Novikov variable of $\tX$; $\tQ^\td$ is identified with $Q^{\varphi_*\td} (y^{-1} S)^{[D]\cdot \td}$ in \\  & $\C[C_{\tX,\N}^\vee]$ via the dual Kirwan map $\kappa_\tX^*$, see \S\ref{subsec:Kirwan}.\\ 
$S_F$ & the variable $\hS^{\sigma_F(1)}$ associated with a fixed component $F\subset W$;  \\ 
& $S_X = S x^{-1}$, $S_Z = S y^{-1}$, $S_\tX = S$, see \S \ref{subsubsec:MB_integrals}. \\ 
$\frq$ & the variable $yS^{-1}=S_Z^{-1}$ of degree $2(r-1)$, see \S\ref{subsec:Fourier_projections}. \\ 
$C_T(W)$ & the $T$-ample cone in $N^1_T(W)_\R$, see \S\ref{subsec:T-ample_cone}. \\ 
$C_Y$ & GIT chamber in $C_T(W)$ corresponding to $Y=X$ or $\tX$, see \S\ref{subsec:T-ample_cone}. \\   
$C_{X,\N}^\vee$ & the monoid \eqref{eq:dualmonoids} dual to $C_X$ such that $\C[C_{X,\N}^\vee] = \C[\QW, xS^{-1}, y^{-1}S]$.  \\  
$C_{\tX,\N}^\vee$ & the monoid \eqref{eq:dualmonoids} dual to $C_\tX$ such that $\C[C_{\tX,\N}^\vee] = \C[\QW, yS^{-1}, S]$. \\ 
$\NEN^T(W)$ & the monoid $C_{\tX,\N}^\vee \cap C_{X,\N}^\vee$ \eqref{eq:dualmonoids} such that $\C[\NEN^T(W)] = \C[\QW, S, x S^{-1}]$. \\ 
$\kappa_Y$ & the Kirwan map from $H^*_T(W)$ to $H^*(Y)$ for $Y=X$ or $\tX$, see \S\ref{subsec:Kirwan}. \\ 
$\sfF_Y$ & discrete Fourier transformation $\cH^{\rm rat}_W[\QW^{-1}] \dasharrow \cH^{\rm ext}_Y[\QW^{-1}]$ associated \\
& with a GIT quotient $Y = X$ or $\tX$, see \S\ref{subsec:discrete_Fourier}. \\ 
$\scrF_{F,j}$ & continuous Fourier transformation $\cH^{\rm rat}_W \to H^*(F)$ associated with a \\ 
& fixed component $F\subset W$, see \S\ref{subsubsec:formal_asymptotics}. \\ 
$\theta$ & parameter taking values in $H^*_T(W)$; we write $\btheta =\{\theta^{i,k}\}$ for the \\ 
& coefficients of the expansion $\theta = \sum_{i,k} \theta^{i,k} \phi_i \lambda^k$, see \S\ref{subsec:QDM(W)}.\\ 
$\tau$, $\ttau$, $\sigma$ & parameters taking values in $H^*(X)$, $H^*(\tX)$, $H^*(Z)$ respectively. \\ 
$\QDM_T(W)$ & the equivariant quantum $D$-module $H^*_T(W)[z][\![\QW,\btheta]\!]$ of $W$, see \S\ref{subsec:QDM(W)}. \\ 
$\QDM_T(W)_\tX\sphat$ & a completion of $\QDM_T(W)_\tX = \C[yS^{-1}] \cdot \QDM_T(W)$, see \S\ref{subsec:completion}. \\ 
$\QDM(X)$ & the quantum $D$-module $H^*(X)[z][\![Q,\tau]\!]$ of $X$, see \S\ref{subsec:extended_QDM}. \\
$\QDM(\tX)$ & the quantum $D$-module $H^*(\tX)[z][\![\tQ,\ttau]\!]$ of $\tX$, see \S\ref{subsec:extended_QDM}. \\ 
$\QDM(Z)$ & the quantum $D$-module $H^*(Z)[z][\![Q_Z,\sigma]\!]$ of $Z$, see \S\ref{subsubsec:projection_Z}. \\ 
$\QDM(Y)^{\rm ext}$ & the base change of $\QDM(Y)$ via $\C[z][\![Q]\!] \hookrightarrow \C[z][\![C^\vee_{Y,\N}]\!]$ \\ 
& for $Y=X$ or $\tX$, see \S\ref{subsec:extended_QDM}. \\ 
$\QDM(F)^{\rm La}$ & the base change of $\QDM(F)$ to $\C[z](\!(\frq^{-1/\frs})\!)[\![\QW]\!]$ for $F=X$, $\tX$, $Z$; \\ 
& see \S\ref{subsubsec:projection_X}, \S\ref{subsec:restriction_slice},  \S\ref{subsubsec:projection_Z} respectively. \\ 
$\QDM(F)^{\rm la}$ & the base change of $\QDM(F)$ to $\C[z](\!(\frq^{-1/\frs})\!)[\![Q]\!]$ for $F = X$, $\tX$, $Z$; \\ 
& see \S\ref{subsec:restriction_slice}. \\
$\FT_Y$ & Fourier transformation from $\QDM_T(W)$ to $\ttt^*\QDM(Y)^{\rm ext}$ for \\ 
& $Y=X$ or $\tX$, see \S\ref{subsec:FT_QDM}. \\ 
$\FT_Y\sphat$ & Fourier transformation extended to the completion $\QDM_T(W)_\tX\sphat$ \\
& for $Y=X$ or $\tX$, see \S\ref{subsubsec:projection_X}, \S\ref{subsec:Fourier_isom} respectively. \\
$\FT_{Z,j}\sphat$ & Fourier transformation from $\QDM_T(W)_\tX\sphat$ to $\sigma_j^*\QDM(Z)^{\rm La}$, \\ 
& $j=0,\dots,r-2$, see \S\ref{subsubsec:projection_Z}. \\
\bottomrule 
\end{longtable} 

\section{Preliminaries} 
\label{sec:preliminaries} 
In this section, we introduce notation on Gromov-Witten invariants, quantum cohomology and Givental cones. We also discuss equivariant shift operators and quantum Riemann-Roch theorem. Throughout this section, $X$ denotes a smooth projective variety over $\C$. We denote by $\N=\Z_{\ge 0}$ the set of non-negative integers. 

\subsection{Gromov-Witten invariants} We refer the reader to \cite{Cox-Katz,Manin:Frobenius_book} and references therein for background material on Gromov-Witten invariants. For a smooth projective variety $X$, let $X_{0,n,d}$ denote the moduli stack of genus-zero, $n$-pointed stable maps to $X$ of degree $d \in H_2(X,\Z)$. It is a proper Deligne-Mumford stack and carries the virtual fundamental class $[X_{0,n,d}]_{\rm vir} \in H_{2D}(X_{0,n,d},\Q)$ with $D= c_1(X)\cdot d+\dim_\C X +  n-3$. For $\alpha_1,\dots,\alpha_n \in H^*(X,\Q)$ and $k_1,\dots,k_n\in \N=\Z_{\ge 0}$, we have the descendant \emph{Gromov-Witten invariant}: 
\[
\corr{\alpha_1\psi^{k_1}, \dots, \alpha_n\psi^{k_n}}_{0,n,d}^X = \int_{[X_{0,n,d}]_{\rm vir}} 
\prod_{i=1}^n \ev_i^*(\alpha_i) \psi_i^{k_i} \in \Q 
\]
where $\ev_i\colon X_{0,n,d} \to X$ is the evaluation map at the $i$th marked point and $\psi_i$ is the first Chern class of the universal cotangent line bundle (over $X_{0,n,d}$) at the $i$th marked point. 

When an algebraic torus $T\cong (\C^\times)^{\rank T}$ acts on $X$, it also acts on the moduli space $X_{0,n,d}$. We then have the \emph{equivariant Gromov-Witten invariant} for $\alpha_1,\dots,\alpha_n \in H^*_T(X,\Q)$ and $k_1,\dots,k_n\in \N$:
\[
\corr{\alpha_1\psi^{k_1}, \dots, \alpha_n\psi^{k_n}}_{0,n,d}^{X,T} \in H^*_T(\pt,\Q) \cong \Q[\lambda] 
\]
by replacing the integral with the equivariant one, where $\lambda=(\lambda_1,\dots,\lambda_{\rank T})$ is a basis of $H^2_T(\pt,\Z)$ called the equivariant parameters. 

\begin{remark} Throughout the paper, we identify an equivariant parameter $\lambda_i \in H^2_T(\pt,\Z)$ with a character of $T$. The identification $\Hom(T,\C^\times) \cong H^2_T(\pt,\Z)$ is given as follows: a character $\chi\colon T \to \C^\times$ naturally defines a $T$-equivariant line bundle $L_\chi \to \pt$ over a point, and gives rise to the class $c_1^T(L_\chi) \in H^2_T(\pt,\Z)$. 
\end{remark} 

\subsection{Formal power series rings} 
\label{subsec:formal_power_series} 
We introduce our conventions on power series ring. In this paper, we mostly work with $\Z$-graded rings or modules and consider their completions in the category of graded rings or modules. If $M = \bigoplus_{n\in \Z} M_n$ is a graded module whose topology is given by a descending chain of graded submodules $N_k = \bigoplus_{n\in \Z} N_{k,n}\subset M$, then the \emph{graded completion} of $M$ is defined to be $\widehat{M} = \bigoplus_{n\in \Z} \widehat{M}_n$ with $\widehat{M}_n = \varprojlim_k M_n/N_{k,n}$. 

Let $\NEN(X)\subset H_2(X,\Z)$ denote the monoid generated by classes of effective curves and let $\NE(X)\subset H_2(X,\R)$ be the convex cone generated by $\NEN(X)$. We write $Q^d\in \C[\NEN(X)]$ for the element corresponding to $d\in \NEN(X)$ and define its degree by 
\[
\deg Q^d = 2 c_1(X) \cdot d.
\]
Let $\omega$ be an ample class. We write $\C[\![Q]\!]=\C[\![\NEN(X)]\!]$ for the graded completion of $\C[\NEN(X)]$ with respect to the descending chain of submodules $I_k=\langle Q^d : \omega \cdot d \ge k\rangle_\C$. The completion does not depend on the choice of $\omega$. 
We call $Q$ the \emph{Novikov variable} and $\C[\![Q]\!]$ the \emph{Novikov ring}. 

This construction generalizes to a submonoid $C_\N$ of a finitely generated abelian group $N$ such that the closure of the cone $C\subset N\otimes \R$ generated by $C_\N$ is strictly convex. 
Suppose that we have a linear function $\deg \colon N \to \Z$ that defines a grading on $\C[C_\N]$. Using an interior point $\omega$ of the dual cone $C^\vee$, we can similarly define the graded completion $\C[\![C_\N]\!]$ of $\C[C_\N]$, which is independent of the choice of $\omega$. More generally, we can define $K[\![C_\N]\!]$ for any $\Z$-graded module $K$. 

For a (possibly infinite) set of parameters $\bs=(s_1,s_2,\dots)$ and a module $K$, we write $K[\![\bs]\!]$ for the completion of the space $K[\bs]= K[s_1,s_2,\dots]$ of polynomials with coefficients in $K$ with respect to the descending chain of submodules $(s_1^n,\dots,s_n^n,s_{n+1}, s_{n+2},\dots) K[\bs]$. In the graded case where $s_i$ has a degree and $K$ is $\Z$-graded, the completion $K[\![\bs]\!]$ should be understood in the graded sense. 

Let $x$ be a variable with degree $\deg x \in \Z$. For a $\Z$-graded module $K$, we denote by $K(\!(x)\!)$ the graded completion of $K[x,x^{-1}]$ with respect to the descending chain of submodules $x^n K[x]$. A homogeneous element of $K(\!(x)\!)$ is of the form $\sum_{n=m}^\infty a_n x^n$ with $\deg a_n + n \deg x$ being independent of $n$, for some $m\in \Z$. 

Most of $\Z$-graded rings or modules in this paper are also equipped with $\Z/2\Z$-grading, called \emph{parity}, so that they are $\Z\times (\Z/2\Z)$-graded. The module or ring structures are assumed to be supercommutative with respect to the parity, i.e.~$a\cdot b = (-1)^{|a| |b|} b \cdot a$ where $|a|,|b|\in \Z/2\Z$ are the parities of $a,b$ respectively. In many cases, the parity $|a|$ is congruent mod $2$ to the degree $\deg a \in \Z$. However it is sometimes convenient to allow the parity to be independent of the degree. For example, the square root $\sqrt{x}$ of a variable $x$ with $\deg x=2$, $|x|=0$ has degree 1 and (still) even parity.


\subsection{Quantum cohomology and quantum connection} 
\label{subsec:qcoh_qconn}
Unless otherwise stated, $H^*(X)$ stands for the cohomology group of $X$ with complex coefficients. Let 
\[
(\alpha_1,\alpha_2)_X = \int_X \alpha_1 \cup \alpha_2 
\]
be the Poincar\'e pairing on $H^*(X)$. We choose a homogeneous basis $\{\phi_i\}_{i=0}^s$ of $H^*(X)$ with $\phi_0=1$ and denote by $\{\tau^i\}_{i=0}^s$ the dual coordinate system on $H^*(X)$. We write $\tau = \sum_i \tau^i \phi_i$ for a general point on $H^*(X)$. The degree of $\tau^i$ is set to be $\deg \tau^i := 2 - \deg \phi_i$. The quantum cohomology is defined over the ring $\C[\![Q,\tau]\!]=\C[\![Q]\!][\![\tau^0,\dots,\tau^s]\!]$. 
We note that odd variables anti-commute with each other: $\tau^i\tau^j = (-1)^{|i||j|} \tau^j \tau^i$, $\tau^i \phi_j = (-1)^{|i||j|} \phi_j \tau^i$ with $|i|:=\deg \phi_i \pmod 2$; the quantum cohomology is a family of ring structures over the cohomology group $H^*(X,\C[\![Q]\!])$ regarded as a (formal) supermanifold \cite{Manin:Frobenius_book}. 

The \emph{quantum product} $\star_\tau$ is a $\C[\![Q,\tau]\!]$-bilinear, associative and supercommutative product structure on $H^*(X)[\![Q,\tau]\!]$ defined by 
\begin{equation} 
\label{eq:qprod} 
(\phi_i\star_\tau \phi_j,\phi_k)_X = \sum_{d\in \NEN(X), n\ge 0} 
\corr{\phi_i,\phi_j,\phi_k,\tau,\dots,\tau}_{0,n+3,d}^X \frac{Q^d}{n!}. 
\end{equation} 
We follow the Koszul sign convention when expanding the correlators $\corr{\phi_i,\phi_j,\phi_k,\tau,\dots,\tau}_{0,n+3,d}^X$ in polynomials of $\{\tau^i\}$. The algebra $(H^*(X)[\![Q,\tau]\!],\star_\tau)$ is called \emph{quantum cohomology}. The quantum product is homogeneous and has $\phi_0=1$ as the identity element.  

Let $z$ be a variable of degree $2$. The \emph{quantum connection} is given by the following set of operators $\nabla_{\tau^i}, \nabla_{z\partial_z}, \nabla_{\xi Q \partial_Q}\colon H^*(X)[z][\![Q,\tau]\!] \to z^{-1} H^*(X)[z][\![Q,\tau]\!]$ (with $\xi \in H^2(X)$) 
\begin{align}
\label{eq:qconn}
\begin{split} 
\nabla_{\tau^i} & = \partial_{\tau^i} + z^{-1} (\phi_i\star_\tau) \\ 
\nabla_{z\partial_z} & = z \partial_z - z^{-1} (E_X\star_\tau) + \mu_X \\ 
\nabla_{\xi Q\partial_Q} & = \xi Q \partial_Q + z^{-1} (\xi \star_\tau)   
\end{split} 
\end{align} 
where the \emph{Euler vector field} $E_X\in H^*(X)[\![Q,\tau]\!]$ and the \emph{grading operator} $\mu_X\in \End(H^*(X))$ are given by 
\begin{equation} 
\label{eq:Euler_grading} 
E_X = c_1(X) + \sum_{i} \left(1-\frac{\deg \phi_i}{2}\right) \tau^i \phi_i, \quad 
\mu_X(\phi_i) = \left( \frac{\deg \phi_i}{2} - \frac{\dim_\C X}{2} \right) \phi_i  
\end{equation} 
and $\xi Q \partial_Q$ is the derivation of the Novikov ring $\C[\![Q]\!]$ given by $(\xi Q \partial_Q) Q^d = (\xi \cdot d) Q^d$. The operators $\nabla_{\tau^i}, \nabla_{z\partial_z}, \nabla_{\xi Q\partial_Q}$ are homogeneous of degree $\deg \phi_i-2$, $0$, $0$ respectively and are supercommutative, e.g.~$\nabla_{\tau^i} \nabla_{\tau^j}=(-1)^{|i||j|} \nabla_{\tau^j} \nabla_{\tau^i}$. The quantum connection can be viewed as a flat connection on the trivial $H^*(X)$-bundle over the $(\tau,z,Q)$-space. 
Let $P_X$ be the $z$-sesquilinear pairing on $H^*(X)[z][\![Q,\tau]\!]$ induced by the Poincar\'e pairing: 
\begin{equation} 
\label{eq:PX}
P_X(f,g) = (f(-z),g(z))_X. 
\end{equation} 
This is flat with respect to the quantum connection $dP_X(f,g) = P_X(\nabla f, g) + P_X(f,\nabla g)$. The module 
\begin{equation*} 
\QDM(X)=H^*(X)[z][\![Q,\tau]\!]
\end{equation*} 
equipped with the flat connection $\nabla$ and the pairing $P_X$ is called the \emph{quantum $D$-module} of $X$. 

\begin{remark} 
The variable $z$ of the quantum connection can be interpreted geometrically as the equivariant parameter associated with the $\C^\times$ action on the domain curve $\PP^1$, see \S \ref{subsec:shift}. 
\end{remark} 

The quantum connection admits a fundamental solution $M_X(\tau)\in \End(H^*(X))[z^{-1}][\![Q,\tau]\!]$ given by descendant Gromov-Witten invariants:  
\begin{equation} 
\label{eq:fundsol} 
(M_X(\tau) \phi_i,\phi_j)_X = (\phi_i,\phi_j)_X + \sum_{\substack{d\in \NEN(X),n\ge 0 \\ (n,d)\neq (0,0)}} 
\corr{\phi_i,\tau,\dots,\tau,\frac{\phi_j}{z-\psi}}_{0,n+2,d}^X \frac{Q^d}{n!}.  
\end{equation} 
where $\phi_j/(z-\psi)$ should be expanded in the power series $\sum_{b=0}^\infty \phi_j \psi^b z^{-b-1}$. 
It is homogeneous of degree zero and is a solution of the quantum connection in the following sense (see \cite[\S 1]{Givental:equivariant}, \cite[Proposition 2]{Pandharipande:afterGivental}, \cite[Proposition 3.1]{CCIT:MS}):    
\begin{align} 
\label{eq:fundsol_qconn}
\begin{split}
M_X(\tau) \circ \nabla_{\tau^i} & = \partial_{\tau^i} \circ M_X(\tau) \\ 
M_X(\tau) \circ \nabla_{z\partial_z} & = (z\partial_z - z^{-1} c_1(X) + \mu_X) \circ M_X(\tau) \\
M_X(\tau) \circ \nabla_{\xi Q\partial_Q} & = (\xi Q \partial_Q + z^{-1} \xi)\circ M_X(\tau) 
\end{split} 
\end{align} 
where $c_1(X)$ and $\xi$ on the right-hand side are regarded as endomorphisms of $H^*(X)$ by the cup product. The fundamental solution preserves the Poincar\'e pairing: 
\begin{equation} 
\label{eq:fundsol_pairing} 
P_X(M_X(\tau) f, M_X(\tau) g) = P_X(f,g)
\end{equation} 
for $f,g\in H^*(X)[z][\![Q,\tau]\!]$. The \emph{$J$-function} of $X$ is defined to be $M_X(\tau)1$; using the String Equation, we have 
\begin{align} 
\label{eq:J-function} 
J_X(\tau) = 1 + \frac{\tau}{z} + \sum_{j=0}^s 
\sum_{\substack{d\in \NEN(X),n\ge 0 \\ (n,d)\neq (0,0),(1,0)}} \phi^j 
\corr{\tau,\dots,\tau,\frac{\phi_j}{z(z-\psi)}}_{0,n+2,d}^X \frac{Q^d}{n!} 
\end{align} 
where $\{\phi^i\}\subset H^*(X)$ is a basis dual to $\{\phi_i\}$ such that $(\phi_i,\phi^j)_X=\delta_i^j$. 

\begin{remark} 
\label{rem:QDM} 
We defined the quantum $D$-module over the ring $\C[z][\![Q,\tau]\!]$. Using the Divisor and String Equations, we can reduce the structure of the quantum $D$-module to a smaller ring. We write $\tau^{(2)}=\sum_{\deg \phi_i=2} \tau^i \phi_i$ for the $H^2$-part of $\tau$ and decompose $\tau = \tau^0\phi_0 + \tau^{(2)} + \tau'$, where $\tau^0$ is the coordinate dual to $\phi_0 =1$. Then the quantum connection \eqref{eq:qconn} multiplied by $z$ preserves the submodule $H^*(X) [z][\![Qe^{\tau^{(2)}},\tau']\!][\tau^0] \subset \QDM(X)$, where we interpret $(Q e^{\tau^{(2)}})^d = Q^d e^{\tau^{(2)}\cdot d}$ for $d\in \NEN(X)$. Note that the quantum product $\star_\tau$ does not depend on $\tau^0$, but $\tau^0$ appears in the Euler vector field $E_X$ (while $\tau^{(2)}$ does not appear in $E_X$). This fact is relevant when considering pullbacks of quantum $D$-modules. 
\end{remark}

\subsection{Equivariant quantum cohomology and quantum connection} 
\label{subsec:equiv_qconn} 
Let $X$ be a smooth projective variety equipped with an algebraic $T$-action. By the equivariant formality \cite{GKM} of $X$, $H^*_T(X)$ is a free module of rank $\dim H^*(X)$ over $H^*_T(\pt)=\C[\lambda_1,\dots,\lambda_{\rank T}]=\C[\lambda]$. Let $\{\phi_i\}_{i=0}^s$ be a homogeneous basis of $H^*_T(X)$ over $\C[\lambda]$ with $\phi_0=1$. A general point $\tau \in H_T^*(X)$ is expanded as $\tau = \sum_{i=0}^s \tau^i \phi_i$ as before, where $\tau^i$'s are now $\C[\lambda]$-valued coordinates. 

We define the \emph{equivariant quantum product} $\star_\tau$ on $H^*_T(X)[\![Q,\tau]\!]$ at $\tau \in H^*_T(X)$ similarly to \eqref{eq:qprod} by replacing the Poincar\'e pairing and the Gromov-Witten invariants with their equivariant counterparts. Note that the equivariant Poincar\'e pairing is perfect over $\C[\lambda]$. 

The definition of the equivariant quantum connection is similar but requires care. We shall define it as a connection on the infinite dimensional space\footnote{We consider the infinite-dimensional base in order to define the connection $\nabla_{z\partial_z}$ in the $z$-direction (which is not $\C[\lambda]$-linear) and also to introduce shift operators (see \S\ref{subsec:shift}).} $H^*_T(X)$ over $\C$. 
We introduce $\C$-valued coordinates $\{\tau^{i,k}\}_{0\le i\le s, k}$ on $H^*_T(X)$ dual to the $\C$-basis $\{\phi_i \lambda^k\}_{0\le i \le s,k}$ so that $\tau^i= \sum_{k} \tau^{i,k} \lambda^k$. The index $k$ here ranges over  $\N^{\rank T}$.  We set $\deg \tau^{i,k} := 2-\deg \phi_i-2|k|$ with $|k| = \sum_{a=1}^{\rank T} k_a$. We use the boldface letter $\btau$ to denote the infinite set $\{\tau^{i,k}\}$ of variables. The \emph{equivariant quantum connection} $\nabla_{\tau^{i,k}}, \nabla_{z\partial_z}, \nabla_{\xi Q\partial_Q} \colon H^*_T(X)[z][\![\btau,Q]\!] \to z^{-1} H^*_T(X)[z][\![\btau,Q]\!]$ is defined by essentially the same formulae as \eqref{eq:qconn}: 
\begin{align*}
\begin{split} 
\nabla_{\tau^{i,k}} & = \partial_{\tau^{i,k}} + z^{-1} (\phi_i\lambda^k \star_\tau), \\ 
\nabla_{z\partial_z} & = z \partial_z - z^{-1} (E_X\star_\tau) + \mu_X, \\ 
\nabla_{\xi Q\partial_Q} & = \xi Q \partial_Q + z^{-1} (\xi \star_\tau).     
\end{split} 
\end{align*} 
Here $\xi \in H^2_T(X)$ is an \emph{equivariant} class; the \emph{Euler vector field} $E_X\in H^*_T(X)[\![Q,\btau]\!]$ and the \emph{grading operator} $\mu_X\in \End_\C(H^*_T(X))$ in the equivariant setting are given by 
\begin{align*} 
E_X & = c_1^T(X) + \sum_{i,k} \left(1-\frac{\deg \phi_i}{2}-|k|\right) \tau^{i,k}  \phi_i\lambda^k, \\ 
\mu_X(\phi_i\lambda^k) & = \left( \frac{\deg \phi_i}{2}+|k| - \frac{\dim_\C X}{2} \right) \phi_i  \lambda^k. 
\end{align*} 
Note that $\nabla_{\xi Q\partial_Q}$ equals the multiplication by $z^{-1}\xi$ if $\xi\in H^2_T(\pt)$. Note also that the grading operator $\mu_X$ is not $\C[\lambda]$-linear --- it contains a derivation in $\lambda$. 
The \emph{equivariant quantum $D$-module} $\QDM_T(X)$ is the module $H^*_T(X)[z][\![Q,\btau]\!]$ equipped with the quantum connection $\nabla$ and the pairing $P_X$ induced by the equivariant Poincar\'e pairing.

We define the fundamental solution $M_X(\tau)$ for the equivariant theory by the same formula as \eqref{eq:fundsol} by using the equivariant Poincar\'e pairing and the equivariant Gromov-Witten invariants in place of the non-equivariant ones. Then $M_X(\tau)$ with $\tau^i= \sum_{k} \tau^{i,k} \lambda^k$ is a solution of the equivariant quantum connection and satisfies \eqref{eq:fundsol_qconn}, \eqref{eq:fundsol_pairing} with $\tau^i$ replaced with $\tau^{i,k}$, $c_1(X)$ with $c_1^T(X)$, and $\mu_X$ and the Poincar\'e pairing with their equivariant counterparts. By using the expansion $1/(z-\psi)=\sum_{b\ge 0} z^{-b-1} \psi^b$, we find that the fundamental solution in the equivariant theory lies in the following ring: 
\begin{equation} 
\label{eq:fundsol_at_infinity}
M_X(\tau) \in \End_{\C[\lambda]}(H^*_T(X))[\![z^{-1}]\!][\![Q,\tau]\!]. 
\end{equation} 
On the other hand, using the virtual localization \cite{Kontsevich:enumeration, Graber-Pandharipande}, we find that $M_X(\tau)$ is a formal power series in $Q$ and $\tau$ with coefficients in \emph{rational functions} of $\lambda$ and $z$: 
\begin{equation} 
\label{eq:fundsol_rational} 
M_X(\tau) \in \End_{\C[\lambda]}(H^*_T(X))\otimes_{\C[\lambda]} \C(\lambda,z)_{\rm hom} [\![Q,\tau]\!]
\end{equation} 
where $\C(\lambda,z)_{\rm hom} := \C(\lambda_1/z,\dots,\lambda_{\rank T}/z)[z,z^{-1}]$ is the localization of $\C[\lambda,z]$ with respect to non-zero homogeneous elements. 

The \emph{equivariant $J$-function} is defined similarly to the non-equivariant case as $J_X(\tau) = M_X(\tau)1$, where $M_X(\tau)$ is the equivariant fundamental solution. It is given by the same formula as \eqref{eq:J-function}, with the correlators replaced by their equivariant counterparts and the bases $\{\phi_i\}$, $\{\phi^i\}$ understood as mutually dual bases of $H_T^*(X)$ over $\C[\lambda]$ with respect to the equivariant Poincar\'e pairing.

\subsection{Equivariant curve classes} 
\label{subsec:equiv_curve_classes} 
Let $N_1(X)\subset H_2(X,\Z)$ denote the group generated by homology classes of algebraic curves. For a $T$-variety $X$, we introduce the equivariant version $N_1^T(X) \subset H_2^T(X,\Z)$. Note that this is unrelated to the Edidin-Graham equivariant Chow group $A_1^T(X)$ \cite{Edidin-Graham:equiv_intersection}, which admits a cycle map $A_1^T(X) \to H^T_{{\rm BM}, 2}(X)$ to the Borel-Moore equivariant homology as opposed to the ordinary one. 

The equivariant homology group $H_2^T(X,\Z)$ is defined to be the second integral homology of the Borel construction $X_T = (X\times ET)/T$, where $ET\to BT$ is a universal $T$-bundle and $T$ acts on $X\times ET$ diagonally. We can take $ET = (\C^{\oplus \infty}\setminus \{0\})^{\rank T}$ and $BT=(\PP^\infty)^{\rank T}$. We define $N_1^T(X)$ to be the group generated by images in $H_2^T(X,\Z)$ of algebraic curves contained in a fibre $X$ of $X_T\to BT$ and classes $s_{x*}(\sigma)$, $\sigma \in H_2(BT,\Z)$ associated with sections $s_x \colon BT \cong \{x\} \times BT \subset X_T$ given by any $T$-fixed points $x\in X$. 
We have the following commutative diagram: 
\begin{equation} 
\label{eq:equiv_homology_seq} 
\begin{CD}
0 @>>> N_1(X) @>>> N_1^T(X) @>>> H_2^T(\pt,\Z) @>>>0 \\
@. @VVV @VVV  @| \\ 
0@>>> H_2(X,\Z) @>>> H_2^T(X,\Z) @>>> H_2^T(\pt,\Z) @>>> 0   
\end{CD}
\end{equation} 
where the vertical arrows are inclusions. 
\begin{lemma}
\label{lem:equiv_H2_split} 
The two rows in \eqref{eq:equiv_homology_seq} are exact sequences. 
\end{lemma} 
\begin{proof} 
The exactness of the bottom row follows from the Serre spectral sequence for $X_T \to BT$. The $E^2$-page is given by $E^2_{p,q} = H_p(BT,H_q(X,\Z))$ and it suffices to show that the differentials $d^2 \colon E^2_{2,0}\to E^2_{0,1}$ and $d^2\colon E^2_{2,1} \to E^2_{0,2}$ are zero. For this, it is sufficient to show that every element in $E^2_{2,0}$ and $E^2_{2,1}$ is represented by a cycle in $X_T$.  This follows from the fact that a class in $H_i(X,\Z)$, $i=0,1$ can be represented by a cycle in the fixed locus $X^T$. For $i=0$, we can take a fixed point. For $i=1$, we see this by using the Bott-Morse theory for the moment map\footnote{The $S^1$-action associated with an algebraic $\C^\times$-action on a smooth projective variety $X$ admits a moment map for an $S^1$-invariant K\"ahler form. This follows from the fact that an algebraic $\C^\times$-action on $X$ has a fixed point, and that an $S^1$-action on a compact K\"ahler manifold preserving the K\"ahler form is Hamiltonian (i.e.~admits a moment map) if and only if it has a fixed point (Frankel's theorem \cite{Frankel:fixed_points}).} $\mu\colon X\to \R$ of the $S^1$-action associated with a generic subgroup $\C^\times \subset T$, where we choose $\C^\times \subset T$ so that $X^T = X^{\C^\times}$. Since every critical component of $\mu$ has an even index, any class in $H_1(X,\Z)$ can be represented by a cycle in the locus where $\mu$ attains the local minimum. 

Next we see the exactness of the top row. The surjectivity of $N_1^T(X) \to H_2^T(\pt,\Z)$ is obvious from the definition. The injectivity of $N_1(X) \to N_1^T(X)$ follows from the commutative diagram and the exactness of the bottom row. It suffices to show that the kernel of $N_1^T(X) \to H_2^T(\pt,\Z)$ is contained in $N_1(X)$. By a discussion similar to \cite[Lemma 2.2]{Gonzalez-Iritani:Selecta}, the classes $s_{x*}(\sigma)$, $s_{x'*}(\sigma)$ associated with different $T$-fixed points $x,x'$ differ by an element in $N_1(X)$: this follows by joining the $T$-fixed components of $x,x'$ by a chain of $T$-invariant rational curves. The claim follows. 
\end{proof} 

\subsection{Extended shift operators} 
\label{subsec:shift} 
For a smooth projective variety $X$ with a $T$-action, the cocharacter lattice $\Hom(\C^\times,T)$ acts on the equivariant quantum $D$-module as shift operators of equivariant parameters. This is a lift of the Seidel representation \cite{Seidel:pi1} on quantum cohomology. The shift operators were introduced by Okounkov-Pandharipande  \cite{Okounkov-Pandharipande:Hilbert} for the study of quantum cohomology of Hilbert schemes and have been applied and extended in many directions  \cite{Maulik-Okounkov:qcoh_qgroup, BMO:Springer, Iritani:shift,LJones:shift_symp, GMP:nil-Hecke}. In this section, we reformulate it as an action of $N_1^T(X)$ from the previous section \S\ref{subsec:equiv_curve_classes}; this action combines the shift operators with the Novikov variables.

The \emph{Seidel space}\footnote{The Seidel space $E_k$ in this paper is denoted by $E_{-k}$ in \cite{Iritani:shift}.} $E_k=E_k(X)$ associated with a cocharacter $k\in \Hom(\C^\times,T)$ is defined to be the quotient of $X\times (\C^2 \setminus \{0\})$ by the $\C^\times$-action $s \cdot (x,(v_1,v_2)) = (s^k x, (sv_1,s v_2))$. This is a fibre bundle over $\PP^1$ with fibre $X$ via the map $E_k \to \PP^1$, $[x,(v_1,v_2)] \mapsto [v_1, v_2]$. Let $f_k \colon \PP^1 \subset \PP^\infty = B\C^\times \to BT$ be the map induced by the cocharacter $k\colon \C^\times \to T$; it is a continuous map determined up to homotopy. The map $f_k$ gives rise to a  fiber square:
\[
\begin{CD} 
E_k @>{\tilde{f}_k}>> X_{T}\\
@VVV @VVV \\ 
\PP^1 @>_{f_k}>> BT 
\end{CD} 
\]
where $X_T\to BT$ is the Borel construction (see \S\ref{subsec:equiv_curve_classes}). By the map 
$\tilde{f}_k$, we identify section classes of $E_k$ with equivariant homology classes of $X$: 
\begin{equation*} 
\tilde{f}_{k*} \colon H_2^{\rm sec}(E_k,\Z) \hookrightarrow H_2^T(X,\Z).  
\end{equation*} 
where $H_2^{\rm sec}(E_k,\Z)$ denotes the set of classes in $H_2(E_k,\Z)$ mapping to the fundamental class $[\PP^1]$. 
Since $\tilde{f}_k$ can be chosen to be fibrewise algebraic and every class in $N_1^{\rm sec}(E_k):=N_1(E_k) \cap H_2^{\rm sec}(E_k,\Z)$ can be written as the sum of a fibre class in $N_1(X)$ and a section class associated with a $T$-fixed point in $X$, this embedding induces the map: 
\begin{equation} 
\label{eq:N1_Seidelspace} 
N_1^{\rm sec}(E_k) \hookrightarrow N_1^T(X).  
\end{equation} 
The image of this embedding consists precisely of classes in $N_1^T(X)$ mapping to $f_{k*}[\PP^1]=k \in H^T_2(\pt,\Z) \cong \Hom(\C^\times,T)$. We may therefore view $N_1^T(X)$ as the disjoint union of $N_1^{\rm sec}(E_k)$ for all $k\in \Hom(\C^\times,T)$. 

There is a unique fixed component $F_{\rm max}(k)\subset X^{k(\C^\times)}$, called the \emph{maximal component}, such that the normal bundle of $F_{\rm max}(k)$ has only negative $k(\C^\times)$-weights. Let $\sigma_{\rm max}(k) \in \NEN(E_k)$ be the section class associated to a fixed point in $F_{\rm max}(k)$. Then we have the following: 
\begin{lemma}[{\cite[Lemma  2.2]{Gonzalez-Iritani:Selecta}, \cite[Lemma 3.6]{Iritani:shift}}] 
\label{lem:effective_section_classes}
Every effective section class of $E_k\to \PP^1$ is of the form $\sigma_{\rm max}(k) + d$ with an effective class $d\in \NEN(X) \subset N_1(X)$.  
\end{lemma} 

Set $\hT = T\times \C^\times$. We introduce a $\hT$-action on $E_k$ by 
\[
(\lambda, z) \cdot [x,(v_1,v_2)] = [\lambda x, (v_1, z v_2)] 
\]
for $(\lambda,z) \in \hT$. The coordinates $\lambda,z$ of $\hT$ will also be used to denote the corresponding equivariant parameters. Let $X_0$, $X_\infty$ denote the fibre of $E_k\to \PP^1$ at $0=[1,0], \infty=[0,1]\in \PP^1$ respectively. The group $\hT$ acts on $X_0$, $X_\infty$ as 
\begin{align*} 
(\lambda,z) \cdot x & =\lambda \cdot x & & \text{for $x\in X_0$;} \\ 
(\lambda,z) \cdot x & = \lambda z^{-k} \cdot x & & \text{for $x\in X_\infty$}. 
\end{align*} 
The identity map $X_0 \to X_\infty$ is equivariant with respect to the automorphism $\hT \to \hT$, $(\lambda,z) \mapsto (\lambda z^{k},z)$, hence it induces an isomorphism $\Phi_k \colon H^*_{\hT}(X_0) \xrightarrow{\cong} H^*_{\hT}(X_\infty)$ satisfying 
\[
\Phi_k(f(\lambda,z) \alpha) = f(\lambda-kz, z) \Phi_k(\alpha)
\]
for $f(\lambda,z) \in \C[\lambda,z]=H^*_{\hT}(\pt)$ and $\alpha \in H^*_{\hT}(X_0)$. We note that $H^*_{\hT}(X_0)=H^*_T(X)[z]$. 
For $\tau\in H^*_T(X)$, let $\htau\in H^*_{\hT}(E_k)$ denote the unique class satisfying 
\begin{equation*} 
\htau|_{X_0} = \tau \qquad \text{and} \qquad \htau|_{X_\infty} = \Phi_k(\tau). 
\end{equation*}
See \cite[Lemma 3.7, Notation 3.8]{Iritani:shift} for the existence and uniqueness of $\htau$. 

\begin{definition}[shift operators]  
\label{def:shift} 
For $\beta \in N_1^T(X)$, let $\ovbeta\in \Hom(\C^\times,T)$ be the image of $\beta$ under $N_1^T(X) \to H_2^T(\pt,\Z) = \Hom(\C^\times,T)$ and consider the Seidel space $E_k \to \PP^1$ associated with $k:=-\ovbeta$. Let $X_0,X_\infty$ be the fibres of $E_k$ at $0$ and $\infty$ as above. 
For $\tau\in H^*_T(X)$, we define the $H_\hT^*(\pt)$-linear operator $\tbS^\beta(\tau) \colon H^*_{\hT}(X_0)[\![Q,\btau]\!] \to Q^{\beta+\sigma_{\rm max}(k)} H^*_{\hT}(X_\infty)[\![Q,\btau]\!]$ in terms of the $\hT$-equivariant Gromov-Witten invariants of $E_k$ as follows: 
\begin{equation} 
\label{eq:def_tbS}
\left(\tbS^\beta(\tau) \alpha_0,\alpha_\infty\right) =  \sum_{n=0}^\infty \sum_{d\in N_1(X)} \corr{i_{0*}\alpha_0,\htau,\dots,\htau,i_{\infty*}\alpha_\infty}_{0,n+2,-\beta+d}^{E_k,\hT} 
\frac{Q^d}{n!}  
\end{equation} 
where $\alpha_0\in H^*_{\hT}(X_0)$, $\alpha_\infty \in H^*_{\hT}(X_\infty)$, $i_0\colon X_0\to E_k$ and $i_\infty \colon X_\infty \to E_k$ are the inclusions, and $(\cdot,\cdot)$ in the left-hand side is the equivariant Poincar\'e pairing on $H^*_{\hT}(X_\infty)$. 
Finally, we define the (extended) \emph{shift operator} as 
\[
\hbS^\beta(\tau) := \Phi_k^{-1} \circ \tbS^\beta(\tau)  \colon 
H^*_T(X)[z][\![Q,\btau]\!] \to Q^{\beta+\sigma_{\rm max}(k)} H^*_T(X)[z][\![Q,\btau]\!]. 
\]
It satisfies $\hbS^\beta(\tau) (f(\lambda,z) \alpha) = f(\lambda-\ovbeta z, z) \hbS^\beta(\tau) \alpha$ for $f(\lambda,z)\in H^*_\hT(\pt)$. 
\end{definition} 

\begin{remark}
\label{rem:shift}
 (1) The degree $-\beta + d\in N_1^T(X)$ appearing in the right-hand side of \eqref{eq:def_tbS} should be regarded as a section class in $N_1^{\rm sec}(E_k)$ under  \eqref{eq:N1_Seidelspace}. It is effective if and only if $d \in \beta+\sigma_{\rm max}(k) +\NEN(X)$ by Lemma \ref{lem:effective_section_classes}. Therefore the lowest exponent of $Q$ appearing in $\tbS^\beta(\tau), \hbS^\beta(\tau)$ is $\beta+\sigma_{\rm max}(k)$ or higher. 

(2) Since the map $\tau \mapsto \htau$ is not $\C[\lambda$]-linear, we regard $\hbS^\beta(\tau)$ as a power series in the infinite set $\btau=\{\tau^{i,k}\}$ of coordinates for $\tau\in H^*_T(X)$ (see \S\ref{subsec:equiv_qconn}). By dimension counting, we see that $\hbS^\beta(\tau)$ is homogeneous of degree $2 c_1^T(X) \cdot \beta$. 

(3) The operator $\hbS^\beta(\tau)$ is related to the shift operator $\bS_{\ovbeta}(\tau)$ in \cite{Iritani:shift} by the formula 
\[
\hbS^\beta(\tau) = Q^{\beta+\sigma_{\rm max}(-\ovbeta)} \bS_{\ovbeta}(\tau).  
\] 
In particular, $\hbS^\beta(\tau) = Q^\beta$ if $\beta \in N_1(X) \subset N_1^T(X)$. The hat in the notation $\hbS^\beta$ indicates that we extend the shift operator action to include the Novikov variables. The operators $\bS_k (\tau)=\hbS^{-\sigma_{\rm max}(-k)}(\tau)$ in \cite{Iritani:shift} arise from a non-linear splitting $k\mapsto -\sigma_{\rm max}(-k)$ of \eqref{eq:equiv_homology_seq}.  
\end{remark} 

The rational \emph{Givental space} of $X$ is defined to be 
\begin{equation} 
\label{eq:rational_Givental_space} 
\cH^{\rm rat}_X := H^*_\hT(X)_{\rm loc}[\![Q]\!]
\end{equation} 
where $\hT=T\times \C^\times$ acts on $X$ via the projection $\hT \to T$ and $H^*_{\hT}(X)_{\rm loc} = H^*_T(X)\otimes_{\C[\lambda]} \C(\lambda,z)_{\rm hom}$ is the localization of $H^*_\hT(X) = H^*_T(X)[z]$ with respect to non-zero homogeneous elements of $\C[\lambda,z]$. We introduce the shift action of $N_1^T(X)$ on $\cH^{\rm rat}_X$. 

\begin{definition}[shift operator on the Givental space]  
\label{def:shift_Givental} 
For a connected component $F$ of $X^T$, let $\cN_F = \cN_{F/X}$ be the normal bundle of $F$ in $X$ and let $\cN_F = \bigoplus_{\alpha} \cN_{F,\alpha}$ be the $T$-weight decomposition, where $T$ acts on $\cN_{F,\alpha}$ by the character $\alpha\in \Hom(T,\C^\times)$. Let $\rho_{F,\alpha,j}$, $j=1,\dots,\rank \cN_{F,\alpha}$ be the Chern roots of $\cN_{F,\alpha}$. For $\beta \in N_1^T(X)$, we define the operator $\hcS^\beta \colon \cH^{\rm rat}_X \to Q^{\beta+\sigma_{\rm max}(-\ovbeta)}\cH^{\rm rat}_X$ using the localization isomorphism $H_{\hT}^*(X)_{\rm loc} \cong \bigoplus_F H_{\hT}^*(F)_{\rm loc}$ as follows: 
\[
\left. \hcS^\beta \bbf \right|_{F} = Q^{\beta + \sigma_{F}(-\ovbeta)} 
\left( \prod_\alpha \prod_{j=1}^{\rank \cN_{F,\alpha}} 
\frac{\prod_{c=-\infty}^0  \rho_{F,\alpha,j}+\alpha+ cz}
{\prod_{c=-\infty}^{-\alpha \cdot \ovbeta} 
\rho_{F,\alpha,j}+ \alpha + cz} 
\right)  
e^{-z \ovbeta \partial_{\lambda}}(\bbf|_{F})
\]
where  $\sigma_{F}(-\ovbeta) \in N_1^{\rm sec}(E_{-\ovbeta})\subset N_1^T(X)$ is the section class of $E_{-\ovbeta} \to \PP^1$ associated with a fixed point in $F$, $\alpha$ is identified with an element of $H^2_T(\pt)$ and $e^{-z \ovbeta \partial_{\lambda}}$ is the shift $f(\lambda,z) \mapsto f(\lambda- \ovbeta z, z)$ of equivariant parameters acting on $\C(\lambda,z)_{\rm hom}$. Note that $\beta + \sigma_{F}(-\ovbeta)$ lies in $N_1(X)$. 
\end{definition} 

\begin{remark} 
\label{rem:sigma_F} 
The map $\sigma_F\colon H_2^T(\pt,\Z) \cong \Hom(\C^\times,T) \to N_1^T(X)\subset H_2^T(X,\Z)$ associated with a fixed component $F$ is the push-forward along the inclusion $\pt \subset F \hookrightarrow X$. It defines a \emph{linear} splitting of \eqref{eq:equiv_homology_seq} and is dual to the restriction map $H^2_T(X,\Z) \to H^2_T(\pt,\Z)$, i.e.~$\xi\cdot \sigma_F(k) = (\xi|_{\pt}) \cdot k$ for $\xi\in H^2_T(X,\Z)$, $\pt \in F$ and $k \in H_2^T(\pt,\Z) \cong \Hom(\C^\times,T)$. We also note that $\sigma_{\rm max}(k) = \sigma_{F_{\rm max}(k)}(k)$. 
\end{remark} 

\begin{proposition}[{\cite[Theorem 3.14]{Iritani:shift}}] 
\label{prop:shift_fundsol} 
The fundamental solution $M_X(\tau)$ in the equivariant theory satisfies $\hcS^\beta \circ M_X(\tau) = M_X(\tau) \circ \hbS^\beta(\tau)$. 
\end{proposition} 

\begin{corollary} 
The extended shift operators define a representation of $N_1^T(X)$, i.e.~$\hbS^0(\tau) = \id$, $\hbS^{\beta_1+\beta_2}(\tau) = \hbS^{\beta_1}(\tau) \circ \hbS^{\beta_2}(\tau)$. Moreover, they satisfy 
\begin{align} 
\label{eq:shift_qconn} 
& [\nabla_{\tau^{i,k}},\hbS^\beta(\tau)] = [\nabla_{z\partial_z}, \hbS^\beta(\tau)] = 0, \quad 
[\nabla_{\xi Q\partial_Q}, \hbS^\beta(\tau)] = (\xi\cdot \beta) \hbS^\beta(\tau), \\ 
\label{eq:shift_pairing} 
& e^{-z \ovbeta \partial_\lambda} P_X(f, g) = P_X(\hbS^{-\beta}(\tau) f, \hbS^\beta(\tau)g) 
\end{align} 
for $\xi\in H^2_T(X)$, $f,g\in H^*_T(X)[z][\![Q,\btau]\!]$. 
\end{corollary} 
\begin{proof} 
By Proposition \ref{prop:shift_fundsol}, it suffices to prove the corresponding properties for $\hcS^\beta$.  We have $\hcS^0=\id$, $\hcS^{\beta_1+\beta_2} = \hcS^{\beta_1} \circ \hcS^{\beta_2}$. The equation \eqref{eq:shift_qconn} follows from \eqref{eq:fundsol_qconn} and the following properties: 
\[
[\partial_{\tau^{i,k}}, \hcS^\beta] = [z\partial_z - z^{-1} c_1^T(X) + \mu_X, \hcS^\beta] =0, \quad 
[\xi Q\partial_Q + z^{-1}\xi, \hcS^\beta] = (\xi \cdot \beta) \hcS^\beta. 
\]
We can verify these properties for $\hcS^\beta$ by using Remark \ref{rem:sigma_F} and $c_1^T(X)|_F = c_1(F) + \sum_{\alpha,j} (\rho_{F,\alpha,j} + \alpha)$, etc. 
The equation \eqref{eq:shift_pairing} follows from \eqref{eq:fundsol_pairing} and the equality 
\[
e^{-z \ovbeta\partial_\lambda} P_X(f,g) = P_X(\hcS^{-\beta} f, \hcS^\beta g) 
\]
that easily follows from the localization formula \cite{Atiyah-Bott,Berline-Vergne} in equivariant cohomology. 
\end{proof} 

\begin{proposition} 
\label{prop:module_over_extended_shift} 
Let $\NEN^T(X)\subset N_1^T(X)$ be the monoid generated by $\NEN(X)$ and $-\sigma_{\rm max}(-k)\in N_1^{\rm sec}(E_k) \subset N_1^T(X)$ for all $k\in \Hom(\C^\times,T)$. Then $\hbS^\beta(\tau)$ with $\beta \in  \NEN^T(X)$ preserves $\QDM_T(X) = H^*_T(X)[z][\![Q,\btau]\!]$. If $T$ acts on $X$ with a finite generic stabilizer, $\QDM_T(X)$ has the structure of a module over $\C[z][\![\NEN^T(X)]\!]$ via the extended shift operators.  
\end{proposition} 
\begin{proof} 
It follows immediately from Definition \ref{def:shift} that $\hbS^\beta(\tau)$ preserves $\QDM_T(X)$ if $\beta \in \NEN^T(X)$. 
The operator $\bS_k(\tau) = \hbS^{-\sigma_{\rm max}(-k)}(\tau)$ has a strictly positive degree unless $k(\C^\times)$ acts trivially on $X$, since $c_1^T(X)\cdot \sigma_{\rm max}(-k)$ is the sum of weights of the normal bundle of $F_{\rm max}(-k)$ with respect to the $(-k)(\C^\times)$-action. In particular, if the generic stabilizer is finite, $\bS_k(\tau)$ has positive degree unless $k=0$. By the convention of graded completion, $\QDM_T(X)$ has the structure of a module over $\C[z][\![\NEN^T(X)]\!]=\C[z, \{\bS_k(\tau)\}_{k\neq 0}][\![Q]\!]$.  
\end{proof}

\subsection{Givental cone} 
\label{subsec:Givental_cone} 
We briefly review the Givental cone encoding all genus-zero descendant Gromov-Witten invariants \cite{Givental:symplectic, Coates-Givental}. 
Note that we adopt the sign convention for $z$ opposite to the original definition \cite{Givental:symplectic,Coates-Givental}. 

Let $X$ be a smooth projective variety with $T$-action; in the non-equivariant case, we choose $T=\{1\}$. The \emph{Givental space} of $X$ is the infinite-dimensional $\C[\![Q]\!]$-module  
\begin{equation} 
\label{eq:Givental_space} 
\cH_X = H^*_T(X)(\!(z^{-1})\!)[\![Q]\!] 
\end{equation} 
equipped with the symplectic form $\Omega(\bbf, \bg) = - \Res_{z=\infty} (\bbf(-z),\bg(z))_Xdz$. Following the convention in  \S\ref{subsec:formal_power_series}, we mean by $H^*_T(X)(\!(z^{-1})\!)$ the graded completion of $H^*_T(X)[z,z^{-1}]$ with respect to the grading $\deg z =2$. In the non-equivariant case, we have $H^*(X)(\!(z^{-1})\!) = H^*(X)[z,z^{-1}]$. 
The Givental space is equipped with the decomposition $\cH_X = \cH_+ \oplus \cH_-$ into maximally isotropic subspaces: 
\[
\cH_+:= H^*_T(X)[z][\![Q]\!] \quad \text{and} \quad \cH_- :=z^{-1}H^*_T(X)[\![z^{-1}]\!][\![Q]\!]. 
\] 
The symplectic form $\Omega$ identifies $\cH_-$ with the dual of $\cH_+$ (via $v \mapsto \Omega(\cdot,v)$) and $\cH_X$ with the cotangent bundle $T^*\cH_+$ of $\cH_+$. The \emph{genus-zero descendant Gromov-Witten potential}  $\cF_X$ is a function on the formal neighbourhood of $z\in \cH_+$ given by 
\begin{equation} 
\label{eq:descendant_GW_potential}
\cF_X(z+ \bt(z)) = \sum_{\substack{n\ge 0, d\in \NEN(X)\\ (n,d) \neq (0,0), (1,0), (2,0)}} \corr{\bt(-\psi),\dots,\bt(-\psi)}_{0,n,d}^{X,T} \frac{Q^d}{n!} 
\end{equation} 
where $\bt(z) = \sum_{n=0}^\infty t_n z^n$ with $t_n= \sum_{i=0}^s t_n^i \phi_i \in H^*_T(X)$ is a coordinate of $\cH_+$ centred at $z$. The \emph{Givental cone} $\cL_X\subset T^*\cH_+ \cong \cH_X$ is defined to be the graph of the differential of $\cF_X$. It consists of points 
\begin{align} 
\label{eq:point_on_the_cone} 
\begin{split}
\cJ_X(\bt) &= z + \bt(z) + \sum_{i=0}^s \sum_{k=0}^\infty (-1)^k\frac{\phi^i}{z^{k+1}} \parfrac{\cF_X}{t_i^k} \\
&= z + \bt(z) + \sum_{i=0}^s\sum_{\substack{n\ge 0, d\in \NEN(X) \\ (n,d)\neq (0,0), (1,0)}}  \phi^i 
\corr{\bt(-\psi),\dots,\bt(-\psi),\frac{\phi_i}{z-\psi}}_{0,n+1,d}^{X,T} \frac{Q^d}{n!} 
\end{split} 
\end{align} 
with $\bt(z) \in \cH_+$. Here $1/(z-\psi)$ should be expanded in the geometric series $\sum_{n=0}^\infty z^{-n-1}\psi^n$; in the equivariant theory, we need infinitely many negative powers of $z$ since $\psi$ may not be nilpotent. On the other hand, the virtual localization formula shows that points of the form \eqref{eq:point_on_the_cone} lie in the rational form (already appearing in  \eqref{eq:rational_Givental_space}): 
\[
\cH^{\rm rat}_X = H^*_{\hT}(X)_{\rm loc}(X)[\![Q]\!] = H^*_T(X)\otimes_{\C[\lambda]}\C(\lambda,z)_{\rm hom}[\![Q]\!].  
\]
We embed $\cH^{\rm rat}_X$ into the localized form $\cH^{\rm loc}_X =H^*_T(X)_{\rm loc}(\!(z^{-1})\!)[\![Q]\!]$ by the Laurent expansion at $z=\infty$, where $H^*_T(X)_{\rm loc}$ is the localization of $H^*_T(X)$ by non-zero homogeneous elements of $H^*_T(\pt)=\C[\lambda]$. The point \eqref{eq:point_on_the_cone} lies in the intersection $\cH^{\rm rat}_X \cap \cH_X$ inside the localized Givental space $\cH^{\rm loc}$. 
The $J$-function \eqref{eq:J-function} multiplied by $z$ is a finite-dimensional family $\tau \mapsto z J_X(\tau)$ of elements on $\cL_X$, which equals $\cJ_X(\bt)$ with $\bt(z) = \tau$. In view of this,  $\cJ_X(\bt)$ is called the \emph{big $J$-function}.  

The Givental cone satisfies very special geometric properties \cite{Givental:symplectic}, often referred to as being \emph{overruled}. Let $T_\tau$ be the tangent space of $\cL_X$ at $z J_X(\tau)$. It is closed under multiplication by $z$ and has the structure of an $H^*_T(\pt)[z][\![Q]\!]$-module. In fact, $T_\tau$ is freely generated by the derivatives $z \partial_{\tau^i} J_X(\tau) =M_X(\tau) \phi_i$ over $H^*_T(\pt)[z][\![Q]\!]$, i.e.~
\[
T_\tau = M_X(\tau) \cH_+. 
\]
The Givental cone $\cL_X$ can be written as the union of semi-infinite subspaces: 
\begin{equation} 
\label{eq:Givental_cone}
\cL_X = \bigcup_{\tau\in H_T^*(X)[\![Q]\!]} z T_\tau 
\end{equation} 
in the formal neighbourhood of $z J_X(0)$, 
and $T_\tau$ is tangent to $\cL_X$ exactly along $z T_\tau \subset \cL_X$.  
We note that the tangent space $T_\tau$ is preserved by the differential operators $z\partial_{\tau^i}, z\xi Q\partial_Q + \xi$, $z^2 \partial_z - c_1^T(X) + z \mu$ by  \eqref{eq:fundsol_qconn}, where $\xi \in H^2_T(X)$. We can think of $\cL_X$ as a linear-algebraic incarnation of the quantum $D$-module. 

A remarkable property in the equivariant case is that \emph{$T_\tau$ is invariant also under shift operators}: 
\begin{equation*} 
\hcS^\beta T_\tau \subset Q^{\beta+\sigma_{\rm max}(-\ovbeta)} T_\tau, \qquad \text{for $\beta \in N_1^T(X)$.}
\end{equation*} 
This follows from $M_X(\tau) \circ \hbS^\beta(\tau) = \hcS^\beta\circ M_X(\tau)$ (see Proposition \ref{prop:shift_fundsol}).

\begin{remark} 
\label{rem:functor_of_points} 
We can understand the Givental cone $\cL_X$ in terms of a functor of points. For any set $\bx = (x_0,x_1,x_2,\dots)$ of formal variables, an \emph{$H_T^*(\pt)[\![Q,\bx]\!]$-valued point} of $\cL_X$ is a point of the form \eqref{eq:point_on_the_cone} for $\bt(z) \in \cH_+[\![\bx]\!]$ satisfying $\bt(z)|_{Q=\bx=0} =0$. 
Such a point lies in $\cH_X[\![\bx]\!]$. 
The equality \eqref{eq:Givental_cone} means that any $H^*_T(\pt)[\![Q,\bx]\!]$-valued point on $\cL_X$ can be written in the form $z M_X(\tau) v$ for unique $\tau\in H^*_T(X)[\![Q,\bx]\!]$ and $v\in \cH_+[\![\bx]\!]$ such that $\tau|_{Q=\bx=0} = 0$ and $v|_{Q=\bx=0} = 1$. We refer the reader to \cite[Appendix B]{CCIT:computing} (see also \cite[Appendix A]{Lutz-Shafi-Webb:Grassmann_flop}) for a treatment of $\cL_X$ as an infinite-dimensional formal scheme over the Novikov ring. 
\end{remark}

\subsection{Twisted Gromov-Witten invariants} 
\label{subsec:twisted_GW} 
In this section we review Gromov-Witten invariants twisted by a vector bundle $V \to X$ and an invertible multiplicative characteristic class $\bc$ as studied by Coates-Givental \cite{Coates-Givental}. We will mostly consider the case where $\bc$ is the inverse of the equivariant Euler class $e_\lambda$ or its variant $\te_\lambda$: 
\begin{equation} 
\label{eq:equiv_Euler} 
e_\lambda(V) = \sum_{i=0}^{\rank (V)} c_i(V) \lambda^{\rank V - i}, 
\qquad 
\te_\lambda(V) = \sum_{i=0}^{\rank V} c_i(V) \lambda^{-i} 
\end{equation} 
where $\lambda$ is the equivariant parameter of the $\C^\times$-action scaling fibres of $V$. 
Consider the universal stable map 
\[
\xymatrix{ 
C_{0,n,d} \ar[r]^{f} \ar[d]_\pi & X \\ 
X_{0,n,d} & 
}
\]
and define a virtual bundle $V_{0,n,d} := \R\pi_* (f^* V) \in K(X_{0,n,d})$ on the moduli space $X_{0,n,d}$.  The \emph{$(V,\bc)$-twisted Gromov-Witten invariants} are defined as 
\[
\corr{\alpha_1 \psi^{k_1},\dots,\alpha_n \psi^{k_n}}_{0,n,d}^{X,(V,\bc)} 
:= \int_{[X_{0,n,d}]_{\rm vir}} \left(\prod_{i=1}^n \ev_i^*(\alpha_i) \psi_i^{k_i} \right) \cup \bc(V_{0,n,d}) 
\]
where $\alpha_1,\dots,\alpha_n \in H^*(X)$, $k_1,\dots,k_n \in \N$. Note that the right-hand side is essentially a \emph{non-equivariant} integral; the torus acts on $X_{0,n,d}$ trivially and can act only on fibres of $V_{0,n,d}$.  Using the twisted Poincar\'e pairing 
\[
(\alpha_1,\alpha_2)_X^{(V,\bc)} := \int_X \alpha_1 \cup \alpha_2 \cup \bc(V)  
\]
and the twisted Gromov-Witten invariants in place of the ordinary ones, we can similarly define the quantum product \eqref{eq:qprod}, the fundamental solution \eqref{eq:fundsol} and the descendant Gromov-Witten potential \eqref{eq:descendant_GW_potential} in the $(V,\bc)$-twisted theory. 

Let $R$ be a coefficient ring of the classes $\bc^\pm$: we take $R=\C[\lambda,\lambda^{-1}]$ for $\bc= e_\lambda^{-1}$ and $R=\C[\lambda^{-1}]$ for $\bc = \te_\lambda^{-1}$ with grading given by $\deg \lambda=2$. 
The \emph{Givental space} of the twisted theory is the module 
\begin{equation} 
\label{eq:twisted_Givental_space} 
\cH^{\rm tw}_X := H^*(X)\otimes R(\!(z)\!)[\![Q]\!] 
\end{equation} 
equipped with the symplectic form $\Omega_{\rm tw}(\bbf,\bg) = \Res_{z=0}(\bbf(-z),\bg(z))_X^{(V,\bc)} dz$. 
Here we allow formal power series in $z$ \emph{infinite in the positive direction}: note that this is opposite to the previous version \eqref{eq:Givental_space}. Using the isotropic decomposition $\cH^{\rm tw}_X = \cH_+^{\rm tw} \oplus \cH_-^{\rm tw}$ with 
\[
\cH_+^{\rm tw} = H^*(X)\otimes R[\![z]\!][\![Q]\!] \quad \text{and} \quad 
\cH_-^{\rm tw} = z^{-1} H^*(X)\otimes R[z^{-1}][\![Q]\!], 
\]
we identify $\cH^{\rm tw}_X$ with the cotangent bundle of $\cH_+^{\rm tw}$ as before. 
The twisted \emph{Givental cone} $\cL_X^{\rm tw}\subset \cH^{\rm tw}_X$ is defined to be the graph of the differential of the twisted descendant Gromov-Witten potential. It consists of points of the form 
\begin{equation} 
\label{eq:point_on_the_twisted_cone} 
\cJ_X^{\rm tw}(\bt)= z + \bt(z) + \sum_{i=0}^s\sum_{\substack{n\ge 0, d\in \NEN(X) \\ (n,d)\neq (1,0)}}  \frac{\phi^i}{\bc(V)} 
\corr{\bt(-\psi),\dots,\bt(-\psi),\frac{\phi_i}{z-\psi}}_{0,n+1,d}^{X,(V,\bc)} \frac{Q^d}{n!} 
\end{equation} 
with $\bt(z) \in \cH_+^{\rm tw}$. Since the $\psi$-class is nilpotent here, this is contained in $\cH^{\rm tw}_X$. The nilpotence of $\psi$ also ensures that the functor of points as in Remark \ref{rem:functor_of_points} also makes sense. An \emph{$R[\![Q,\bx]\!]$-valued point} of $\cL_X^{\rm tw}$ is a point of the form \eqref{eq:point_on_the_twisted_cone} for $\bt(z) \in \cH_+^{\rm tw}[\![\bx]\!]$ satisfying $\bt(z)|_{Q=\bx=0} = 0$ (where we work with a discrete topology on $R$, cf.~\S\ref{subsec:QRR}); such a point lies in $\cH^{\rm tw}_X[\![\bx]\!]$ by the nilpotence of $\psi$. 



\subsection{Quantum Riemann-Roch theorem} 
\label{subsec:QRR} 
We review the quantum Riemann-Roch theorem of Coates-Givental \cite{Coates-Givental} that describes twisted Gromov-Witten invariants in terms of ordinary ones. We only need the statement at genus zero. 

Let $\bc(\cdot) = \exp(\sum_{k=0}^\infty s_k \ch_k(\cdot))$ be the universal invertible multiplicative class with parameters $\bs =(s_0,s_1,s_2,\dots)$. We set $\deg s_k = - 2k$ and take the coefficient ring $R$ for $\bc$ to be  $R=\C[\![\bs]\!]$. We can write the Givental space \eqref{eq:twisted_Givental_space} in this case as 
\[
\cH^{\rm tw}_X =  H^*(X)[z,z^{-1}] [\![Q,\bs]\!] 
\]
because we are working with the graded completion (see \S \ref{subsec:formal_power_series}). 
The quantum Riemann-Roch theorem requires us to work with a topology where the parameters $s_k$ are considered small (as described in  \cite[Appendix B]{CCIT:computing}); for any set $\bx=(x_0,x_1,x_2,\dots)$ of formal parameters, a $\C[\![Q,\bs,\bx]\!]$-valued point on $\cL_X^{\rm tw}$ is a point of the form \eqref{eq:point_on_the_twisted_cone} for  $\bt(z) \in \cH_+^{\rm tw}[\![\bx]\!]$ satisfying $\bt(z)|_{Q=\bx=\bs=0}=0$. 
The \emph{quantum Riemann-Roch operator} $\Delta_{(V,\bc)} \colon \cH^{\rm tw}_X \to \cH^{\rm tw}_X$ is defined to be: 
\begin{equation} 
\label{eq:QRR_operator} 
\Delta_{(V,\bc)} = \exp\left(\sum_{l,m\ge 0, l+m\ge 1} s_{l+m-1} \frac{B_m}{m!} \ch_l(V) (-z)^{m-1} \right) 
\end{equation} 
where $B_m$ are the Bernoulli numbers given by $\sum_{m=0}^\infty \frac{B_m}{m!} x^m = \frac{x}{e^x-1}$;  we have $B_0=1$, $B_1 = -\frac{1}{2}$, $B_2 = \frac{1}{6}, \dots$. 

\begin{theorem}[{\cite[Corollary 4]{Coates-Givental}}]  
\label{thm:QRR} 
The $(V,\bc)$-twisted Givental cone is given by $\cL_X^{\rm tw}=\Delta_{(V,\bc)} \cL_X$, where $\cL_X$ is the untwisted cone defined within $\cH_X^{\rm tw}$. 
\end{theorem} 

\begin{remark} 
\label{rem:e_te_twist}
Technically speaking, Theorem \ref{thm:QRR} does not apply directly to $\bc = e_\lambda^{-1}$ where the corresponding parameter $s_0=-\log \lambda$ is not ``small": 
\[
s_k = \begin{cases} 
- \log \lambda &\text{for  $k=0$}; \\ 
(-1)^k (k-1)! \lambda^{-k} & \text{for $k>0$}. 
\end{cases} 
\]
There is no such problem with $\bc=\te_\lambda^{-1}$ where $s_0 =0$ and $s_k$ with $k>0$ is the same as above. To determine the $e_\lambda^{-1}$-twisted cone, we can use the following relationship between the big $J$-functions $\cJ_X^{\rm tw}(\bt)$, $\tcJ_X^{\rm tw}(\bt)$, respectively, of the $(V,e_\lambda^{-1})$- and $(V,\te_\lambda^{-1})$-twisted theories 
\[
\cJ_X^{\rm tw}(\bt) = \tcJ_X^{\rm tw}(\bt)\bigr|_{Q \to Q \lambda^{-c_1(V)}} 
\]
where $Q\to Q\lambda^{-c_1(V)}$ means to replace $Q^d$ with $Q^d \lambda^{-c_1(V) \cdot d}$. 
Note that the Novikov variables have different degrees between these theories. The degree of $Q^d$ equals $2 c_1(X)\cdot d$ in the $(V,\te_\lambda^{-1})$-twisted theory and equals $2 (c_1(X) +c_1(V))\cdot d$ in the $(V,e_\lambda^{-1})$-twisted theory. 
\end{remark} 

\begin{remark} 
\label{rem:more_than_one_twists}
Later we consider a more general fibrewise $T$-action on a vector bundle $V\to X$ and the Gromov-Witten invariants twisted by $V$ and the inverse equivariant Euler class $e_T^{-1}$. Decomposing $V$ into $T$-eigenbundles $V_\alpha$, we can calculate $e_T^{-1}$-twisted invariants by the following (straightforward) generalization of Theorem \ref{thm:QRR}: the Givental cone twisted by several bundles $V_\alpha$ and classes $\bc_\alpha$ is given by $\cL^{\rm tw}_X= (\prod_\alpha \Delta_{(V_\alpha,\bc_\alpha)})\cL_X$, see \cite[Theorem 1.1]{Tonita:twisted_orbifoldGW}. 
\end{remark}

\section{Geometry of blowups} 
\label{sec:geom_blowup}

\subsection{Blowup as a variation of GIT quotient} 
\label{subsec:GIT} 
Let $X$ be a smooth projective variety and $Z \subset X$ be a smooth subvariety of codimension $r$. We describe the blowup $\varphi\colon \tX\to X$ along $Z$ as a variation of GIT quotient, following \cite[Example 3.5.5]{BFK:VGIT} (see also \cite{Guillemin-Sternberg:birational,Lerman:cut}). 

Let $\hvarphi \colon W\to X\times \PP^1$ be the blowup of $X\times \PP^1$ along $Z\times \{0\}$. We identify $\tX$ with the strict transform of $X\times \{0\}$ in $W$ and $X$ with the subspace $X\times \{\infty\} \subset W$. 
We denote by $D$ the exceptional divisor of $\tX \to X$ and by $\hD$ the exceptional divisor of $W \to X\times \PP^1$. Let $\cN_{Z/X} \to Z$ be the normal bundle of $Z$ in $X$. Then $\hD\cong \PP(\cN_{Z/X}\oplus 1)$ contains $D \cong \PP(\cN_{Z/X})$ as the infinity divisor and $\hD$ intersects $\tX$ transversely along $D$. The strict transform of $Z\times \PP^1$ in $W$ is $Z\times \PP^1$ itself and we regard $Z\times \PP^1$ as the subspace of $W$. We also identify $Z$ with the subspace $Z\times \{0\} = (Z\times \PP^1) \cap \hD$ of $W$. See Figure \ref{fig:W}. 

\begin{figure}[t] 
\centering 
\begin{tikzpicture}[x=0.7cm, y=0.7cm]
\draw (0,5) -- (8,5) -- (10,6) -- (2,6) -- (0,5);
\draw (0,0) -- (8,0) -- (10,1) -- (5.6,1); 
\draw (4.4,1) -- (2,1) -- (0,0);
\filldraw (5,5.5) circle [radius=0.05]; 
\draw[thick, dotted] (5,5.5) -- (5,4.9); 
\draw (5,4.9) -- (5,3); 
\filldraw (5,3) circle [radius =0.05]; 
\draw[thick,dotted] (5.5,0.5) arc [x radius = 0.5, y radius =0.2, start angle=0, end angle =180];
\draw[thick] (4.5,0.5) arc [x radius =0.5, y radius =0.2, start angle =180, end angle = 360]; 
\draw (5,3) .. controls (4.6,2.9) .. (4.5,0.5) ;
\draw (5,3) .. controls (5.4,2.9) .. (5.5,0.5) ;
\draw (-0.5,0.5) node {$\tX$}; 
\draw (-1.5,5.5) node {$X= X\times \{\infty\}$}; 
\draw (4,0.5) node {$D$}; 
\draw (5,1.8) node {$\hD$}; 
\draw (4.3,3) node {$Z$}; 
\draw (6,4.3) node {$Z\times \PP^1$}; 
\draw[->] (14,0)--(14,6.5); 
\filldraw (14,5.5) circle [radius=0.05]; 
\filldraw (14,3) circle [radius =0.05]; 
\filldraw (14,0.5) circle [radius =0.05]; 
\draw (14.3,5.5) node {$t$}; 
\draw (14.3,3) node {$\epsilon$}; 
\draw (14.3,0.5) node {$0$}; 
\draw[dotted] (10,5.5) -- (14,5.5); 
\draw[dotted] (10,3) -- (14,3); 
\draw[dotted] (10,0.5) -- (14,0.5); 
\draw[->] (11,4) -- (12,4); 
\draw (11.5,4.3) node {$\mu$}; 
\end{tikzpicture} 
\caption{$W = \Bl_{Z\times \{0\}} (X\times \PP^1)$ and a moment map $\mu \colon W\to \R$}
\label{fig:W}
\end{figure} 

The subspace $X = X\times \{\infty\} \subset W$ has the Zariski open neighbourhood $X\times (\PP^1\setminus \{0\})$, $\tX \subset W$ has the Zariski open neighbourhood $W\setminus (X\cup (Z\times \PP^1))$ isomorphic to the total space of the line bundle $\cO_{\tX}(-D)$ and $Z\subset W$ has the normal bundle $\cN_{Z/X} \oplus 1$. 

The spaces $X$, $\tX$ arise as the GIT quotients of $W$ for different stability conditions. We henceforth denote by $T$ the algebraic torus $\C^\times$ of rank 1. Consider the $T$-action on $X\times \PP^1$ given by the standard $\C^\times$-action $v \mapsto \lambda v$ on $\PP^1=\C\cup \{\infty\}$ and the trivial one on $X$. This $T$-action naturally lifts to $W$ and the $T$-fixed locus is given by $W^{T} = X \sqcup Z \sqcup \tX$. There are two possibilities of nonempty GIT quotients (without strictly semistable loci): the stable loci 
\begin{align} 
\label{eq:stable_loci} 
\begin{split} 
W^\st_X &= W \setminus (X \cup \hD \cup \tX) = X\times \C^\times\\ 
W^\st_\tX &= W \setminus (X \cup (Z\times \PP^1) \cup \tX) = \Tot(\cO_\tX(-D))\setminus \tX 
\end{split} 
\end{align} 
yield the quotients $X=W^\st_X/T$, $\tX = W^\st_\tX/T$ respectively. 

\subsection{Symplectic picture} 
\label{subsec:symp} The above GIT quotients can be described also as symplectic quotients (see \cite{Kirwan:coh_quotient}). Let $N^1(X) \subset H^2(X,\Z)$ and $N^1_{T}(W) \subset H^2_{T}(W,\Z)$ denote the (equivariant) N\'eron-Severi groups; they are the images of $c_1\colon \Pic(X)\to H^2(X,\Z)$ and $c_1^T\colon \Pic^{T}(W) \to H^2_{T}(W,\Z)$ respectively. Let $\omega \in N^1(X)_\R=N^1(X)\otimes\R$ be an ample class. Let $[X], [\hD] \in N^1_{T}(W)$ be the classes of the divisors $X=X\times \{\infty\}$, $\hD$ respectively. Consider the equivariant class 
\[
\homega = \hvarphi^*\pr_1^*\omega + t[X] - \epsilon [\hD] \in N^1_{T}(W)_\R = N^1_T(W) \otimes \R
\]
where $\pr_1 \colon X\times \PP^1\to X$ is the first projection and $t,\epsilon>0$. When $\epsilon>0$ is sufficiently small, the non-equivariant limit of $\homega$ is ample and hence represented by an $S^1$-invariant K\"ahler form $\alpha$. 
A de Rham representative of $\homega$ in the Cartan model\footnote{The Cartan model for $S^1$-equivariant cohomology is the complex $\Omega^*(W)^{S^1} [\lambda]$ equipped with the differential $d - \lambda \iota$, where $\iota$ is the contraction with respect to the fundamental vector field of the $S^1$-action.} is given by an equivariantly closed 2-form $\alpha - \lambda \mu$, where $\mu \colon W\to \R$ is the moment map with respect to $\alpha$ and $\lambda$ is the equivariant parameter. The values of $\mu$ on the $S^1$-fixed loci are determined by the cohomology class $\homega=[\alpha -\lambda \mu]$; we have $\mu|_{\tX}= 0$, $\mu|_Z = \epsilon$ and $\mu|_X = t$. Then we can describe $X$, $\tX$ as symplectic quotients: 
\[
\mu^{-1}(a)/S^1 \cong 
\begin{cases}
X & \text{when $\epsilon<a<t$;} \\
\tX & \text{when $0<a<\epsilon$.}  
\end{cases}
\]
\subsection{Equivariant divisor (curve) classes of $W$} 
\label{subsec:N1_W}
We fix a presentation of the equivariant N\'eron-Severi group $N^1_{T}(W)$ and its dual $N_1^T(W)$. Using the class $[X]\in N^1_{T}(X\times \PP^1)$ of the divisor $X=X\times \{\infty\}$, we have the splitting:  
\[
N^1_{T}(X\times \PP^1) = \pr_1^* N^1(X) \oplus \Z [X]  \oplus \Z \lambda.  
\]
We also have  
\[
N^1_{T}(W) = \hvarphi^* N^1_{T}(X\times \PP^1) \oplus \Z (-[\hD]).  
\]
Combining these isomorphisms, we get a decomposition 
\begin{equation} 
\label{eq:2nd_cohomology} 
N^1_T(W) = \hvarphi^*\pr_1^* N^1(X) \oplus \Z [X] \oplus \Z (-[\hD])\oplus \Z \lambda \cong 
 N^1(X) \oplus \Z^3
\end{equation} 
where an element $(\omega, t,\epsilon,a)$ on the right-hand side corresponds to the class $\hvarphi^* \pr_1^*\omega + t[X] - \epsilon [\hD] + a \lambda \in N^1_{T}(W)$. The divisor class $[\tX]$ equals $[X]-[\hD]+\lambda=(0,1,1,1)$.  
We also consider the dual decomposition of the equivariant curve class group (see \S\ref{subsec:equiv_curve_classes}): 
\begin{equation} 
\label{eq:2nd_homology} 
N_1^T(W) \cong N_1(X) \oplus \Z^3.  
\end{equation} 
An element $(d,k,l,m)$ on the right-hand side corresponds to the class $\beta\in N_1^{T}(W)$ with $\pr_{1*}\hvarphi_* \beta = (d,m)\in N_1^T(X)=N_1(X)\oplus \Z$, $[X] \cdot \beta = k$ and $-[\hD] \cdot \beta = l$. The subgroup $N_1(W)\subset N_1^T(W)$ consists of elements of the form $(d,k,l,0)$. See Table \ref{tab:homology_classes} for various curve classes in this presentation. 

\begin{table}[t]
\caption{Curve classes in the presentation \eqref{eq:2nd_homology} and Novikov variables.} 
\label{tab:homology_classes} 
\begin{tabular}{lll} 
\toprule
Class & Variable & Description \\ \midrule 
$(d,0,0,0)$ & $Q^d$ & the push-forward $i_{X*}d$ along $i_X\colon X\times \{\infty\}\hookrightarrow W$ \\
$(0,1,0,0)$ & $x$ & the class of a line $\pt\times \PP^1\subset W$ with $\pt \in X\setminus Z$\\
$(0,0,1,0)$ & $y$ & the class of a line in the fibre of $\hD=\PP(\cN_{Z/X}\oplus 1) \to Z$ \\ 
$(0,1,-1,0)$ &  $xy^{-1}$ & the class of a line $\pt\times \PP^1$ contained in $Z\times \PP^1\subset W$ \\ 
$(0,0,0,1)$ & $S$ & the generator of $H_2^{T}(\pt)\hookrightarrow N_1^{T}(W)$ for $\pt\in\tX\subset W$ \\ 
\bottomrule
\end{tabular} 
\end{table} 

We write $Q^d, x, y, S$ for the elements of the group ring $\C[N_1^{T}(W)]$ corresponding respectively to $(d,0,0,0)$, $(0,1,0,0)$, $(0,0,1,0)$, $(0,0,0,1)$ under the isomorphism \eqref{eq:2nd_homology}. The variables $Q,x,y$ play the role of the Novikov variables of $W$ and we denote them collectively by $\QW=(Q,x,y)$. 
We also use a letter $\hS=(\QW,S)$ to denote the variable of the group ring $\C[N_1^T(W)]$.  
The degrees of $Q^d,x,y,S$ are given by $2 c_1^T(W)$; by Lemma \ref{lem:c_1(W)} below, we have $\deg Q^d = 2 c_1(X)\cdot d$, $\deg x = 4$, $\deg y = 2r$ and $\deg S = 2$. We will also use the variable $\frq := y S^{-1}$ later (see \S\ref{subsec:Fourier_projections}). 

\begin{lemma} 
\label{lem:NEN_W} 
The monoid $\NEN(W)$ of effective curves in $W$ is generated by $(d,0,0,0)$ with $d\in \NEN(X)$,  $(\varphi_*\td,0,-[D]\cdot \td,0)$ with $\td \in \NEN(\tX)$ and $(0,1,-1,0)$. 
\end{lemma} 
\begin{proof} 
Since every curve can be deformed to a $T$-invariant curve, it suffices to examine the classes of $T$-invariant curves in $W$. A $T$-invariant irreducible curve $C$ is either contained in the fixed locus $X\sqcup Z \sqcup \tX$ or the closure of a $T$-orbit. 
If $C$ is contained in $X\sqcup Z$, the class is of the form $(d,0,0,0)$ with $d\in \NEN(X)$; if $C$ is contained in $\tX$, the class is of the form $(\varphi_*\td, 0, - [D]\cdot \td,0)$ with $\td \in \NEN(\tX)$; and if $C$ is the closure of a $T$-orbit, the class is one of $(0,1,-1,0)$, $(0,0,1,0)$, $(0,1,0,0)$ (see Table \ref{tab:homology_classes}). Note that the third is the sum of the first two and $(0,0,1,0)$ equals $(\varphi_*\td,0,-[D]\cdot \td,0)$ with $\td$ being the class of a line in a fibre of $D\to Z$. 
\end{proof} 

\begin{lemma} 
\label{lem:c_1(W)} 
We have $c_1^{T}(W)=(c_1(X), 2, r, 1)$ under \eqref{eq:2nd_cohomology}, where recall that $r$ is the codimension of $Z$ in $X$. 
\end{lemma} 
\begin{proof} 
Using $K_W = \hvarphi^* K_{X\times \PP^1} +r [\hD]$, we have 
\begin{align*} 
c_1^{T}(W) & = \hvarphi^*c_1^{T}(X\times \PP^1)  - r [\hD] \\
& = \hvarphi^*(\pr_1^*c_1(X) + \pr_2^*([0]+[\infty]))  - r[\hD] \\ 
& = \hvarphi^*\pr_1^*c_1(X) + 2[X] - r[\hD] + \lambda 
\end{align*} 
where $\pr_1\colon X\times \PP^1 \to X$, $\pr_2 \colon X\times \PP^1 \to \PP^1$ are the projections and $[0],[\infty]\in N^1_{T}(\PP^1)$ denote the class of the divisors $0,\infty \in\PP^1$ respectively. We used the identities $[0]-[\infty] = \lambda$, $\hvarphi^*\pr_2^*[\infty] = [X]$ in the last step. 
\end{proof}

\subsection{Wall and chamber structure of the $T$-ample cone} 
\label{subsec:T-ample_cone}
We describe the wall and chamber structure of the $T$-ample cone of $W$, studied by Dolgachev-Hu \cite{Dolgachev-Hu:VGIT}, Thaddeus \cite{Thaddeus:GIT} and Ressayre \cite{Ressayre:GIT} for general GIT quotients.  

By the Hilbert-Mumford numerical criterion, the stable locus $W^\st(\homega)$ for an equivariant ample class $\homega \in N^1_T(W)_\R$ is given as follows:  
\[
W^\st(\homega) = \left\{x \in W : - \homega|_{x_0} \in \R_{<0} \lambda,\, -\homega|_{x_\infty} \in \R_{>0} \lambda \right \} 
\]
where we set $x_0=\lim_{\lambda \to 0} \lambda \cdot x$, $x_\infty = \lim_{\lambda \to \infty} \lambda \cdot x$ for $x\in W$. Let $Y$ denote $X$ or $\tX$. We define 
\begin{equation*} 
C_Y =\{ \homega \in N^1_T(W)_\R : \text{$\homega$ is ample and $W^\st(\homega) = W^\st_Y$}\},   
\end{equation*} 
where $W^\st_Y$ is given in \eqref{eq:stable_loci}. 
In terms of symplectic quotients, if $\mu_{\homega} \colon W \to \R$ denotes the moment map associated with a de Rham representative of an ample class $\homega\in N^1_T(W)_\R$ as in \S\ref{subsec:symp}, then $\homega\in C_Y$ if and only if $0$ is the regular value of $\mu_{\homega}$ and $Y \cong \mu_{\homega}^{-1}(0)/S^1$.  
The cone 
\[
C_T(W) := \Int(\overline{C_X}\cup \overline{C_\tX})=C_X+C_\tX
\] 
consists of equivariant ample classes whose stable loci are nonempty: we call it the \emph{$T$-ample cone} of $W$. 

\begin{remark} 
Dolgachev-Hu \cite[Definition 0.2.1]{Dolgachev-Hu:VGIT} defined the $T$-ample cone to be the set of equivariant ample classes whose \emph{semistable} loci are nonempty. Our $T$-ample cone is the interior of theirs (see  \cite[Proposition 11]{Ressayre:GIT}).  
\end{remark} 

\begin{remark} 
For a $T$-variety $W$ equipped with a linearization $L\to W$, Shoemaker \cite[Theorem 1.3]{Shoemaker:Kleiman} introduced a cone $\NE(L)^\vee\subset N^1_T(W)_\R$. The cone $\NE(L)^\vee$ is dual to the ``effective cone'' in the sense of quasimap theory, and is potentially bigger than the GIT chamber $C_Y\subset C_T(W)$ associated with $Y=W/\!/_LT$. This suggests that the quasimap $J$-function can have a bigger domain of definition than the $I$-function in Conjecture \ref{conj:reduction} (under the support condition conjectured there). 
\end{remark} 

\begin{lemma} 
\label{lem:GIT_chambers} 
For an ample class $\omega \in N^1(X)_\R$, let $\varepsilon(\omega) = \max\{\epsilon\in \R : \text{$\varphi^*\omega -\epsilon [D]$ is nef\,}\}\in (0,\infty)$ be the Seshadri constant at the subvariety $Z\subset X$. Then we have under \eqref{eq:2nd_cohomology} 
\begin{align*}
C_X & = \{(\omega,t,\epsilon, a) : \text{$\omega$ is ample}, 0<\epsilon<a<t, \epsilon<\varepsilon(\omega)\},  \\
C_\tX & = \{(\omega,t,\epsilon,a) : \text{$\omega$ is ample}, 0<a<\epsilon<t, \epsilon<\varepsilon(\omega)\}.
\end{align*} 
\end{lemma} 
\begin{proof} 
Set $\homega =\hvarphi^*\pr_1^*\omega + t [X] - \epsilon[\hD] + a \lambda$. We have by definition 
\begin{align*} 
C_X &= \{ \homega : \text{$\homega$ is ample}, -\homega|_X \in \R_{>0} \lambda, -\homega|_Z \in \R_{<0}\lambda, 
-\homega|_\tX \in \R_{<0} \lambda \}, \\ 
C_\tX & = \{  \homega : \text{$\homega$ is ample}, -\homega|_X \in \R_{>0} \lambda, -\homega|_Z \in \R_{>0}\lambda, 
-\homega|_\tX \in \R_{<0} \lambda \}. 
\end{align*} 
We note the fact that $-\homega|_X =(t-a)\lambda$, $-\homega|_Z=(\epsilon-a) \lambda$, $-\homega|_\tX = -a \lambda$. The ampleness of $\homega$ implies (and is in fact equivalent to) that $\omega = \homega|_{X\times \infty}$ is ample and that $0<\epsilon<\varepsilon(\omega,t)$, where $\varepsilon(\omega,t)$ is the Seshadri constant at the subvariety $Z\times \{0\}\subset X\times \PP^1$: 
\[
\varepsilon(\omega,t) = \max\{\epsilon\in \R : \text{$\hvarphi^*\pr_1^*\omega + t[X] -\epsilon [\hD]$ is nef}\}. 
\]
Using Lemma \ref{lem:NEN_W}, we can easily find that $\varepsilon(\omega,t) = \min(\varepsilon(\omega), t)$. The presentations for $C_X$, $C_\tX$ follow from these facts. 
\end{proof}

Let $\ovNE(-)$ denote the closure of the cone $\NE(-)\subset N_1(-)_\R$ of effective curves. The dual cones $C_Y^\vee, C_T(W)^\vee \subset N_1^{T}(W)_\R$ of $C_Y$, $C_T(W)$ are given as follows (see Figure \ref{fig:Mori_cones}): 
\begin{align*}
\begin{split} 
C_X^\vee &= \ovNE(W)+\langle (0,1,0,-1), (0,0,-1,1)\rangle_{\R_{\ge 0}}, \\ 
C_\tX^\vee & = \ovNE(W) + \langle (0,0,1,-1), (0,0,0,1) \rangle_{\R_{\ge 0}}, \\ 
C_T(W)^\vee 
& =C_X^\vee \cap C_\tX^\vee  = \ovNE(W) + \langle (0,1,0,-1), (0,0,0,1)\rangle_{\R_{\ge 0}}. 
\end{split} 
\end{align*} 
We also introduce the corresponding monoids in $N_1^T(W)$ as follows: 
\begin{align}
\label{eq:dualmonoids}  
\begin{split} 
C_{X,\N}^\vee & := \NEN(W)+\langle (0,1,0,-1), (0,0,-1,1)\rangle_{\N}, \\ 
C_{\tX,\N}^\vee & := \NEN(W) + \langle (0,0,1,-1), (0,0,0,1) \rangle_{\N}, \\ 
\NEN^T(W) & :=C_{X,\N}^\vee \cap C_{\tX,\N}^\vee = \NEN(W) + \langle (0,1,0,-1), (0,0,0,1)\rangle_{\N}.  
\end{split} 
\end{align} 
where $\NEN(W)=\langle (d,0,0,0),   (\varphi_*\td, 0, -[D]\cdot \td,0),(0,1,-1,0) : d\in \NEN(X), \td\in \NEN(\tX) \rangle_{\N}$ by Lemma \ref{lem:NEN_W}. 

\begin{figure}[t]
\centering 
\begin{tikzpicture}[>=stealth, x=0.7cm, y=0.7cm]
\draw (-5,0) -- (5,0); 
\draw[->] (0,0) -- (-3,0); 
\draw[->] (0,0) -- (3,0);
\draw[->] (0,0) -- (0,3); 
\draw[->] (0,0) -- (3,1.732);
\draw[->] (0,0) --(-3,3);  
\draw (0,3.5) node {\scriptsize $\ovNE(W)$}; 
\draw (-3,-0.4) node {\scriptsize $y S^{-1}$}; 
\draw (3,-0.4) node {\scriptsize $y^{-1} S$}; 
\draw (-3.3,3.2) node {\scriptsize $x S^{-1}$}; 
\draw (3.3,1.8) node {\scriptsize $S$};

\draw [<->] (-2.2,0) arc[start angle =  180, end angle = 30, radius = 2.2]; 
\draw [<->] (1.7,0) arc[start angle =0, end angle = 135, radius = 1.7]; 
\draw (-1,2.5) node {$\tX$}; 
\draw (0.55,1.2) node {$X$}; 
\end{tikzpicture} 

\caption{A schematic picture of the cones $C_X^\vee$ and $C_\tX^\vee$ in $N_1^T(W)$.} 
\label{fig:Mori_cones} 
\end{figure} 


\subsection{Cohomology}  
\label{subsec:cohomology} 
Let $\jmath \colon D\to \tX$, $\hjmath \colon \hD \to W$ be the inclusions of the exceptional divisors and let $\pi=\varphi|_D \colon D\to Z$, $\hpi= \hvarphi|_{\hD} \colon \hD \to Z$ be the projections. Let $p= c_1(\cO_D(1))\in H^2(D)$, $\hp= c_1^T(\cO_\hD(1))\in H^2_T(\hD)$ be the relative hyperplane classes of the projective bundles $D= \PP(\cN_{Z/X}) \to Z$, $\hD = \PP(\cN_{Z/X}\oplus 1) \to Z$ respectively; we have $p= -\jmath^*[D]$ and $\hp= -\hjmath^*[\hD]$. 
Since $\varphi\colon \tX \to X$ is birational, $\varphi_* \varphi^*=\id$ on $H^*(X)$ and hence $H^*(\tX)=\varphi^*H^*(X) \oplus \Ker \varphi_*$. More precisely, we have the following additive decomposition of the (equivariant) cohomology groups of $\tX$ and $W$ (see e.g.~\cite[Theorem 7.31]{Voisin:Hodge_I} in the non-equivariant case; the extension to equivariant cohomology is straightforward). 
\begin{align}
\label{eq:cohomology_bu}
\begin{split}   
H^*(\tX) & = \varphi^*H^*(X) \oplus \bigoplus_{k=0}^{r-2} \jmath_* (p^k \pi^* H^*(Z)) \cong H^*(X)\oplus H^*(Z)^{\oplus (r-1)}, \\
H^*_T(W) &= \hvarphi^*H^*_T(X\times \PP^1) \oplus \bigoplus_{k=0}^{r-1} \hjmath_*(\hp^k \hpi^*H^{*}_T(Z)) \cong H^*_T(X\times \PP^1) \oplus H^*_T(Z)^{\oplus r},
\end{split} 
\end{align} 
where $H^*_T(X\times \PP^1) \cong H^*(X)[\,[0],[\infty]\,]/([0]\cdot [\infty])$ with $[0]-[\infty]=\lambda$ and $H^*_T(Z) = H^*(Z)[\lambda]$. Note that the second components in the decompositions \eqref{eq:cohomology_bu} are $\Ker \varphi_*$, $\Ker \hvarphi_*$ respectively.  

Equivariant cohomology classes of $W$ can be also described in terms of their localizations to the fixed loci. The restriction of a cohomology class 
\[
f=\hvarphi^*\alpha + \sum_{k=0}^{r-1} \hjmath_*(\hp^k \hpi^* \gamma_k)\in H^*_T(W)
\] 
with $\alpha \in H^*_T(X\times \PP^1)$, $\gamma_i \in H^*_T(Z)$ to the $T$-fixed loci $X=X\times \{\infty\}$, $Z$, $\tX$ are given as follows: 
\begin{align}
\label{eq:restrictions_of_f}
\begin{split}  
f_X &:= f|_X  = \alpha|_{X\times \{\infty\}},\\
f_Z &:= f|_Z  = \alpha|_{Z\times \{0\}} - \sum_{k=0}^{r-1} (-\lambda)^{k+1} \gamma_k,\\
f_\tX &:= f|_\tX  = 
\varphi^*(\alpha|_{X\times \{0\}}) + \sum_{k=0}^{r-1} \jmath_*(p^k \pi^*\gamma_k).  
\end{split} 
\end{align} 
Here we use the fact that $[\hD]|_Z = \lambda$, $\hp|_Z = -\lambda$, $\hp|_D = p$ and that the intersection $\hD \cap \tX = D$ is transversal. It follows that the restriction maps $H^*_T(W) \to H^*_T(X)$, $H^*_T(W) \to H^*_T(\tX)$ are surjective (but $H^*_T(W)\to H^*_T(Z)$ is not necessarily so). 
\begin{remark} 
The presentation of $f_\tX$ compatible with the first line of \eqref{eq:cohomology_bu} is given as follows: 
\[
f_\tX=  
\varphi^*(\alpha|_{X\times \{0\}} + \imath_*\gamma_{r-1}) + \sum_{k=0}^{r-2} \jmath_*(p^k \pi^*(\gamma_k - c_{r-1-k}(\cN_{Z/X}) \gamma_{r-1}))
\]
where $\imath\colon Z\to X$ is the inclusion. 
This follows from the computation of $\varphi_*g$, $\jmath^*g$ for $g=\jmath_*(p^{r-1} \pi^*\gamma_{r-1})$ using $\pi_*(p^{r-1}) = 1$ and the relation $p^r + p^{r-1} \pi^* c_1(\cN_{Z/X}) + \cdots + \pi^*c_{r}(\cN_{Z/X})=0$ in $H^*(D)$.  
\end{remark} 

\begin{lemma} 
\label{lem:GKZ_type} 
For a subset $A$ of $W$, we write $i_A \colon A \to W$ for the inclusion map. For $f\in H^*_T(W)$, we have the following:
\begin{itemize} 
\item[(1)] $f_X|_Z - f_Z$ is divisible by $\lambda$ in $H^*(Z)[\lambda]$.   
\item[(2)] $f_\tX|_D- \pi^* f_Z$ is divisible by $\lambda+p$ in $H^*_T(D)$, or equivalently, when we regard $g(p,\lambda)= f_\tX|_D - \pi^*f_Z$ as a polynomial in $p$ and $\lambda$ with coefficients in $H^*(Z)$, $g(-\lambda,\lambda)$ is divisible by $e_{-\lambda}(\cN_{Z/X})$ in $H^*(Z)[\lambda]$, where $e_\lambda$ is the equivariant Euler class \eqref{eq:equiv_Euler}. 
\item[(3)] Suppose $f_X=0$. Then $f_Z$ is divisible by $\lambda$ by part (1) and we have 
\[
f=\hjmath_* (\hpi^*f_Z/\lambda) + i_{\tX*}g
\] 
for some $g \in H^*_T(\tX)$. In particular, if $f_X=f_Z=0$, then $f_\tX$ is divisible by $\lambda - [D]$. 

\item[(4)] Suppose $f_\tX = 0$. Then $f_Z$ is divisible by $e_{-\lambda}(\cN_{Z/X})$ by part $(2)$ and we have 
\[
f= i_{Z\times \PP^1*} (\pr_1^*f_Z/e_{-\lambda}(\cN_{Z/X})) + i_{X*} g
\] 
for some $g \in H^*_T(X)$. In particular, if $f_\tX = f_Z= 0$, then $f_X$ is divisible by $\lambda$. 
\end{itemize} 
\end{lemma} 
\begin{proof} 
Parts (1) and (2) follow from the formulae \eqref{eq:restrictions_of_f} together with the fact that $\alpha|_{Z\times \{0\}} - \alpha|_{Z\times\{\infty\}}$ is divisible by $\lambda$ for $\alpha \in H^*_T(X\times \PP^1)$ and that $\jmath^*\jmath_*\gamma = -p \gamma$ for $\gamma \in H^*_T(D)$. 


Suppose $f_X=0$ and set $f' = f - \hjmath_*(\hpi^*f_Z/\lambda) \in H^*_T(W)$. 
Then $f'_X = f'_Z = 0$. 
For any $\alpha \in H^*(\tX)$, there exists a lift $\halpha\in H^*_T(W)$ such that $\halpha|_\tX = \alpha$. 
Then the $T$-equivariant integral of $\halpha \cup f'$ is given by 
\[
\int_W^T \halpha \cup f' = \int_\tX \alpha \cup \frac{f'_\tX}{\lambda -[D]} 
\] 
by the localization formula. Since this lies in $\C[\lambda]$ for every $\alpha \in H^*(\tX)$, it follows that $f'_\tX$ is divisible by $\lambda -[D]$ in $H^*(Z)[\lambda]$. 
Part (3) follows by setting $g = f'_\tX/(\lambda -[D]) \in H^*(Z)[\lambda]$. 
Part (4) follows by a similar argument. 
\end{proof}

\subsection{Kirwan map}
\label{subsec:Kirwan} 
The Kirwan map $\kappa_Y \colon H^*_{T}(W) \to H^*(Y)$ for $Y = X$ or $\tX$ is defined to be the composition: 
\[
\kappa_Y \colon H^*_T(W) \to H^*_T(W^\st_Y) \cong H^*(Y), 
\]
where the first map is the restriction to the stable locus $W^\st_Y\subset W$ (see \eqref{eq:stable_loci}). For $\alpha \in H^*_{T}(W)$, we obtain the class $\kappa_X(\alpha)$ by restricting $\alpha$ to $X=X\times \{\infty\}$ and setting $\lambda=0$; we obtain the class $\kappa_{\tX}(\alpha)$ by restricting $\alpha$ to $\tX$ and setting $\lambda = [D]$ (recall that the normal bundle of $\tX$ is $\cO_\tX(-D)$). The Kirwan map is known to be surjective in general \cite{Kirwan:coh_quotient}; in the case at hand this fact follows from the surjectivity of $i_{X}^*$, $i_{\tX}^*$. 
The Kirwan map on $H^2_T(W)$ is given by 
\begin{align*} 
\kappa_X\colon \hvarphi^*\pr_1^*\alpha &\longmapsto \alpha 
& \kappa_\tX \colon \hvarphi^*\pr_1^*\alpha & \longmapsto \varphi^*\alpha \\ 
[X] &\longmapsto 0 
& [X] & \longmapsto  0 \\ 
-[\hD] & \longmapsto 0 
& -[\hD] & \longmapsto -[D] 
\\ 
\lambda & \longmapsto 0 
& \lambda & \longmapsto [D]
\end{align*} 
for $\alpha \in H^2(X)$. The Kirwan map $\kappa_Y$ maps the cone $C_Y\subset N_T^1(W)_\R$ to the ample cone $C(Y)\subset N^1(Y)_\R$ of $Y$. The map $\kappa_X \colon C_X \to C(X)$ is surjective but $\kappa_\tX \colon C_\tX \to C(\tX)$ is not necessarily so (since $\varphi^*\omega- \epsilon [D]\in N^1(\tX)_\R$ being ample does not imply $\omega\in N^1(X)_\R$ being ample in general). 

The dual Kirwan maps $\kappa_Y^* \colon N_1(Y) \to N_1^T(W)$ of $\kappa_Y \colon N^1_T(W) \to N^1(Y)$ for curve classes are given as follows: 
\begin{align*} 
\kappa_X^*\colon d &\longmapsto (d,0,0,0) & \kappa_\tX^* \colon \td & \longmapsto (\varphi_* \td, 0,-[D]\cdot \td, [D] \cdot \td). 
\end{align*} 
We have that $\kappa_Y^*(\NEN(Y)) \subset C_{Y,\N}^\vee$, see \eqref{eq:dualmonoids}. We identify the Novikov ring $\C[\![\NEN(Y)]\!]$ of $Y$ with a subring of $\C[\![C_{Y,\N}^\vee]\!]$ via $\kappa_Y^*$. Note that this identification is compatible with the grading since $\kappa_Y(c_1^T(W)) = c_1(Y)$. If we write $\tQ$ for the Novikov variable of $\tX$, then $\tQ^\td\in \C[\![\NEN(\tX)]\!]$ is identified with $Q^{\varphi_*\td} (y^{-1} S)^{[D]\cdot \td}=\QW^{i_* \td} S^{[D]\cdot \td}\in \C[\![C_{\tX,\N}^\vee]\!]$ where $i=i_\tX\colon \tX \to W$ is the inclusion. 

\begin{lemma} 
\label{lem:Kirwan_kernel} 
The kernel of the Kirwan map $\kappa_\tX \colon H^*_T(W) \to H^*(\tX)$ equals $i_{\tX*} H^*_T(\tX) + i_{X*}H^*_T(X) + i_{Z\times \PP^1*} (\pr_1^*H^*_T(Z))$. 
\end{lemma} 
\begin{proof} 
It is clear that $\Ker(\kappa_\tX)$ contains  $i_{\tX*} H^*_T(\tX) + i_{X*}H^*_T(X) + i_{Z\times \PP^1*} (\pr_1^*H^*_T(Z))$ since $\tX \cup (Z\times \PP^1) \cup X$ is the unstable locus for the GIT quotient $\tX$. We show the converse. Take $f \in \Ker(\kappa_\tX)$. Then $i_\tX^*f$ lies in the ideal in $H^*_T(\tX)$ generated by $\lambda - [D]$ and can be written in the form $i_\tX^*f = (\lambda -[D]) g$ for some $g\in H^*_T(\tX)$. 
Then $i_\tX^*(f - i_{\tX*}g) = 0$. By Lemma \ref{lem:GKZ_type}(4), $f -i_{\tX*} g$ lies in  $i_{X*}H^*_T(X) + i_{Z\times \PP^1*} (\pr_1^*H^*_T(Z))$ and the conclusion follows. 
\end{proof}

\subsection{Equivariant quantum $D$-module of $W$ with shift operators} 
\label{subsec:QDM(W)} 
We fix notation for the equivariant quantum $D$-module and shift operators for $W$. 

Recall the generators $Q^d,x,y,S$ of the group ring $\C[N_1^T(W)]$ from \S\ref{subsec:N1_W}. The parameters $\QW=(Q,x,y)$ give the Novikov variable of $W$ and $S$ corresponds to the shift operator. 
We also write $\hS =(\QW,S)$. 
As before, for a (graded) module $K$, we write $K[\![\QW]\!]$ for the (graded) completion of $K[\QW]:=K[\NEN(W)]$.  
We write $\theta$ for a general point of $H^*_T(W)$ and $\btheta=\{\theta^{i,k}\}$ for the infinite set of variables for $\theta$ dual to a $\C$-basis $\{\phi_i\lambda^k\}$ of $H^*_T(W)$: they correspond respectively to $\tau$ and $\btau$ in the notation of \S\ref{subsec:equiv_qconn}. The equivariant quantum $D$-module of $W$ is the module 
\[
\QDM_T(W) := H^*_T(W) [z][\![\QW,\btheta]\!] 
\]
equipped with the equivariant quantum connection $\nabla$ and the pairing $P_W$ (given by the equivariant Poincar\'e pairing), see \S\ref{subsec:equiv_qconn}. It is also equipped with the action of shift operators $\hbS^\beta(\theta)$ for $\beta\in N_1^T(W)$, see Definitions \ref{def:shift}. Using the presentation \eqref{eq:2nd_homology}, we set 
\[
\bS(\theta) := \hbS^{(0,0,0,1)}(\theta). 
\]
The operator for a general $\beta = (d,k,l,m) \in N^T_1(W)$ is given by $\hbS^\beta(\theta) = Q^d x^k y^l \bS(\theta)^m$ by Remark \ref{rem:shift}(3). We also write $\cS = \hcS^{(0,0,0,1)}$ for the corresponding operator on the rational Givental space $\cH^{\rm rat}_W=H^*_{\hT}(W)_{\rm loc}[\![\QW]\!]$ (see Definition \ref{def:shift_Givental}). 

The section class $\sigma_F(k)\in N_1^{\rm sec}(E_k(W)) \subset N_1^T(W)$ (see Definition \ref{def:shift_Givental}) associated with a fixed component $F\subset W^T$ and $k\in \Hom(\C^\times,T) = \Z$ is given as follows: 
\begin{equation} 
\label{eq:sigma_F} 
\sigma_X(k) = (0,-k,0,k), \quad  \sigma_Z(k) = (0,0,-k,k), \quad \sigma_{\tX}(k) = (0,0,0,k). 
\end{equation} 
The maximal fixed component $F_{\rm max}(k)$ equals $X$ for $k>0$, $W$ for $k=0$ and $\tX$ for $k<0$. Thus the maximal section class $\sigma_{\rm max}(k)=\sigma_{F_{\rm max}(k)}(k)$ equals $(0,\min(-k,0),0,k)$ and the shift operator defines a map  
$\bS(\theta)^k \colon \QDM_T(W) \to x^{\min(0,k)} \QDM_T(W)$ by Definition \ref{def:shift}. 
We can also see that $\NEN^T(W)$ given in \eqref{eq:dualmonoids} coincides with the one in Proposition \ref{prop:module_over_extended_shift}. 
Proposition \ref{prop:module_over_extended_shift} immediately implies:

\begin{proposition}
\label{prop:equiv_QDM_module_over} 
The equivariant quantum $D$-module $\QDM_T(W)$ has the structure of a $\C[z][\![\NEN^T(W)]\!]$-module, where the $\C[z]$-algebra $\C[z][\![\NEN^T(W)]\!]$ is topologically generated by the Novikov variables $\QW^\delta$ with $\delta \in \NEN(W)$, $\bS(\theta)$ and $x \bS(\theta)^{-1}$ by \eqref{eq:dualmonoids}. 
\end{proposition} 
%

Unpacking Definition \ref{def:shift_Givental}, we find that the shift operators $\cS^k \colon \cH_W^{\rm rat} \to x^{\min(0,k)} \cH_W^{\rm rat}$ on the Givental space are given as follows (see also \eqref{eq:sigma_F}): 
\begin{align}
\label{eq:shiftop_W} 
\begin{split}  
(\cS^k \bbf)_X & = x^k \frac{\prod_{c=-\infty}^0 -\lambda + cz }{\prod_{c=-\infty}^k -\lambda + cz}
e^{-k z\partial_\lambda} \bbf_X \\ 
(\cS^k \bbf)_Z & = y^k \frac{\prod_{c=-\infty}^0 e_{-\lambda+cz}(\cN_{Z/X})}{\prod_{c=-\infty}^k e_{-\lambda+cz}(\cN_{Z/X})}\frac{\prod_{c=-\infty}^0 \lambda + cz}{\prod_{c=-\infty}^{-k} \lambda+ cz} 
e^{-kz \partial_\lambda} \bbf_Z \\ 
(\cS^k \bbf)_\tX & = \frac{\prod_{c=-\infty}^0 -[D]+ \lambda + cz}{\prod_{c=-\infty}^{-k} -[D]+\lambda+cz} e^{-kz\partial_\lambda} \bbf_{\tX} 
\end{split} 
\end{align} 
where $\bbf_F$ denotes the restriction of $\bbf\in H^*_\hT(W)_{\rm loc}$ to a fixed component $F$ and $e_\lambda(\cN_{Z/X})$ is the equivariant Euler class \eqref{eq:equiv_Euler} of the normal bundle $\cN_{Z/X}$ of $Z$ in $X$. 

\section{Fourier transformation for the equivariant $J$-function}
\label{sec:Fourier} 
In this section, we discuss the Fourier transformation for the equivariant $J$-function of $W=\Bl_{Z\times\{0\}}(X\times \PP^1)$, or equivalently, the equivariant Givental cone of $W$. We introduce \emph{discrete} and \emph{continuous} Fourier transformations and compare them. 
The main results (Corollaries \ref{cor:Fourier_transform_JW} and \ref{cor:Fourier_transform_JW_GIT}) of this section say that  certain (discrete or continuous) Fourier transforms of the equivariant $J$-function of $W$ lie in the Givental cone of the fixed loci or the GIT quotients. 
As in the previous section, $T$ denotes the rank 1 torus $\C^\times$ acting on $W$. 

\subsection{Discrete Fourier transformation associated with a GIT chamber} 
\label{subsec:discrete_Fourier} 
Let $Y$ denote $X$ or $\tX$. We define the extended Givental space of $Y$ as 
\[
\cH_Y^{\rm ext} := H^*(Y)[z,z^{-1}][\![C_{Y,\N}^\vee]\!]. 
\] 
Recall from \S\ref{subsec:Kirwan} that $\C[\![C_{Y,\N}^\vee]\!]$ contains the Novikov ring $\C[\![\NEN(Y)]\!]$ via the dual Kirwan map $\kappa_Y^* \colon \NEN(Y) \to C_{Y,\N}^\vee$; hence this is a base change of the original Givental space $\cH_Y$ in \eqref{eq:Givental_space}. The discrete Fourier transformation $\sfF_Y$ is defined to be the following map: 
\[
\sfF_Y \colon 
\cH^{\rm rat}_W[\QW^{-1}] \dasharrow \cH_Y^{\rm ext}[\QW^{-1}], \qquad 
\sfF_Y(\bbf) := \sum_{k\in \Z} S^k \kappa_Y(\cS^{-k}\bbf)  
\]
where $K[\QW^{-1}]$ denotes the localization\footnote{We note that $C_{Y,\N}^\vee + (-\NEN(W))=N_1^T(W)$ by \eqref{eq:dualmonoids}.} of a $\C[\![\QW]\!]$-module $K$ by $\{\QW^\delta: \delta\in \NEN(W)\}$ and $\cS$ is given in \eqref{eq:shiftop_W}. 
The dashed arrow indicates that the map $\sfF_Y$ is not defined everywhere: $\sfF_Y(\bbf)$ is well-defined if $\cS^{-k} \bbf|_Y$ is regular at $\lambda=0$ for all $k\in \Z$ (so that $\kappa_Y(\cS^{-k}\bbf)$ is well-defined), and if the power series $\sum_{k\in \Z} S^k \kappa_Y(\cS^{-k} \bbf)$ in $\hS =(\QW,S)$ is supported on $\delta+ C_{Y,\N}^\vee$ for some $\delta \in N_1(W)$. 
Here we say that a rational function $f(\lambda,z)$ of $\lambda$ and $z$ is \emph{regular at} $\lambda=0$ if it is written in the form $p(\lambda,z)/q(\lambda,z)$ with $p,q\in \C[\lambda,z]$ such that $q(0,z)$ is a non-zero polynomial. For example, $1/(\lambda+z)$ is regular at $\lambda=0$ but $1/\lambda$ is not. 

Using the commutation relation $[\lambda,\cS] = z \cS$, we can easily check that $\sfF_Y$ satisfies
\begin{align}
\label{eq:properties_discrete_Fourier}
\begin{split} 
\sfF_Y( \cS^l \bbf) & = S^l \sfF_Y(\bbf),  \qquad  l \in \Z, \\
\sfF_Y( \lambda \bbf) & = (z S \partial_S + \kappa_Y(\lambda)) \sfF_Y(\bbf)  
\end{split} 
\end{align} 
whenever $\sfF_Y(\bbf)$ is well-defined. The second equation can be generalized for $\xi \in H^2_T(W)$ as follows: 
\begin{equation} 
\label{eq:properties_discrete_Fourier_Qdiff} 
\sfF_Y( (z\xi \QW \partial_{\QW} + \xi) \bbf ) = (z \xi \hS \partial_{\hS} + \kappa_Y(\xi)) \sfF_Y(\bbf) 
\end{equation} 
where $\xi \hS \partial_{\hS}$ denotes the derivation of $\C[\![C_{Y,\N}^\vee]\!]$ given by $(\xi \hS \partial_{\hS}) \hS^\beta = (\xi \cdot \beta) \hS^\beta$ for $\beta \in C_{Y,\N}^\vee \subset N_1^T(W)$. The fact that $\sfF_Y$ preserves the degree implies that $\sfF_Y \circ (z\partial_z + c_1^T(W) \QW\partial_{\QW} + \mu_W+ \frac{1}{2}) = (z\partial_z + c_1^T(W) \hS \partial_\hS + \mu_Y)\circ \sfF_Y$. This combined with \eqref{eq:properties_discrete_Fourier_Qdiff} implies that 
\begin{equation} 
\label{eq:properties_discrete_Fourier_zdiff} 
\sfF_Y( (z\partial_z - z^{-1} c_1^T(W) + \mu_W + \tfrac{1}{2}) \bbf ) = 
(z \partial_z - z^{-1} c_1(Y) + \mu_Y) \sfF_Y(\bbf) 
\end{equation} 
whenever $\sfF_Y(\bbf)$ is well-defined. 

\begin{proposition} 
\label{prop:support_conjecture} 
If $\bbf\in\cH_W^{\rm rat}$ lies in a tangent space of the equivariant Givental cone $\cL_W$ of $W$, then $\sfF_Y(\bbf)$ is a well-defined element of $\cH_Y^{\rm ext}$. 
\end{proposition} 
\begin{proof} 
Let $T$ be a tangent space of $\cL_W$. Recall from \S\ref{subsec:Givental_cone} that $T$ is of the form $M_W(\theta) \cH_+$ and is invariant under shift operators up to localizing $\QW$. We also note that $\bbf_Y=\bbf|_Y$ is regular at $\lambda=0$ for $\bbf \in T$; this follows from the fact that the expansion at $z=\infty$ of $M_W(\theta)$ is regular at $\lambda=0$ (see \eqref{eq:fundsol_at_infinity}). Hence $\kappa_Y(\cS^{-k} \bbf)$ is well-defined for $\bbf \in T$ for all $k$. 

It remains to show that $\sfF_Y(\bbf)=\sum_{k\in \Z} S^k \kappa_Y(\cS^{-k} \bbf)$ with $\bbf \in T$ is supported on $C_{Y,\N}^\vee$ as power series in $\hS$. 
We first consider the case where $Y= \tX$. Since $\C[\![C_{\tX,\N}^\vee]\!] = \C[\![\QW,S, y S^{-1}]\!]$ by \eqref{eq:dualmonoids}, it suffices to show that $(\cS^k \bbf)_\tX$ lies in $y^k H^*_\hT(\tX)_{\rm loc}[\![\QW]\!]$ for $k\ge 0$ and in $H^*_\hT(\tX)_{\rm loc}[\![\QW]\!]$ for $k\le 0$. In view of \eqref{eq:shiftop_W}, this is trivial for $k\le 0$ but requires proof for $k>0$. We expand $\bbf$ in power series of $\QW$ as $\bbf = \sum_{\delta\in \NEN(W)} \bbf_\delta \QW^\delta$. 
We need to show that the denominator of $e^{-kz\partial_\lambda} \bbf_{\delta,\tX}$ (as a rational function of $\lambda$ and $z$) with $k>0$ does not contain the factor $\lambda-[D]$ appearing in the numerator of 
\[
\frac{\prod_{c=-\infty}^0 -[D]+ \lambda + cz}{\prod_{c=-\infty}^{-k} -[D]+\lambda+cz}
\]
unless $\QW^\delta \in y^k \C[\![\QW]\!]$; if it does not, we have $\kappa_\tX(\cS^k \bbf_\delta) =0$ by the third line of \eqref{eq:shiftop_W} and $\kappa_\tX(\lambda -[D]) = 0$. 
Suppose that $\delta \notin (0,0,k,0)+\NEN(W)$ with $k>0$. It suffices to show that $e^{-kz \partial_\lambda} \bbf_{\delta,\tX}$ does not have poles along $\lambda=0$, or equivalently, $\bbf_{\delta,\tX}$ does not have poles along $\lambda + kz=0$. By the virtual localization formula \cite{Kontsevich:enumeration, Graber-Pandharipande}, poles of $\bbf_{\delta,\tX}$ along $\lambda + kz =0$ can arise only from the localization contribution to a correlator of the form 
\[
\corr{\phi_i, \theta,\dots,\theta, \frac{\phi_j}{z-\psi}}_{0,n+2,\delta'} \quad \text{for some $\delta'$ with $\delta \in \delta' + \NEN(W)$} 
\]
of a $T$-fixed locus in the moduli space where the last marked point maps to $\tX$ and lies in a component that is a $k$-fold covering of a $T$-invariant rational curve in $W$ (as the restriction of $\psi_{n+2}$ to such a fixed stable map equals $-\lambda/k$). 
If the $T$-invariant curve is of class $\delta_0$, we have $\delta' \in k\delta_0 + \NEN(W)$.  Thus if $\bbf_{\delta,\tX}$ has poles along $\lambda+ kz=0$, we have $\delta \in k \delta_0+\NEN(W)$ for some class $\delta_0$ of an irreducible $T$-invariant curve meeting $\tX$. Since $\delta_0$ can be either $(0,0,1,0)$ or $(0,1,0,0)\in (0,0,1,0) +\NEN(W)$ (see the proof of Lemma \ref{lem:NEN_W}), we have $\delta\in (0,0,k,0) + \NEN(W)$, contradicting the assumption.  

The discussion in the case where $Y=X$ is similar. We need to show that $(\cS^{-k}\bbf)_X$ lies in $y^{-k}H^*_\hT(X)_{\rm loc}[\![\QW]\!]$ for $k>0$. By the first line of \eqref{eq:shiftop_W}, it suffices to show that the denominator of $e^{k z\partial_\lambda} \bbf_{\delta,X}$ (with $k>0$) does not contain the factor $\lambda$ if $\delta \notin (0,k,-k,0) + \NEN(W)$. This again follows from the virtual localization formula together with the fact that the class of a $T$-invariant irreducible curve meeting $X$ belongs to $(0,1,-1,0)+\NEN(W)$.  
\end{proof}

\subsection{Continuous Fourier transformation associated with a fixed component} 
\label{subsec:continuous_Fourier} 
In this section, we introduce continuous Fourier transformations $\scrF_{F,j}$ associated with a $T$-fixed component $F$ of $W$. The results in this section make sense and hold for a general smooth projective variety $W$ equipped with a rank-one torus action. (Some of them also hold for higher-rank torus actions.) 

\subsubsection{Restriction to a fixed component and the inverse-Euler twisted Givental cone} 

\begin{proposition}[\cite{Brown:toric_fibration, Fan-Lee:projective_bundle}] 
\label{prop:restriction_twisted_cone}
Let $F$ be a $T$-fixed component of $W$ and let $\bbf$ be a point on the equivariant Givental cone $\cL_W$ of $W$. Consider the restriction $\bbf_F$ of $\bbf$ to $F$. The Laurent expansion at $z=0$ of $\bbf_F$ lies in the Givental cone of $F$ twisted by the normal bundle $\cN_{F/W}$ of $F$ and the inverse equivariant Euler class $e_T^{-1}$. 
\end{proposition} 

\begin{proof} 
This follows from the argument in Brown \cite[Theorem 2(i)]{Brown:toric_fibration} for toric bundles, where it was attributed to Givental.  The general case has been proved by Fan-Lee \cite[Theorem 3.5(1)]{Fan-Lee:projective_bundle}. 
We outline the argument for the convenience of the reader. Let $\bbf \in \cL_W$ be a point of the form \eqref{eq:point_on_the_cone} (with $X$ there replaced with $W$). The restriction of $\bbf$ to $F$ is 
\[
\bbf_F = z + \bt_F(z) + \sum_i \sum_{n,\delta} \phi_F^i \corr{\bt(-\psi),\dots,\bt(-\psi),\frac{i_{F*} \phi_{F,i}}{z-\psi}}_{0,n+1,\delta}^{W,T} \frac{\QW^\delta}{n!} 
\]
where $\bt_F(z) = \bt(z)|_F$, $\{\phi_{F,i}\}$ is a basis of $H^*(F)$, $\{\phi_F^i\}$ is the dual basis such that $(\phi_{F,i},\phi_F^j)_F = \delta_i^j$ and $i_F\colon F\to W$ is the inclusion map. By the virtual localization \cite{Kontsevich:enumeration,Graber-Pandharipande}, we can express $\bbf_F$ as the sum of contributions from $T$-fixed components of $W_{0,n+1,d}$. We only need to look at $T$-fixed components where the last marking maps to $F$. 
We classify such $T$-fixed components into the following two types: 
\begin{description} 
\item[type-A] the $\psi$-class $\psi_{n+1}$ at the last marking carries a nontrivial $T$-weight; this happens when the last marking lies in a component that covers a $T$-invariant curve in $W$;  

\item[type-B] the $\psi$-class $\psi_{n+1}$ at the last marking carries the trivial $T$-weight; this happens when the last marking lies in a component whose image is contained in $F$. 
\end{description}  
We single out the contributions from type-A fixed loci and set 
\begin{equation} 
\label{eq:restriction_input} 
\tau_F(z) := \bt_F(z) + \sum_i \sum_{n,\delta} \phi_F^i \corr{\bt(-\psi),\dots,\bt(-\psi),\frac{i_{F*} \phi_{F,i}}{z-\psi}}_{0,n+1,\delta}^{\textrm{(A)}} \frac{\QW^\delta}{n!} 
\end{equation} 
where the superscript (A) means the sum of contributions from type-A fixed loci to the correlator. Since $\psi_{n+1}$ is invertible on type-A fixed loci, $\tau_F(z)$ is regular at $z=0$ and its Taylor expansion at $z=0$ yields an element of $\cH_+^{\rm tw} \subset \cH^{\rm tw}_F$. On the other hand, the type-B contributions yield polynomials in $z^{-1}$ because $\psi_{n+1}$ is nilpotent there. 
Using a recursive structure of the fixed loci, 
we find that the type-B contributions can be expressed as the $(\cN_{F/W},e_T^{-1})$-twisted Gromov-Witten invariants of $F$ with insertions $\tau_F(-\psi)$:
\begin{multline*} 
\bbf_F = z + \tau_F(z) \\
+ \sum_i \sum_{n\ge 0, d\in \NEN(F)} e_T(\cN_{F/W}) \phi_F^i 
\corr{\tau_F(-\psi),\dots,\tau_F(-\psi), \frac{\phi_{F,i}}{z-\psi}}_{0,n+1,d}^{F,(\cN_{F/W},e_T^{-1})} \frac{\QW^{i_{F*}d}}{n!},
\end{multline*} 
arriving at the conclusion. Note that, by replacing $z$ with $-\psi_i$ in the type-A contributions of $\tau_F(z)$, we get the factor $1/(-\psi_i-\psi)$ that corresponds to smoothing nodes. 
\end{proof} 

\begin{remark} 
The point $\bbf$ and its restriction $\bbf_F$ are defined over the Novikov ring $\C[\![\QW]\!]$ of $W$. We regard $\bbf_F$ as a point over the Novikov ring $\C[\![Q_F]\!]$ of $F$ via the map $Q_F^d \mapsto \QW^{i_{F*}d}$ where $i_F\colon F\hookrightarrow W$ is the inclusion. Recall from Remark \ref{rem:e_te_twist} that the degree of $Q_F^d$ in the $(\cN_{F/W},e_T^{-1})$-twisted theory is $2 (c_1(F) + c_1(N_{F/W}))\cdot d$, which equals $\deg(\QW^{i_{F*}d})$. 
\end{remark} 

\begin{remark} 
The localization of Gromov-Witten invariants with respect to a torus action whose fixed points and 1-dimensional orbits are isolated is well-known. Here we work with a more general torus action (with not necessarily isolated fixed points or 1-dimensional orbits), for which we refer the reader to \cite{Mustata-Mustata:GW_Cstar, Fan-Lee:projective_bundle}. 
\end{remark} 

\subsubsection{$\Gamma$-function solutions and Fourier (Mellin-Barnes) integrals}
\label{subsubsec:MB_integrals}
We introduce an operator $G_F$ for a fixed component $F$ as follows: 
\begin{equation} 
\label{eq:G_F} 
G_F = \prod_{\varrho} \frac{1}{\sqrt{-2\pi z}} (-z)^{-\varrho/z}\Gamma\left(-\frac{\varrho}{z}\right) 
\end{equation} 
where the product is taken over the \emph{$T$-equivariant} Chern roots $\varrho$ of the normal bundle $\cN_{F/W}$ such that $c_T(\cN_{F/W}) = \prod_\varrho (1+\varrho)$. We may regard $G_F$ as an $\End(H^*(F))$-valued analytic function of $\lambda$ and $z$ (where cohomology classes act by the cup product). Concretely, we have 
\begin{align*}
G_X & = \frac{1}{\sqrt{-2\pi z}} (-z)^{\lambda/z} \Gamma\left(\frac{\lambda}{z}\right),\\ 
G_Z & = \frac{1}{(\sqrt{-2\pi z})^{r+1}} (-z)^{-\lambda/z} \Gamma\left(-\frac{\lambda}{z}\right) 
\prod_{\epsilon}(-z)^{(\lambda-\epsilon)/z} \Gamma\left(\frac{\lambda-\epsilon}{z} \right),\\ 
G_\tX & = \frac{1}{\sqrt{-2\pi z}} (-z)^{([D]-\lambda)/z} 
\Gamma\left(\frac{[D]-\lambda}{z} \right),  
\end{align*} 
where the product in the second line is taken over the Chern roots $\epsilon$ of $\cN_{Z/X}$. 
These operators give ``solutions'' of the shift operator: we can easily see from Definition \ref{def:shift_Givental} that 
\begin{equation} 
\label{eq:Gamma_solution}
G_F (\hcS^\beta \bbf)_F = \QW^{\beta+ \sigma_F(-\ovbeta)} e^{-z\ovbeta \partial_\lambda} 
(G_F \bbf_F)
\end{equation} 
for $\bbf \in \cH_W^{\rm rat}$ and $\beta \in N_1^T(W)$. In other words, the map $\bbf \mapsto G_F \bbf_F$ intertwines the shift operator $\hcS^\beta$ with the trivial one $\QW^{\beta+ \sigma_F(-\ovbeta)} e^{-z\ovbeta \partial_\lambda}$.  
We are thus led to consider the continuous Fourier transformation 
\begin{equation} 
\label{eq:Fourier_integral} 
\cH^{\rm rat}_W \ni \bbf \longmapsto \int e^{\lambda \log S_F/z} G_F \bbf_F d\lambda 
\end{equation} 
with $S_F := \hS^{\sigma_F(1)}$; we have $S_X = Sx^{-1}$, $S_Z =Sy^{-1}$, $S_\tX =S$ by \eqref{eq:sigma_F}. This should intertwine $\hcS^\beta$ with $\hS^\beta$ and $\lambda$ with $z S_F\parfrac{}{S_F} = z S\parfrac{}{S}$. The integral \eqref{eq:Fourier_integral} is also known as a \emph{Mellin-Barnes integral} and appears in toric mirror symmetry (see e.g.~\cite[\S 8]{Iritani:toric}). Below, we compute the asymptotics as $z\to 0$ using the stationary phase (or saddle-point) method. 

\begin{notation} 
\label{nota:normal_bundle} 
Let $\cN_{F/W} = \bigoplus_\alpha \cN_{\alpha}$ be the $T$-weight decomposition where $T$ acts on $\cN_{\alpha}$ by the character $\alpha \in \Hom(T,\C^\times)$. We identify the weight $\alpha$ with an element of $H^2_T(\pt,\Z) \cong \Z \lambda$ and set $\alpha = w_\alpha \lambda$ for $w_\alpha \in \Z$. We also set $\rho_\alpha := c_1(\cN_\alpha)$, $r_\alpha := \rank(\cN_\alpha)$, $c_F := \sum_\alpha r_\alpha w_\alpha$, $\rho_F := c_1(\cN_{F/W}) = \sum_\alpha \rho_\alpha$, $r_F := \rank(\cN_{F/W}) = \sum_\alpha r_\alpha$. 
We also write $\Delta_\alpha := \Delta_{(\cN_\alpha, e_\alpha^{-1})}$, $\tDelta_\alpha := \Delta_{(\cN_\alpha, \te_\alpha^{-1})}$ for the quantum Riemann-Roch operators \eqref{eq:QRR_operator} associated with the $e_\alpha^{-1}$- and $\te_\alpha^{-1}$-twists. Note that $\tDelta_\alpha$ lies in $\End(H^*(F))[\lambda^{-1}](\!(z)\!)$ and $\Delta_\alpha = \alpha^{\rho_\alpha/z+r_\alpha/2} \tDelta_\alpha$.  
\end{notation}

\begin{remark} 
\label{rem:Stirling} 
The asymptotic expansion of $G_F$ gives rise to a product of the quantum Riemann-Roch operators $\Delta_\alpha$. 
The Stirling approximation gives 
\begin{align*} 
\log G_F \sim & 
\sum_{\varrho} \left( -
\frac{\varrho \log \varrho - \varrho}{z} - \frac{1}{2} \log \varrho - \sum_{n=2}^\infty \frac{B_n}{n(n-1)} \left(\frac{z}{\varrho}\right)^{n-1} \right) \\ 
& = 
\sum_\alpha \left(-r_\alpha \frac{\alpha \log \alpha -\alpha}{z} - \log \Delta_\alpha \right) 
\end{align*} 
as $z\to 0$ along an appropriate angular region, where the first sum is over $T$-equivariant Chern roots $\varrho$ of $\cN_{F/W}$ and the second is over $T$-weights $\alpha$ of $\cN_{F/W}$. If $z$ approaches zero from the \emph{negative} real axis, the above asymptotics hold when $\arg(\alpha) \in (-\pi, \pi)$ for all the $T$-weights $\alpha$ appearing in $\cN_{F/W}$. 

Suppose that $\bbf$ lies in the equivariant Givental cone $\cL_W$. Then $\bbf_F$ lies in the $e_T^{-1}$-twisted Givental cone of $F$ by Proposition \ref{prop:restriction_twisted_cone} and hence  $(\prod_\alpha \Delta_\alpha^{-1}) \bbf_F$ lies in the untwisted cone by Theorem \ref{thm:QRR} (see also Remark \ref{rem:more_than_one_twists}). Thus we expect that the formal asymptotics of \eqref{eq:Fourier_integral} should lie in the Givental cone of $F$: we will elaborate on this in \S\ref{subsubsec:Fourier_integral_on_the_cone} below. 
\end{remark} 

\subsubsection{Formal asymptotic expansion of the Fourier integral}
\label{subsubsec:formal_asymptotics} 
We now define the formal asymptotic expansion of  the integral \eqref{eq:Fourier_integral} and discuss its properties. The discussion here parallels that in \cite[\S 5.3]{Iritani-Koto:projective_bundle}. It suffices to consider the case where $\bbf$ is independent of $\QW$, i.e.~lies in $H^*_\hT(W)_{\rm loc}$; we then extend the definition linearly over $\C[\![\QW]\!]$.   Replacing $G_F$ with the asymptotic expansion in Remark \ref{rem:Stirling}, we rewrite the integral \eqref{eq:Fourier_integral} as 
\begin{equation} 
\label{eq:formal_Fourier_integral} 
\int e^{\eta(\lambda)/z} \left( \prod_\alpha 
\alpha^{-\rho_\alpha/z-r_\alpha/2} \tDelta_\alpha^{-1}\right)  \bbf_F d\lambda 
\end{equation} 
where $\eta(\lambda) := \lambda \log S_F - \sum_\alpha r_\alpha (\alpha \log \alpha -\alpha)$ is the phase function. 
We expand $\bbf_F \in H^*_\hT(F)_{\rm loc} = H^*(F) \otimes \C(\lambda,z)_{\rm hom}$ in Laurent series at $z=0$ and regard it as an element of $H^*(F)[\lambda,\lambda^{-1}](\!(z)\!)$. Note that $H^*(F)[\lambda,\lambda^{-1}](\!(z)\!) = H^*(F)[z,z^{-1}](\!(\lambda^{-1})\!)$ because we are working with graded completions (see \S\ref{subsec:formal_power_series}). 
We take a critical point  $\lambda_0$ of $\eta(\lambda)$ given by  
\[
\lambda_0 := (S_F)^{1/c_F} \prod_\alpha w_\alpha^{-r_\alpha w_\alpha/c_F} 
\]
and consider the stationary phase approximation of the integral \eqref{eq:formal_Fourier_integral} associated with $\lambda_0$. 
Here we are assuming $c_F\neq 0$. (In the case at hand, we have $c_X=-1$, $c_Z=-(r-1)$, $c_\tX=1$.) 
The stationary phase method approximates the integral by contribution near the critical point $\lambda_0$; the resulting asymptotics as $z\to 0$ depends only on the formal germ of the integrand at $\lambda = \lambda_0$. 
We set $\lambda = \lambda_0 \exp(u/\sqrt{c_F\lambda_0})$ and expand\footnote{The choice of the coordinate $u$ is not important; we can use any other local coordinates near $\lambda_0$ normalizing the Hessian of $\eta(\lambda)$ at $\lambda_0$.} the integrand in power series of $u$. We write  
\begin{align*} 
&\eta(\lambda) = c_F \lambda_0  -  
\frac{u^2}{2}  - g(u,\lambda_0)  \quad \text{with} \quad 
g(u,\lambda_0) := \sum_{n=3}^\infty \frac{n-1}{n!} \frac{u^n}{(c_F \lambda_0)^{n/2-1}} \\
\text{and} \ &\left( \prod_\alpha \alpha^{-\rho_\alpha/z-r_\alpha/2}\tDelta_\alpha^{-1}\right)  \bbf_F 
= \left( \prod_\alpha (w_\alpha \lambda_0)^{-\rho_\alpha/z-r_\alpha/2} \right) \Phi(u) 
\end{align*} 
with 
\[
\Phi(u) := e^{-\frac{u}{\sqrt{c_F\lambda_0}}(\rho_F/z + r_F/2)} \left(\prod_\alpha \tDelta_\alpha^{-1}\right) \bbf_F
\]
in $H^*(F)\otimes \C[u,z,z^{-1}](\!(\lambda_0^{-1/2})\!)$ (where we set $\deg z = \deg \lambda_0 =2$, $\deg u =1$). 
Then the above integral \eqref{eq:formal_Fourier_integral} can be written as a  Gaussian integral 
\begin{equation*} 
e^{c_F\lambda_0/z} \sqrt{\frac{\lambda_0}{c_F}}\left( \prod_\alpha (w_\alpha \lambda_0)^{-\rho_\alpha/z-r_\alpha/2} \right) 
\int e^{-u^2/(2z)} e^{-g(u,\lambda_0)/z}e^{u/\sqrt{c_F\lambda_0}} \Phi(u) du. 
\end{equation*}  
We perform term-by-term integration of the $u$-power series $e^{-g(u,\lambda_0)/z}e^{u/\sqrt{c_F\lambda_0}}  \Phi(u)$ over the contour $\sqrt{z} \R$ and arrive at the formal asymptotic expansion as $z\to 0$: 
\begin{equation} 
\label{eq:stationary_phase_approximation} 
\int e^{\lambda \log S_F/z} G_F \bbf_F d\lambda 
\sim \sqrt{2\pi z} \cdot  e^{c_F \lambda_0/z} \scrF_{F,0}(\bbf) 
\end{equation} 
with 
\[
\scrF_{F,0}(\bbf) := \sqrt{\frac{\lambda_0}{c_F}} \left( \prod_\alpha (w_\alpha \lambda_0)^{-\rho_\alpha/z-r_\alpha/2} \right) 
\left[ e^{z(\partial_u)^2/2} e^{-g(u,\lambda_0)/z}e^{u/\sqrt{c_F \lambda_0}} \Phi(u) \right]_{u=0}. 
\]
The half-integer powers of $\lambda_0$ in $e^{-g(u,\lambda_0)/z}e^{u/\sqrt{c_F \lambda_0}}\Phi(u)$ must be accompanied by odd powers of $u$ and therefore do not contribute to the Gaussian integral. Thus we have 
\[
\scrF_{F,0}(\bbf) \in S_F^{-\rho_F/(c_Fz)-(r_F-1)/(2c_F)}  H^*(F)[z,z^{-1}](\!(S_F^{-1/c_F})\!) 
\]
for $\bbf \in H^*_\hT(W)_{\rm loc}$. Note that $\deg S_F=2 c_F$. By applying the monodromy transformation with respect to $S_F$, we get $\scrF_{F,j}(\bbf) = \scrF_{F,0}(\bbf)|_{S_F \to e^{2\pi\iu j} S_F}$ with $j=0,\dots,|c_F|-1$. This is the formal asymptotics associated with the critical point $\lambda_j := e^{2\pi \iu j/c_F} \lambda_0$ of $\eta(\lambda)$. Extending linearly over $\C[\![\QW]\!]$, we obtain the continuous Fourier transformations 
\[
\scrF_{F,j} \colon \cH_W^{\rm rat} \to S_F^{-\rho_F/(c_Fz)-(r_F-1)/(2c_F)}  H^*(F)[z,z^{-1}](\!(S_F^{-1/c_F})\!)[\![\QW]\!] 
\]
for $j=0,\dots,|c_F|-1$. 

When $\bbf_F= \lambda^n \phi + O(\lambda^{n-1})$ with $\phi\in H^*(F)$, the leading order term of $\scrF_{F,j}(\bbf)$ as a Laurent series of $S_F^{-1/c_F}$ is given by   
\begin{equation} 
\label{eq:leadingterm_scrF} 
\scrF_{F,j}(\bbf) = q_{F,j} e^{h_{F,j}/z} S_F^{-\rho_F/(c_Fz)} \lambda_j^n \left(\phi + O(S_F^{-1/c_F})\right)
\end{equation} 
where $\lambda_j = e^{2\pi \iu j/c_F} (\prod_\alpha w_\alpha^{-r_\alpha w_\alpha/c_F}) S_F^{1/c_F}$ as above and we set 
\begin{align}
\label{eq:qFj-hFj} 
\begin{split} 
q_{F,j}&:= \sqrt{\frac{\lambda_j}{c_F}} \prod_\alpha (w_\alpha \lambda_j)^{-r_\alpha/2}\in \C S_F^{-(r_F-1)/(2c_F)}, \\
h_{F,j} &:= - \frac{2\pi \iu j}{c_F} \rho_F + \sum_\alpha \left(\frac{r_\alpha w_\alpha}{c_F} \rho_F - \rho_\alpha\right)\log w_\alpha\in H^2(F,\C). 
\end{split} 
\end{align} 
\begin{proposition} 
\label{prop:properties_conti_Fourier}
The map $\bbf \mapsto \scrF_{F,j}(\bbf)$ satisfies the following: 
\begin{itemize} 
\item[(1)] $\scrF_{F,j}(\hcS^\beta \bbf) = \hS^\beta \scrF_{F,j}(\bbf)$; 
\item[(2)] $\scrF_{F,j}(\lambda \bbf) = (z S \parfrac{}{S} + \lambda_j) \scrF_{F,j}(\bbf)$; 
\item[(3)] $\scrF_{F,j}((z\partial_z - z^{-1} c_1^T(W) + \mu_W+\frac{1}{2})\bbf) = (z\partial_z-z^{-1} ( c_1(F)+c_F \lambda_j) + \mu_F) \scrF_{F,j}(\bbf)$. 
\end{itemize} 
\end{proposition} 
\begin{proof} 
Similar results appear in \cite[Proposition 5.5]{Iritani-Koto:projective_bundle}. 
Parts (1), (2) are translations of standard properties of the Fourier transformation in terms of formal asymptotics (together with \eqref{eq:Gamma_solution}). To see part (3), observe that $\bbf \mapsto \lambda_j^{c_1(\cN_{F/W})/z} \scrF_{F,j}(\bbf)$ is homogeneous of degree $-(r_F-1)$. This implies 
\[
\scrF_{F,j}((z\partial_z + \mu_W + \tfrac{1}{2}) \bbf ) = (c_F S\partial_S + z\partial_z + z^{-1} c_1(\cN_{F/W})+ \mu_F )
 \scrF_{F,j}(\bbf). 
\]
Part (3) follows from this, part (2) and $c_1^T(W)|_F= c_1(F) + c_1(\cN_{F/W}) + c_F \lambda$. 
\end{proof} 

\subsubsection{The Fourier transforms of points on the Givental cone} 
\label{subsubsec:Fourier_integral_on_the_cone} 
We start with the Fourier transform of the big $J$-function $\cJ_F^{\rm tw}(\bt)$ \eqref{eq:point_on_the_twisted_cone} on the $(\cN_{F/W},e_T^{-1})$-twisted theory.  
Using a basis $\{\phi_{F,i}\}$ of $H^*(F)$, we expand the parameter as $\bt(z) = \sum_{n=0}^\infty \sum_i t_n^i \phi_{F,i} z^n$, $t_n^i = \sum_{k\in \Z} t_n^{i,k} \lambda^k$ and identify $\bt$ with the infinite set $\{t_n^{i,k}\}_{n\in \N, k\in\Z, i}$ of variables. We set $\deg t_n^{i,k} = 2-\deg \phi_i - 2n-2k$. 
We regard $\cJ_F^{\rm tw}(\bt)$ as an $R[\![Q_F, \bt]\!]$-valued point, where $R=\C[\lambda,\lambda^{-1}]$ and $Q_F$ is the Novikov variable of $F$. Let 
\[
\int e^{\lambda \log S_F/z} G_F \cJ_F^{\rm tw}(\bt) d\lambda \sim \sqrt{2\pi z}\cdot  e^{c_F\lambda_j/z} \hcJ_{F,j}^{\rm tw}(\bt) 
\]
be the formal asymptotics associated with the critical point $\lambda_j =e^{2\pi \iu j/c_F} S_F^{1/c_F} \prod_\alpha w_\alpha^{-r_\alpha w_\alpha/c_F}$ defined in the previous section \S\ref{subsubsec:formal_asymptotics}. 

\begin{proposition} 
\label{prop:Fourier_bigJtw}
With notation as above $($see also Notation $\ref{nota:normal_bundle}$$)$, there exist 
\begin{align*} 
\tau &=\tau(\bt) \in h_{F,j} + \frakm' H^*(F)[\![ S_F^{-1/c_F},Q_F S_F^{-\rho_F/c_F}, \bt S_F^{\bullet/c_F}]\!] 
\quad \text{ and } \\
v & = v(\bt) \in q_{F,j} \left(1+ \frakm'' H^*(F)[z][\![S_F^{-1/c_F},Q_FS_F^{-\rho_F/c_F}, \bt S_F^{\bullet/c_F}]\!] \right) 
\end{align*} 
such that $\hcJ^{\rm tw}_{F,j}(\bt) = S_F^{-\rho_F/(c_F z)} z M_F(\tau; Q_F S_F^{-\rho_F/c_F}) v$ where $q_{F,j}, h_{F,j}$ are as in \eqref{eq:qFj-hFj} and 
\begin{itemize} 
\item $M_F(\tau) = M_F(\tau;Q_F)$ is the fundamental solution \eqref{eq:fundsol} for $F$; 
\item $\bt S_F^{\bullet/c_F}$ denotes the infinite set $\{t^{i,k}_n S_F^{k/c_F}\}_{n\in \N,k\in \Z, i}$ of variables; 
\item $\frakm'\subset \C[\![S_F^{-1/c_F}, Q_F S_F^{-\rho_F/c_F}, \bt S_F^{\bullet/c_F}]\!]$ and $\frakm''\subset \C[z][\![S_F^{-1/c_F},Q_FS_F^{-\rho_F/c_F},\bt S_F^{\bullet/c_F}]\!]$ are the closed ideals generated by $S_F^{-1/c_F}, Q_F^d S_F^{-\rho_F\cdot d/c_F}, t^{i,k}_n S_F^{k/c_F}$ with $d \in\NEN(F)\setminus \{0\}$.  
\end{itemize} 
Moreover, $\tau$ is homogeneous of degree two and $v/q_{F,j}$ is homogeneous of degree zero, where the degree of $Q_F$ is given by $\deg Q_F^d = 2(c_1(F) + \rho_F) \cdot d$ for $d\in \NEN(F)$ $($as in the $(\cN_{F/W},e_T^{-1})$-twisted theory, see Remark $\ref{rem:e_te_twist}$$)$.  
\end{proposition} 
\begin{proof} 
Recall the characteristic classes $e_\lambda, \te_\lambda$ in \eqref{eq:equiv_Euler}. For a vector bundle $V$ with a fibrewise $T$-action, we have $e_T(V) = \prod_\alpha e_\alpha(V_\alpha)$ where $V=\bigoplus_\alpha V_\alpha$ is the $T$-eigenbundle decomposition where $T$ acts on $V_\alpha$ by the weight $\alpha \in \Hom(T,\C^\times) = H^2_T(\pt,\Z) \cong \Z \lambda$. Let $\te_T$ be a modified equivariant Euler class given by $\te_T(V) = \prod_\alpha \te_\alpha(V_\alpha)$. 
Let $\tcL_F^{\rm tw}$ be the Givental cone of the $(\cN_{F/W},\te_T^{-1})$-twisted theory and let $\tcJ_F^{\rm tw}(\bt)$ be the big $J$-function on it.  As discussed in Remark \ref{rem:e_te_twist}, we have 
\[
\tcJ_F^{\rm tw}(\bt) =\cJ_F^{\rm tw}(\bt)\bigr|_{Q_F \to Q_F \prod_\alpha \alpha^{\rho_\alpha}}.
\] 
This is a $\C[\lambda^{-1}][\![Q_F, \bt \lambda^\bullet]\!]$-valued point of $\tcL_F^{\rm tw}$, where $\bt\lambda^\bullet$ denotes the set $\{t_n^{i,k} \lambda^k\}$ of variables. Hence by Theorem \ref{thm:QRR} (see also Remark \ref{rem:more_than_one_twists}), we have 
$(\prod_\alpha \tDelta_\alpha^{-1}) \tcJ_F^{\rm tw}(\bt)$ lies in the cone $\cL_F$. 
The Divisor Equation says that the cone $\cL_F$ is invariant under $\bbf(Q_F) \mapsto e^{s h/z} \bbf(Q_F e^{sh})$ for $h\in H^2(F)$; the String Equation says that $\cL_F$ is invariant under $\bbf \mapsto e^{s/z} \bbf$, where $s$ is a formal parameter. Substituting $\lambda_j e^{u/\sqrt{c_F \lambda_j}}$ for $\lambda$ and using the Divisor and String Equations (and the fact that $\cL_F$ is a cone), we get a point 
\begin{equation} 
\label{eq:after_Divisor_String} 
e^{-g(u,\lambda_j)/z}e^{s} \left( \prod_{\alpha} e^{-s(\rho_\alpha/z+r_\alpha/2)} \tDelta_\alpha^{-1}\right) \tcJ_F^{\rm tw}(\bt)\Bigr|_{\substack{Q_F \to Q_F e^{-s \cdot \rho_F} \\ \lambda=\lambda_j e^s \phantom{aaaaaaa}}} 
\quad \text{with $s=u/\sqrt{c_F \lambda_j}$} 
\end{equation} 
lying in $\cL_F$. This is a $\C[\![\lambda_j^{-1/2},u,Q_F,\bt \lambda_j^{\bullet}]\!]$-valued point. 
By the special geometric properties of $\cL_F$ explained in \S\ref{subsec:Givental_cone}, the operator  $e^{z (\partial_u)^2/2} = e^{(z\partial_u)^2/(2z)}$ preserves points on the cone (see \cite[\S 8]{Coates-Givental}, \cite[\S 4.2]{CCIT:computing}, \cite[Lemma 2.7]{Iritani-Koto:projective_bundle}). 
Hence we have 
\[
e^{z(\partial_u)^2/2} (\text{the quantity in \eqref{eq:after_Divisor_String}})\bigr|_{u=0} 
=z M_F(\ttau) \tv
\]
for some $\ttau=\ttau(Q_F), \tv = \tv(Q_F) \in H^*(F)[z][\![\lambda_j^{-1}, Q_F,\bt \lambda_j^\bullet]\!]$ such that $\ttau$ is independent of $z$ and that $\ttau$ and $\tv-1$ vanish at the origin $\lambda_j^{-1} = Q_F = \bt\lambda_j^\bullet=0$. 
By the definition of the formal asymptotics $\hcJ^{\rm tw}_{F,j}(\bt)$ in \S\ref{subsubsec:formal_asymptotics}, we have 
\[
\hcJ^{\rm tw}_{F,j}(\bt) = q_{F,j} \left( \prod_\alpha (w_\alpha \lambda_j)^{-\rho_\alpha/z} \right) 
z M_F(\ttau) \tv\Bigr|_{Q_F \to Q_F \prod_\alpha (w_\alpha\lambda_j)^{-\rho_\alpha}}. 
\] 
Using again the Divisor Equation $e^{h/z}M_F(\tau; Q_Fe^h) = M_F(\tau+h;Q_F)$ for the fundamental solution $M_F(\tau) = M_F(\tau;Q_F)$, we arrive at the conclusion by setting $\tau = h_{F,j} + \ttau(Q_F \textstyle\prod_\alpha (w_\alpha \lambda_j)^{-\rho_\alpha})$ and $v = q_{F,j} \cdot \tv(Q_F \textstyle\prod_\alpha (w_\alpha \lambda_j)^{-\rho_\alpha})$.  

The big $J$-function $\cJ^{\rm tw}_X(\bt)$ is homogeneous of degree two for the grading $\deg Q_F^d = 2 (c_1(F)+ \rho_F) \cdot d$. Keeping track of the degrees, we find that $M_F(\ttau) \tv$ is homogeneous of degree zero with respect to the grading $\deg Q_F^d = 2 c_1(F) \cdot d$ and thus $\ttau$ and $\tv$ are homogeneous of degree two and zero respectively. Hence $\tau$ and $v/q_{F,j}$ are homogeneous of degree two and zero respectively for the grading $\deg Q_F^d = 2 (c_1(F) + \rho_F)\cdot d$. 
\end{proof} 

Let $J_W(\theta)$ with $\theta \in H^*_T(W)$ denote the equivariant $J$-function of $W$ (see \S\ref{subsec:equiv_qconn} and   \eqref{eq:J-function}). Recall from \S\ref{subsec:Givental_cone} that $z J_W(\theta)$ is the restriction of the big $J$-function $\cJ_W$ to the ``small phase space'' $H^*_T(W)$ inside $\cH^W_+$. 

\begin{corollary} 
\label{cor:Fourier_transform_JW}
The Fourier transform of the equivariant $J$-function $J_W(\theta)$ of $W$ is of the form 
\[
\scrF_{F,j}(J_W(\theta)) = S_F^{-\rho_F/(c_Fz)} M_F(\tau; Q_F S_F^{-\rho_F/c_F}) v \bigr|_{Q_F\to \QW} 
\]
for some $\tau \in H^*(F)[S_F^{1/c_F},S_F^{-1/c_F}][\![\QW,\btheta]\!]$ and $v\in q_{F,j} H^*(F)[z](\!(S_F^{-1/c_F})\!)[\![\QW,\btheta]\!]$ satisfying 
\begin{align}
\label{eq:condition_at_Q=0}
\begin{split}  
\tau|_{\QW=0} &\in h_{F,j} + \frakm' H^*(F)[\![S_F^{-1/c_F},\btheta_F S_F^{\bullet/c_F}]\!] \\ 
v|_{\QW=0} & \in q_{F,j} \left(1+\frakm'' H^*(F)[z][\![S_F^{-1/c_F},\btheta_F S_F^{\bullet/c_F} ]\!]\right)
\end{split} 
\end{align} 
where the subscript $Q_F\to \QW$ means to replace $Q_F^d$ with $\QW^{i_{F*}d}$ for $d\in \NEN(F)$ and 
\begin{itemize} 
\item $\btheta = \{\theta^{i,k}\}_{k\in\N, i}$ is the parameter for $\theta =\sum_{i,k} \theta^{i,k} \phi_{i} \lambda^k \in H^*_T(W)$ associated with a basis $\{\phi_i\}$ of $H^*_T(W)$ over $H_T^*(\pt)$ $($as in {\rm \S\ref{subsec:QDM(W)}}$)$; 
\item $\btheta_F=\{\theta_F^{i,k}\}_{k\in\N,i}$ is the parameter for the restriction $\theta|_F = \sum_{i,k} \theta_F^{i,k} \phi_{F,i}\lambda^k$; each $\theta_F^{i,k}$ can be written as a $\C$-linear combination of $\{\theta^{j,l}\}_{0\le l\le k, j}$; 
\item $\btheta_FS_F^{\bullet/c_F} $ is the infinite set $\{\theta^{i,k}_F S_F^{k/c_F} \}_{k\in \N, i}$ of variables; 
\item $\frakm'\subset \C[\![S_F^{-1/c_F},\btheta_F S_F^{\bullet/c_F} ]\!]$ and $\frakm'' \subset \C[z][\![S_F^{-1/c_F},\btheta_F S_F^{\bullet/c_F}]\!]$ are the closed ideals generated by $S_F^{-1/c_F}$ and $\theta_F^{i,k}S_F^{k/c_F}$. 
\end{itemize} 
Note that the condition \eqref{eq:condition_at_Q=0} ensures that $M_F(\tau; Q_F S_F^{-\rho_F/c_F})v$ is well-defined. 
Moreover, $\tau$ and $v/q_{F,j}$ are homogeneous of degree two and zero respectively. 
\end{corollary} 
\begin{proof} 
Set $\bbf = z J_W(\theta)$ and consider the restriction $\bbf_F = \bbf|_F$ to $F$.  
By Proposition \ref{prop:restriction_twisted_cone}, the Laurent expansion of $\bbf_F$ at $z=0$ lies in the $(\cN_{F/W},e_T^{-1})$-twisted cone of $F$ and is therefore of the form $\cJ^{\rm tw}_F(\bt)$ with $\bt(z)$ being the non-negative part $[\bbf_F-z]_+$ of the Laurent expansion of $\bbf_F-z$ at $z=0$ (also given in \eqref{eq:restriction_input}). Hence we have 
\[
\scrF_{F,j}(z J_W(\theta)) = \hcJ^{\rm tw}_{F,j}(\bt)\bigr|_{Q_F\to \QW, \bt(z)= [\bbf_F-z]_+}.  
\]
The right-hand side equals $S_F^{-\rho_F/(c_Fz)} z M_F(\tau; Q_F S_F^{-\rho_F/c_F})v|_{Q_F\to \QW,\bt(z)=[\bbf_F-z]_+}$ for  $\tau=\tau(\bt)$, $v= v(\bt)$ in Proposition \ref{prop:Fourier_bigJtw}. 
It remains to show that the substitution $Q_F\to \QW$, $\bt(z) = [\bbf_F-z]_+$ in $\tau(\bt)$, $v(\bt)$ gives rise to well-defined elements $\tau, v$ in $H^*(F)[z](\!(S_F^{-1/c_F})\!)[\![\QW,\btheta]\!]$ satisfying \eqref{eq:condition_at_Q=0}. 
Let $a^{i,k}_n \in \C[\![\QW,\btheta]\!]$ be the coefficient of $[\bbf_F-z]_+$ in front of $\lambda^k \phi_{F,i} z^n$ so that the substitution $\bt(z) = [\bbf_F-z]_+$ is equivalent to $t^{i,k}_n=a^{i,k}_n$ for all $i,k,n$. 
Note that we have $a_n^{i,k}|_{\QW = 0} = \delta_{n,0} \theta_F^{i,k}$ as $z J_W(\theta)|_{\QW=0} =z e^{\theta/z} = z + \theta+O(z^{-1})$.  This implies \eqref{eq:condition_at_Q=0}. 

Consider the closed ideal $I_N\subset \C[\![\QW,\btheta]\!]$ generated by $\QW^\delta$ with $\omega_W\cdot \delta\ge N$, $\theta^{i,k}$ with $k\ge N$ and $(\theta^{i,l})^N$ for all $l\ge 0$, where $\omega_W$ is a fixed ample class. The ring $\C[\![\QW,\btheta]\!]$ is complete with respect to the topology given by $\{I_N\}$ (in the graded sense). Fix $N\in \N$. When $|\deg a_n^{i,k}| = |\deg t_n^{i,k}| =|2 - \deg \phi_i - 2 n- 2k|$ is sufficiently large, $a_n^{i,k}$ belongs to $I_N$. Therefore it suffices to show that, for a fixed $M\in \N$, the substitution $t_n^{i,k} = a_n^{i,k}$, $Q_F\to \QW$ in an element of 
\[
\C[z][\![S_F^{-1/c_F},  \bx, Q_FS_F^{-\rho_F/c_F}]\!] \quad \text{where $\bx = \left\{t_n^{i,k}S_F^{k/c_F}: n\in \N, k\in \Z, |n+k|\le M\right\}$} 
\]
gives an element of $\C[z](\!(S_F^{-1/c_F})\!)[\![\QW,\btheta]\!]$. We write the set $\bx$ of variables as the disjoint union of $\bx_- =\{t_n^{i,k} S_F^{k/c_F}\}_{k<0}$ and $\bx_+ = \{t_n^{i,k}S_F^{k/c_F}\}_{k\ge 0}$ and rewrite the above ring as 
\[
\C[z][\![S_F^{-1/c_F},\bx_-]\!][\![\bx_+,Q_FS_F^{-\rho_F/c_F}]\!] 
\]
Noting that $\bx_+$ is a finite set of variables, modulo the ideal $I_N$, only finitely many monomials in the variables $(\bx_+, Q_F S_F^{-\rho_F/c_F})$ contribute to the substitution. Hence it suffices to see that the substitution is well-defined for elements in $\C[z][\![S_F^{-1/c_F},\bx_-]\!]$. But this is obvious as the exponents of $S_F^{1/c_F}$ of the variables $\bx_-$ are negative and diverge to $-\infty$. 
\end{proof} 


\subsection{Continuous versus discrete Fourier transformations} 
\label{subsec:cont_vs_discrete} 
Let $Y$ be $X$ or $\tX$. Since $c_X = -1$ and $c_\tX=1$, we have exactly $|c_Y|=1$ Fourier transformation $\scrF_{Y,0}$ associated with $Y$ (see \S\ref{subsubsec:formal_asymptotics}). We show that the discrete and continuous Fourier transformations $\sfF_Y(\bbf)$, $\scrF_{Y,0}(\bbf)$ for $Y$ (defined respectively in \S\ref{subsec:discrete_Fourier} and \S\ref{subsec:continuous_Fourier}) agree up to a simple factor. It follows that the discrete Fourier transformation $\sfF_Y$ maps the Givental cone of $W$ to the Givental cone of $Y$. 
\begin{proposition} 
\label{prop:conti_vs_discrete}
Let $Y$ be $X$ or $\tX$. 
For any set $\bx=(x_0,x_1,x_2,\dots)$ of formal parameter and a $\C[\![\QW,\bx]\!]$-valued tangent vector $\bbf \in \cH_W^{\rm rat}[\![\bx]\!]$ of the equivariant Givental cone $\cL_W$ of $W$,  we have 
\begin{itemize} 
\item[(1)] $\scrF_{Y,0}(\bbf) \in S_Y^{-\rho_Y/(c_Yz)}H^*(Y)[z,z^{-1},S_Y,S_Y^{-1}][\![\QW,\bx]\!]$ and
\item[(2)] $\exp(S_Y^{1/c_Y}/z)\scrF_{Y,0}(\bbf)=c_Y^{-1} S_Y^{-\rho_Y/(c_Yz)} \sfF_Y(\bbf)$. 
\end{itemize}  
Here the product $\exp(S_Y^{1/c_Y}/z) \scrF_{Y,0}(\bbf)$ in part $(2)$ is well-defined by part $(1)$. 
\end{proposition} 
\begin{proof} 
We prove the proposition for $Y=\tX$. In this case $c_Y=1$, $S_Y=S$ and $\rho_Y= -[D]$. The discussion for $Y= X$ is similar and is left to the reader. Let $\bbf=\bbf(\lambda,z)$ be a tangent vector of $\cL_W$. Recall that $\sfF_\tX(\bbf)$ is well-defined by Proposition \ref{prop:support_conjecture}. As discussed there, by the virtual localization formula, $\bbf_\tX(\lambda,z)=\bbf|_{\tX}$ can have poles (as a function of $\lambda$) only at $\lambda = -kz$ with $k>0$. Moreover, since the shift operator preserves the tangent spaces of $\cL_W$, 
\[
(\cS^{k} \bbf)_\tX = \frac{\prod_{j\le 0} -[D]+\lambda + jz }{\prod_{j\le -k} -[D] + \lambda +jz} \bbf_\tX(\lambda-kz, z) 
\] 
is regular at $\lambda=0$ for all $k\in \Z$. This implies that, for $k>0$, $(-[D]+\lambda)\bbf_\tX(\lambda-kz,z)$ is regular at $\lambda=0$, or equivalently, $(-[D]+ \lambda +kz) \bbf_\tX(\lambda,z)$ is regular at $\lambda=-kz$. 
Therefore, the coefficient $f(\lambda,z)\in H^*(\tX)\otimes \C(\lambda,z)_{\rm hom}$ of a monomial  $\QW^\delta\bx^m$ in $\bbf_\tX$ has a partial fraction decomposition of the form: 
\begin{equation} 
\label{eq:pfd}
f(\lambda,z) = b(\lambda, z) + \sum_{i=1}^l \frac{a_i(z)}{-[D]+ \lambda+ k_i z} 
\end{equation} 
with $a_i(z) \in H^*(\tX)[z,z^{-1}]$, $b(\lambda,z) \in H^*(\tX)[z,z^{-1},\lambda]$ and $k_i\in \Z_{>0}$. 
Let $S>0$ and $z<0$ and consider the Mellin-Barnes integral 
\[
I_{f}(S,z) := 
\frac{1}{2\pi \iu (-z)}
\int_{S-\iu \infty}^{S+\iu \infty} (-z)^{([D]-\lambda)/z}
\Gamma\left(\frac{\lambda-[D]}{-z}\right)
S^{\lambda/z} f(\lambda,z) d\lambda. 
\]
We assume that $|z|$ is sufficiently small so that the poles $\lambda=-k_i z$, $i=1,\dots,l$ of $f(\lambda,z)$ are on the left side of the integration contour $\{\Re \lambda = S\}$  (see Figure \ref{fig:integration_contour}). The Stirling approximation shows that the integrand has exponential decay as $|\Im \lambda|\to \infty$ and this integral converges. 
The stationary phase method yields the asymptotic expansion\footnote{We choose $\sqrt{z} \in \R_{>0} \iu$ when defining $\scrF_{\tX,0}(\bbf)$ in \eqref{eq:stationary_phase_approximation}.} 
(cf.~\eqref{eq:stationary_phase_approximation}) 
\begin{equation} 
\label{eq:expansion_ISz} 
e^{-S/z} I_f(S,z) \sim \text{the coefficient of $\QW^\delta\bx^m$ in $\scrF_{\tX,0} (\bbf)$} \quad \text{as $z\nearrow 0$}.  
\end{equation} 
Recall here that the phase function $\eta(\lambda)=\lambda - \lambda \log \lambda + \lambda \log S$ arising from the Stirling approximation (see \S\ref{subsubsec:formal_asymptotics}) has the critical value $S$ at $\lambda = S$. A detailed proof of \eqref{eq:expansion_ISz} will be given in Appendix \ref{append:asymptotics}. We shall see below that this asymptotics is exact.

\begin{figure}[t]
\centering 
\begin{tikzpicture}[>=stealth, x=0.7cm, y=0.7cm]
\draw[->] (-5.5,0) -- (6,0); 
\draw[->] (0,-3) -- (0,3); 

\draw (5.8,0.4) node {\scriptsize $\Re \lambda$}; 
\draw (0.4,2.8) node {\scriptsize $\Im \lambda$}; 

\filldraw (-5,0) circle [radius =0.05]; 
\filldraw (-4,0) circle [radius =0.05]; 
\filldraw (-3,0) circle [radius =0.05]; 
\filldraw (-2,0) circle [radius =0.05]; 
\filldraw (-1,0) circle [radius =0.05]; 
\filldraw (0,0) circle [radius =0.05]; 

\draw (-0.2,-0.34) node {\scriptsize $0$}; 
\draw (-1,-0.34) node {\scriptsize $z$}; 
\draw (-2,-0.34) node {\scriptsize $2z$}; 
\draw (-3,-0.34) node {\scriptsize $3z$}; 
\draw (-4,-0.34) node {\scriptsize $4z$}; 
\draw (-5,-0.34) node {\scriptsize $5z$}; 

\draw (1,0) node {$\times$}; 
\draw (3,0) node {$\times$};
\draw (1,-0.4) node {\scriptsize $-k_1 z$}; 
\draw (3,-0.4) node {\scriptsize $-k_2 z$}; 

\draw[very thick] (4,-3) -- (4,3); 
\draw[->, very thick] (4,-3) -- (4,1); 
\draw (5,-2.5) node {\small $\Re \lambda = S$}; 
\end{tikzpicture} 
\caption{Integration contour on the $\lambda$-plane} 
\label{fig:integration_contour} 
\end{figure} 

The integral $I_f(S,z)$ can be also evaluated by residue computation. In the course of the computation, we regard $[D]$ as a small complex parameter; at the end we take the Taylor expansion in $[D]$ and substitute $[D]$ for the cohomology class. By closing the contour to the left and summing the contributions over poles at $\lambda = [D]+nz$ with $n\ge 0$ and $\lambda = [D] -k_i z$ with $i=1,\dots,l$, we arrive at the series: 
\begin{align} 
\label{eq:series_ISz} 
I_f(S,z) & = \sum_{n=0}^\infty \frac{S^{n+[D]/z}}{n!} f([D]+nz,z)  + 
\sum_{i=1}^l (k_i-1)! (-z)^{-(k_i-1)}  S^{-k_i + [D]/z} a_i(z) \\ 
\nonumber 
& = S^{[D]/z} \sum_{n\in \Z} \kappa_\tX\left( \frac{\prod_{j\le 0} -[D]+\lambda + jz }{\prod_{j\le n} -[D] + \lambda +jz} f(\lambda+nz, z) \right) S^n. 
\end{align} 
This is the coefficient of $\QW^\delta\bx^m$ in $S^{[D]/z}\sfF_\tX(\bbf)$. 

Finally we claim that $e^{-S/z} S^{-[D]/z} I_f(S,z)$ is a Laurent polynomial in $z$ and $S$. If this holds, the asymptotics \eqref{eq:expansion_ISz} is necessarily exact and we have the equality $S^{-[D]/z}\scrF_{\tX,0}(\bbf) =e^{-S/z}  \sfF_\tX(\bbf)$ in $H^*(\tX)[z,z^{-1},S,S^{-1}][\![\QW,\bx]\!]$. 
In view of the partial fraction decomposition \eqref{eq:pfd}, it suffices to show the claim for $f= \lambda^n$ with $n\ge 0$ and $f= (-[D]+\lambda + kz)^{-1}$ with $k>0$. The integral $I_1(S,z)$ corresponding to $f=1$ equals $S^{[D]/z} e^{S/z}$ by the series expansion \eqref{eq:series_ISz}. The integral $I_{\lambda^n}(S,z)$ equals $(z S\partial_S)^n I_1(S,z)$; it is $S^{[D]/z} e^{S/z}$ multiplied by a polynomial in $S$ and $z$. The integral $I_f(S,z)$ with $f=(-[D]+\lambda+kz)^{-1}$ satisfies the differential equation: 
\[
(z S \partial_S + kz) S^{-[D]/z} I_f(S,z) = S^{-[D]/z} I_1(S,z) = e^{S/z}.  
\]
A general solution to this differential equation is given by 
\[
S^{-[D]/z}I_f(S,z) =  e^{S/z} 
\left(\sum_{i=0}^{k-1} (-1)^i \frac{(k-1)!}{(k-1-i)!} S^{-i-1} z^i \right)+ g(z) S^{-k} 
\]
with $g(z)$ independent of $S$. Comparing the coefficient of $S^{-k}$ to the series \eqref{eq:series_ISz} reveals that $g(z) = 0$ and that $e^{-S/z} S^{-[D]/z} I _f(S,z)$ is a polynomial in $S^{-1}$ and $z$. 
\end{proof} 

\begin{corollary} 
\label{cor:Fourier_transform_JW_GIT} 
Let $Y$ be $X$ or $\tX$. 
The discrete Fourier transform $\sfF_Y(z J_W(\theta))$ of the equivariant $J$-function of $W$ lies in the Givental cone of $Y$. More precisely, there exist $\tau\in H^*(Y)[\![C^\vee_{Y,\N}]\!][\![\btheta]\!]$ and $v \in H^*(Y)[z][\![C^\vee_{Y,\N}]\!][\![\btheta]\!]$ such that 
\[
\sfF_Y(z J_W(\theta)) = z M_Y(\tau) v
\]
and that $\tau|_{\btheta=0}\equiv S_Y^{1/c_Y}$, $v|_{\btheta=0}\equiv 1$ modulo the closed ideal of $\C[z][\![C_{Y,\N}^\vee]\!]$ generated by $\QW^\delta$, $\delta \in \NEN(W)\setminus \{0\}$ and $(y S^{-1})^{c_Y}$. Here, as in $\S\ref{subsec:Kirwan}$ and $\S\ref{subsec:discrete_Fourier}$, we regard the Novikov ring $\C[\![\NEN(Y)]\!]$ of $Y$ as a subring of $\C[\![C_{Y,\N}^\vee]\!]$ via the dual Kirwan map $\kappa_Y^* \colon \NEN(Y) \to C_{Y,\N}^\vee$.  
\end{corollary} 
\begin{proof} 
By Corollary \ref{cor:Fourier_transform_JW}, there exist 
\[
\tau \in H^*(Y)[S_Y^{1/c_Y},S_Y^{-1/c_Y}][\![\QW,\btheta]\!] \quad \text{and} \quad 
v\in H^*(Y)[z](\!(S_Y^{-1/c_Y})\!)[\![\QW,\btheta]\!]
\]
satisfying the condition \eqref{eq:condition_at_Q=0} for $(F,j) = (Y,0)$ (note that $q_{Y,0} = c_Y^{-1}$ and $h_{Y,0} = 0$ here) such that 
\[
S^{\rho_Y/(c_Yz)} \scrF_{Y,0}(J_W(\theta)) = M_Y(\tau;Q_Y S_Y^{-\rho_Y/c_Y}) v|_{Q_Y \to \QW}. 
\]  
Note that $M_Y(\tau; Q_Y S_Y^{-\rho_Y/c_Y})|_{Q_Y \to \QW}$ is precisely the image of $M_Y(\tau)$ under the base change $\C[\![\NEN(Y)]\!] \to \C[\![C_{Y,\N}^\vee]\!]$ given by the dual Kirwan map, so we henceforth write it as $M_Y(\tau)$. Proposition \ref{prop:conti_vs_discrete} implies that we have 
\begin{equation} 
\label{eq:Fourier_JW} 
\exp(-S_Y^{1/c_Y}/z) \sfF_Y(J_W(\theta)) = M_Y(\tau) c_Y^{-1}v 
\end{equation} 
and this belongs to $H^*(Y)[z,z^{-1},S_Y,S_Y^{-1}][\![\QW,\btheta]\!]$. 
We claim that $\tau$ and $v$ are supported on the monoid $C_{Y,\N}^\vee$ as power series in $\hS=(\QW,S)$. 
Recall that $\C[\![C_{Y,\N}^\vee]\!] = \C[\![\QW, S_Y^{1/c_Y}, u_Y S_Y^{-1/c_Y}=(yS^{-1})^{c_Y}]\!]$ by \eqref{eq:dualmonoids}, where we set $u_\tX= y$ and $u_X=xy^{-1}$. 

First we study the equality \eqref{eq:Fourier_JW} at the limit\footnote{\label{footnote:meaning_of_QW=0}Here we consider the limit $\QW=\btheta=0$ as power series of $S_Y^\pm$, $\QW$ and $\btheta$; the meaning of the limit differs between $Y=X$ and $Y=\tX$.} $\QW=\btheta=0$. The left-hand side equals $\exp(-S_Y^{1/c_Y}/z) \sfF_Y(1) = 1$ at $\QW=\btheta=0$ by direct calculation. Since $M_Y(\tau)|_{\QW=0} = e^{\tau/z}$, this implies that $e^{-\tau/z}|_{\QW=\btheta=0} = v|_{\QW=\btheta=0}$. Thus we find  $\tau|_{\QW=\btheta=0}=0$, $v|_{\QW=\btheta=0}=1$ by comparing power series of $z$. 

We solve for $\tau$ from \eqref{eq:Fourier_JW} by the ``Newton method'' starting from $\tau|_{\QW=\btheta=0}=0$. Set $\exp(-S_Y^{1/c_Y}/z) \sfF_Y(J_W(\theta)) = 1 + g$ and $M_Y(\tau)^{-1} = \id - (\tau\cup)/z + z^{-1} A(\tau)$. We know that $g\in H^*(Y)[z,z^{-1},S_Y,S_Y^{-1}][\![\QW,\btheta]\!]$ vanishes at $\QW=\btheta=0$ and, by Proposition \ref{prop:support_conjecture}, $g$ is supported on $C_{Y,\N}^\vee$ as $\hS$-power series.  We also know that $A(\tau)\in \End(H^*(Y))[z^{-1}][\![\QW,\tau]\!]$ vanishes at $\QW=\tau=0$ and $A(\tau) 1 = O(1/z)$. Equation \eqref{eq:Fourier_JW} implies that the coefficient of $z^{-1}$ in $M_Y(\tau)^{-1}(1+g)=c_Y^{-1}v$ vanishes and hence 
\[
\tau = [g]_{-1} - \tau [g]_{0} + [A(\tau)g]_0
\]
where $[\cdots]_n$ denotes the coefficient of $z^n$. This can be solved recursively in powers of $(\QW,\btheta)$ and $\tau$ is the limit of the sequence $\{\tau_n\}_{n\ge 0}$ given by $\tau_0= 0$, $\tau_{n+1} = [g]_{-1} - \tau_n [g]_0 + [A(\tau_n) g]_0$. Since $A(\tau)$ and $g$ are supported on $C_{Y,\N}^\vee$ as $\hS$-series, it follows that $\tau$ is also supported on $C_{Y,\N}^\vee$ and so is $v = M_Y(\tau)^{-1} g$. The conclusion follows by $\exp(S_Y^{1/c_Y}/z) M_Y(\tau) = M_Y(\tau+S_Y^{1/c_Y})$. 
\end{proof} 


\section{Decomposition of quantum $D$-modules}
\label{sec:decomposition} 
In this section we prove the decomposition theorem for the quantum $D$-module of blowups. 
First we show that a certain completion of $\QDM_T(W)$ gives rise to $\QDM(\tX)$. Then we show that the Fourier transformations $\sfF_X$ and $\scrF_{Z,j}$, $j=0,\dots,r-2$ induce a decomposition of $\QDM(\tX)$ into $\QDM(X)$ and $r-1$ copies of $\QDM(Z)$. 

\subsection{Completion of the equivariant quantum $D$-module of $W$} 
\label{subsec:completion} 

Recall the equivariant quantum $D$-module $\QDM_T(W)$ from \S\ref{subsec:QDM(W)}. 
Let $\QDM_T(W)[\QW^{-1}]$ be the localization of $\QDM_T(W)$ with respect to the multiplicative set $\{\QW^\delta\}_{\delta \in \NEN(W)}$.  
Via the action of the extended shift operators (\S\ref{subsec:shift}), $\QDM_T(W)[\QW^{-1}]$ has the structure of a $\C[N_1^T(W)]$-module, where $\hS^\beta \in \C[N_1^T(W)]$ (with $\beta \in N_1^T(W)$) acts by $\hbS^\beta(\theta)$. 
We consider the subring $\C[C_{\tX,\N}^\vee]\subset \C[N_1^T(W)]$ generated by the monoid $C_{\tX,\N}^\vee$ in \eqref{eq:dualmonoids} and the  $\C[C_{\tX,\N}^\vee]$-submodule $\QDM_T(W)_\tX$ of $\QDM_T(W)[\QW^{-1}]$ generated by $\QDM_T(W)$: 
\begin{equation} 
\label{eq:QDMW_tX}
\QDM_T(W)_\tX := \C[C_{\tX,\N}^\vee] \cdot \QDM_T(W) \subset \QDM_T(W)[\QW^{-1}].  
\end{equation} 
Since $\C[C^\vee_{\tX,\N}] = \C[\QW][S, yS^{-1}]$ and $\bS(\theta)$ preserves $\QDM_T(W)$ (see Proposition \ref{prop:equiv_QDM_module_over}), we can also write this as 
\begin{equation} 
\label{eq:QDMW_tX_2nd} 
\QDM_T(W)_\tX = \C[y S^{-1}] \cdot \QDM_T(W). 
\end{equation} 
Recall here that $x S^{-1}$ preserves $\QDM_T(W)$ but $y S^{-1}$ may not. 
Then $\QDM_T(W)_\tX$ is a module over $\C[z][\![\QW,\btheta]\!][S,y S^{-1}]$. Let $\frakm\subset \C[z][\![\QW,\btheta]\!][S,yS^{-1}]$ be the ideal generated by $S$ and $y S^{-1}$. We consider the graded $\frakm$-adic completion of $\QDM_T(W)_\tX$: 
\begin{equation} 
\label{eq:QDMW_completed} 
\QDM_T(W)_\tX\sphat := \varprojlim_{k} \QDM_T(W)_\tX/\frakm^k \QDM_T(W)_\tX 
\end{equation} 
where the inverse limit is taken in the category of graded modules. The module $\QDM_T(W)_\tX\sphat$ naturally has the structure of a $\C[z][\![C_{\tX,\N}^\vee,\btheta]\!]$-module. 
Because of the commutation relation $\lambda \cdot \bS(\theta) = \bS(\theta) \cdot (\lambda +  z)$,  the action of $\lambda$ on $\QDM_T(W)$ induces an operator $\lambda$ on $\QDM_T(W)_\tX\sphat$. 
Also, because of the commutation relation \eqref{eq:shift_qconn}, the quantum connection on $\QDM_T(W)$ induces operators $z\nabla_{\theta^{i,k}}$, $z\nabla_{z\partial_z}$, $z\nabla_{\xi \QW\partial_\QW}$ (with $\xi \in H^2_T(W)$) on $\QDM_T(W)_\tX\sphat$. We shall see in Theorem \ref{thm:Fourier_isom} that $\QDM_T(W)_\tX\sphat$ is a finite free $\C[z][\![C_{\tX,\N}^\vee,\btheta]\!]$-module; hence these operators $\lambda$, $\nabla$ can be viewed as a flat connection on it. 

\subsection{Extension of the quantum $D$-modules of $X$ and $\tX$}
\label{subsec:extended_QDM} 
We introduce extensions of the Novikov rings for the quantum $D$-modules of $X$ and $\tX$. 

Let $\tau$ and $\ttau$ denote the parameters in $H^*(X)$ and $H^*(\tX)$ respectively. We choose a homogeneous basis $\{\phi_{X,i}\} \subset H^*(X)$, $\{\phi_{\tX,i}\} \subset H^*(\tX)$ and write  $\{\tau^i\}$, $\{\ttau^i\}$ for the dual coordinates on $H^*(X)$, $H^*(\tX)$ respectively. 
The quantum $D$-modules of $X$ and $\tX$ are the modules 
\[
\QDM(X) = H^*(X)[z][\![Q,\tau]\!], \qquad \QDM(\tX) = H^*(\tX)[z][\![\tQ,\ttau]\!] 
\]
equipped with the quantum connection $\nabla$ and the pairings $P_X$, $P_\tX$, see \S\ref{subsec:qcoh_qconn}.  
Recall the extensions $\C[\![Q]\!] \hookrightarrow  \C[\![C_{X,\N}^\vee]\!]$, $\C[\![\tQ]\!]\hookrightarrow \C[\![C_{\tX,\N}^\vee]\!]$ of the Novikov rings given by the dual Kirwan maps $\kappa_\tX^*, \kappa_X^*$ (see \S\ref{subsec:Kirwan}). 
We define 
\begin{align*} 
\QDM(X)^{\rm ext} & := \QDM(X) \otimes_{\C[z][\![Q,\tau]\!]} \C[z][\![C_{X,\N}^\vee, \tau]\!] 
= H^*(X)[z][\![C_{X,\N}^\vee,\tau]\!], 
\\  
\QDM(\tX)^{\rm ext} & := \QDM(\tX) \otimes_{\C[z][\![\tQ,\ttau]\!]}\C[z][\![C_{\tX,\N}^\vee,\ttau]\!]
= H^*(\tX)[z][\![C_{\tX,\N}^\vee,\ttau]\!].  
\end{align*} 
The quantum connection \eqref{eq:qconn} can be naturally extended to $\QDM(X)^{\rm ext}$, $\QDM(\tX)^{\rm ext}$. We define the operators $\nabla_{\tau^i}, \nabla_{z\partial_z}, \nabla_{\xi \hS\partial_{\hS}}\colon \QDM(X)^{\rm ext}\to z^{-1} \QDM(X)^{\rm ext}$ by 
\begin{align*} 
\nabla_{\tau^i} &= \nabla_{\tau^i} \otimes \id + \id \otimes \partial_{\tau^i} 
= \partial_{\tau^i} + z^{-1} (\phi_{X,i}\star_{\tau}),\\ 
\nabla_{z\partial_z} & = \nabla_{z\partial_z} \otimes \id + \id \otimes z\partial_z 
= z\partial_z - z^{-1} (E_{X}\star_{\tau}) + \mu_X, \\ 
\nabla_{\xi \hS \partial_{\hS}} & = \nabla_{\kappa_X(\xi) Q\partial_Q} \otimes \id + \id \otimes \xi\hS \partial_{\hS} 
= \xi \hS \partial_{\hS} + z^{-1} (\kappa_X(\xi)\star_\tau). 
\end{align*} 
We define the operators $\nabla_{\ttau^i}, \nabla_{z\partial_z}, \nabla_{\xi \hS \partial_{\hS}}\colon \QDM(\tX)^{\rm ext}\to z^{-1}\QDM(\tX)^{\rm ext}$ similarly. 
Here $\xi \in H^2_T(W)$ and $\xi \hS\partial_{\hS}$ denotes the derivation of $\C[\![C_{Y,\N}^\vee]\!]$ (with $Y$ being $X$ or $\tX$) given by $(\xi \hS \partial_{\hS}) \hS^\beta = (\xi \cdot \beta) \hS^\beta$ for $\beta \in C_{Y,\N}^\vee\subset N_1^T(W)$. 

\subsection{Fourier transformation for quantum $D$-modules} 
\label{subsec:FT_QDM} 
In this section we restate the fact that the discrete Fourier transform of the $J$-function $J_W(\theta)$ lies in the Givental cone of $X$ or $\tX$ (Corollary \ref{cor:Fourier_transform_JW_GIT}) in terms of the quantum $D$-modules. 

By Corollary \ref{cor:Fourier_transform_JW_GIT}, there exist $\tau=\tau(\theta) \in H^*(X)[\![C_{X,\N}^\vee,\btheta]\!]$, $\ttau=\ttau(\theta) \in H^*(\tX)[\![C_{\tX,\N}^\vee,\btheta]\!]$, $v = v(\theta)\in H^*(X)[z][\![C_{X,\N}^\vee,\btheta]\!]$ and $\tv=\tv(\theta)\in H^*(\tX)[z][\![C_{\tX,\N}^\vee, \btheta]\!]$ such that 
\begin{equation} 
\label{eq:initialcond_tau_v}
\begin{alignedat}{4}
\tau(0) &\equiv x S^{-1} \ & &\text{and} \quad  & v(0) & \equiv 1 \quad &&\bmod \overline{(\QW,y^{-1}S)} \\ 
\ttau(0) & \equiv S \ & & \text{and}\quad   & \tv(0) & \equiv 1 \quad&& \bmod \overline{(\QW, yS^{-1})}
\end{alignedat} 
\end{equation} 
and that 
\begin{equation} 
\label{eq:Fourier_JW_for_X_tX} 
\sfF_X(J_W(\theta)) = M_X(\tau(\theta)) v(\theta), \qquad  
\sfF_\tX(J_W(\theta)) = M_\tX(\ttau(\theta)) \tv(\theta) 
\end{equation}
where $\overline{(\QW,y^{-1}S)} \subset \C[z][\![C_{X,\N}^\vee]\!]$ is the closed ideal generated by $\{\QW^\delta\}_{\delta \in \NEN(W)\setminus \{0\}}$ and $y^{-1}S$; the ideal  $\overline{(\QW,yS^{-1})}\subset \C[z][\![C_{\tX,\N}^\vee]\!]$ is defined similarly.  
We consider the pullbacks $\tau^*\QDM(X)^{\rm ext}$, $\ttau^*\QDM(\tX)^{\rm ext}$ by $\tau=\tau(\theta), \ttau=\ttau(\theta)$; the pullback $\tau^*\QDM(X)^{\rm ext}$ is the module $H^*(X)[z][\![C_{X,\N}^\vee,\btheta]\!]$ equipped with the pulled-back quantum connection 
\begin{align}
\label{eq:pull-back_qconn}
\begin{split}  
\nabla_{\theta^{i,k}} &= \partial_{\theta^{i,k}} + z^{-1} (\partial_{\theta^{i,k}} \tau(\theta)) \star_{\tau(\theta)},  \\ 
\nabla_{z\partial_z} & = z \partial_z - z^{-1} (E_X \star_{\tau(\theta)}) + \mu_X, \\ 
\nabla_{\xi \hS\partial_\hS} & = \xi \hS \partial_\hS + z^{-1} (\kappa_X(\xi)\star_{\tau(\theta)}) + z^{-1} (\xi\hS \partial_\hS \tau(\theta))\star_{\tau(\theta)}  
\end{split} 
\end{align} 
and $\ttau^*\QDM(\tX)^{\rm ext}$ is defined similarly. 
We observe that the Fourier transformation $\sfF_Y$ in \S\ref{subsec:discrete_Fourier} induces a map from $\QDM_T(W)[\QW^{-1}]$ to $\tau^*\QDM(X)^{\rm ext}[\QW^{-1}]$ or to $\ttau^*\QDM(\tX)^{\rm ext}[\QW^{-1}]$. Here, as before, $K[\QW^{-1}]$ means the localization of a $\C[\![\QW]\!]$-module $K$ with respect to $\{\QW^\delta\}_{\delta \in \NEN(W)}$.  
\begin{proposition} 
\label{prop:FT_QDM} 
Let $Y$ be $X$ or $\tX$. Let $\ttt=\ttt(\theta)$ denote $\tau(\theta)$ for $Y=X$ and $\ttau(\theta)$ for $Y=\tX$ in \eqref{eq:Fourier_JW_for_X_tX}. 
We have a map
\[
\FT_Y \colon \QDM_T(W)[\QW^{-1}] \to \ttt^*\QDM(Y)^{\rm ext}[\QW^{-1}] 
\] 
of $\C[z][\![\QW,\btheta]\!][\QW^{-1}]$-modules such that $M_Y(\ttt(\theta))\circ \FT_Y = \sfF_Y \circ M_W(\theta)$, i.e.~the following diagram commutes:  
\begin{equation} 
\label{eq:FT_sfF} 
\begin{aligned} 
\xymatrix{
\QDM_T(W)[\QW^{-1}] \ar[r]^{\FT_Y} \ar[d]^{M_W(\theta)} 
&  \ttt^*\QDM(Y)^{\rm ext}[\QW^{-1}] \ar[d]^{M_Y(\ttt(\theta))} \\ 
\cH_W^{\rm rat}[\QW^{-1}] \ar@{.>}[r]^{\sfF_Y} & \cH_Y^{\rm ext}[\QW^{-1}]  
}
\end{aligned}
\end{equation} 
where note that the composition $\sfF_Y \circ M_W(\theta)$ is well-defined by Proposition $\ref{prop:support_conjecture}$.  Moreover, we have the following: 
\begin{itemize} 
\item[(1)] $\FT_Y(\QDM_T(W)) \subset \ttt^*\QDM(Y)^{\rm ext}$; 
\item[(2)] $\FT_Y$ intertwines $\bS(\theta)$ with $S=\hS^{(0,0,0,1)}$; and more generally, $\hbS^\beta(\theta)$ with $\hS^\beta$ for $\beta \in N_1^T(W)$;  
\item[(3)] $\FT_Y$ intertwines $\lambda$ with $z \nabla_{S\partial_S}=z\nabla_{\lambda \hS\partial_\hS}$; and more generally, $z\nabla_{\xi \QW \partial_{\QW}}$ with $z\nabla_{\xi \hS\partial_{\hS}}$ for $\xi \in H^2_T(W)$; 
\item[(4)] $\FT_Y$ commutes with $\nabla_{\theta^{i,k}}$; 
\item[(5)] $\FT_Y$ is homogeneous of degree zero and intertwines $\nabla_{z\partial_z}+\frac{1}{2}$ with $\nabla_{z\partial_z}$. 
\end{itemize}  
\end{proposition} 
\begin{proof} 
Differentiating \eqref{eq:Fourier_JW_for_X_tX} by $\theta^{i,k}$, we obtain $\sfF_X(M_W(\theta) \phi_i \lambda^k) = M_X(\tau(\theta)) z \nabla_{\theta^{i,k}} v(\theta)$, where $\{\phi_i \lambda^k\}$ is a $\C$-basis of $H^*_T(W)$ dual to $\theta^{i,k}$. 
We define a $\C[z][\![\QW,\btheta]\!]$-module map $\FT_X\colon \QDM_T(W) \to \tau^*\QDM(X)^{\rm ext}$ by sending a (topological) basis $\phi_i \lambda^k$ to $z\nabla_{\theta^{i,k}} v(\theta)$. Then we extend it to the localization by $\QW$. It clearly satisfies the commutative diagram \eqref{eq:FT_sfF} and part (1). 
The construction of $\FT_\tX$ is similar. 

Parts (2)-(5) follow from the corresponding properties for $\sfF_Y$ via the commutative diagram \eqref{eq:FT_sfF}.  
Part (2) follows from $\sfF_Y \circ \cS^l = S^l \circ \sfF_Y$ (see \eqref{eq:properties_discrete_Fourier}) and Proposition \ref{prop:shift_fundsol}. 
Part (3) follows from $\sfF_Y \circ (z\xi \QW\partial_{\QW}+\xi) = (z\xi \hS \partial_{\hS} + \kappa_Y(\xi)) \circ \sfF_Y$ (see \eqref{eq:properties_discrete_Fourier_Qdiff}) and \eqref{eq:fundsol_qconn}. 
Part (4) follows from $\sfF_Y \circ \partial_{\theta^{i,k}} = \partial_{\theta^{i,k}} \circ \sfF_Y$ and \eqref{eq:fundsol_qconn}. 
Part (5) follows from the homogeneity of $\sfF_Y$, \eqref{eq:properties_discrete_Fourier_zdiff} and \eqref{eq:fundsol_qconn}. 
\end{proof}

\subsection{Fourier isomorphism} 
\label{subsec:Fourier_isom} 
Since $\FT_\tX$ is a homomorphism of $\C[C_{\tX,\N}^\vee]$-modules, we have (see Proposition \ref{prop:FT_QDM}(1), (2) and \eqref{eq:QDMW_tX})  
\[
\FT_\tX(\QDM_T(W)_{\tX})  \subset \ttau^*\QDM(\tX)^{\rm ext}.  
\]
Thus $\FT_\tX|_{\QDM_T(W)_\tX}$ extends to a map 
\begin{equation} 
\label{eq:FT_tX_completed} 
\FT_\tX\sphat \colon \QDM_T(W)_\tX\sphat \to \ttau^*\QDM(\tX)^{\rm ext} 
\end{equation} 
on the completion \eqref{eq:QDMW_completed}.

\begin{theorem} 
\label{thm:Fourier_isom} 
The Fourier transformation $\FT_\tX\sphat$ in \eqref{eq:FT_tX_completed} is an isomorphism of $\C[z][\![C_{\tX,\N}^\vee,\btheta]\!]$-modules and satisfies parts $(2)$-$(5)$ of Proposition $\ref{prop:FT_QDM}$. If $s_1,\dots,s_N \in H^*_T(W)$ are homogeneous elements such that $\kappa_{\tX}(s_1),\dots,\kappa_{\tX}(s_N)$ form a basis of $H^*(\tX)$, then $s_1,\dots,s_N$ form a basis of $\QDM_T(W)_\tX\sphat$ over $\C[z][\![C_{\tX,\N}^\vee,\btheta]\!]$. 
\end{theorem} 
\begin{proof} It is clear that $\FT_\tX\sphat$ is a homomorphism of $\C[z][\![C_{\tX,\N}^\vee,\btheta]\!]$-modules and satisfies parts $(2)$-$(5)$ of Proposition $\ref{prop:FT_QDM}$. It suffices to show that $\FT_\tX\sphat$ is an isomorphism. 

We set $\cM_0:=\QDM_T(W)$ and $\cM := \QDM_T(W)_\tX$. 
Let $\frakm\subset R:=\C[z][\![\QW,\btheta]\!][S,y S^{-1}]$ be the ideal generated by $S$ and $yS^{-1}$ as in \S\ref{subsec:completion}. 
We first study the fiber $\cM/\frakm \cM$ at $\frakm$; it is a module over $R/\frakm = \C[z][\![\QW,\btheta]\!]/(y)$. 
Let $s_1,\dots,s_N\in H^*_T(W)$ be homogeneous elements such that $\kappa_\tX(s_1),\dots,\kappa_\tX(s_N)$ form a basis of $H^*(\tX)$. 
We claim that $\cM/\frakm \cM$ is generated by $s_1,\dots,s_N$ as an $R/\frakm$-module. 
We have  
\begin{equation} 
\label{eq:zero-fibre} 
\cM/\frakm \cM  = \frac{\cM_0 + yS^{-1} \C[y S^{-1}] \cM_0}{S \cM_0 + y S^{-1} \C[yS^{-1}]\cM_0} 
\cong \frac{\cM_0}{S\cM_0 + (\cM_0\cap yS^{-1} \C[yS^{-1}]\cM_0)} 
\end{equation} 
where we used the fact that $S$ preserves $\cM_0$. 
For $f\in \cM_0$, we write $\LT(f)$ for $f|_{\QW=\btheta=0}$. 
Using Lemma \ref{lem:Kirwan_kernel_shift}(4) below, for any $f\in \cM_0$, we can find $a_1,\dots,a_N\in \C[z]$ and $c_1,c_2,c_3\in H^*_T(W)[z]$ such that $y \bS(\theta)^{-1}c_3 \in \cM_0$ and that 
\[
\LT(f) -\sum_{i=1}^N a_i s_i = \LT\left(\bS(\theta) c_1 + x \bS(\theta)^{-1} c_2 + y \bS(\theta)^{-1} c_3\right) 
\]
where recall that $\bS(\theta)c_1, x\bS(\theta)^{-1} c_2 \in \cM_0$ by Proposition \ref{prop:equiv_QDM_module_over}. 
For a homogeneous $f$, we can choose $a_i,c_i$ to be homogeneous. 
Applying Lemma \ref{lem:Kirwan_kernel_shift}(4) recursively from the lower order of the $(\QW,\btheta)$-power series $f$, we can construct $A_i \in \C[z][\![\QW,\btheta]\!]$, $C_i \in H^*_T(W)[z][\![\QW,\btheta]\!]$ such that 
\[
f - \sum_{i=1}^N A_i s_i = \bS(\theta) C_1 + x \bS(\theta)^{-1} C_2 + y \bS(\theta)^{-1} C_3. 
\]
The right-hand side belong to $S \cM_0 + (\cM_0\cap yS^{-1} \C[yS^{-1}]\cM_0)$ since $xS^{-1}$ preserves $\cM_0$ and $xS^{-1} \cM_0 = y S^{-1} (x y^{-1}) \cM_0 \subset y S^{-1} \cM_0$.  
This together with \eqref{eq:zero-fibre} implies that $\cM/\frakm \cM$ is generated by $s_1,\dots,s_N$. 

Since finitely many elements $s_1,\dots,s_N$ generate $\cM/\frakm \cM$, the completion $\QDM_T(W)_\tX\sphat = \varprojlim_k \cM/\frakm^k \cM$ is also generated by $s_1,\dots,s_N$ over $\C[z][\![C_{\tX,\N}^\vee,\btheta]\!]=\varprojlim_k R/\frakm^k$. We claim that $\FT_\tX\sphat(s_i)=\FT_\tX(s_i)$, $i=1,\dots,N$ form a $\C[z][\![C_{\tX,\N}^\vee,\btheta]\!]$-basis of $\ttau^*\QDM(\tX)^{\rm ext}$. This implies that $s_1,\dots,s_N\in \QDM_T(W)_\tX\sphat$ are linearly independent over $\C[z][\![C_{\tX,\N}^\vee,\btheta]\!]$ and that $\FT_\tX\sphat$ is an isomorphism (since it sends a basis to a basis). 
It suffices to show that the leading terms of $\FT_\tX(s_i)$ at $\QW=S=yS^{-1} =\btheta= 0$ form a $\C[z]$-basis of $H^*(\tX)[z]$. Using the commutative diagram \eqref{eq:FT_sfF} and the fact that $\ttau(\theta)|_{\btheta=\QW=S=yS^{-1}=0}=0$ (see \eqref{eq:initialcond_tau_v}), we find that 
\[
\FT_\tX(s_i)|_{\QW=S=yS^{-1}=\btheta=0} = \sfF_\tX(s_i)|_{\QW=S=yS^{-1}=\btheta=0} = \kappa_\tX(s_i) 
\] 
and the claim follows. 
\end{proof}

\begin{lemma} 
\label{lem:Kirwan_kernel_shift} 
For $f\in \QDM_T(W) = H^*_T(W)[z][\![\QW,\btheta]\!]$, let $\LT(f)=f|_{\QW=\btheta=0} \in H^*_T(W)[z]$ denote the leading term of $f$. 
We have the following: 
\begin{itemize} 
\item[(1)] $\LT(\bS(\theta) c) = i_{\tX*} (e^{-z\partial_\lambda} i_\tX^*c)$ for $c\in H^*_T(W)[z]$. 
\item[(2)] $\LT(x \bS(\theta)^{-1} c) = i_{X_*}(e^{z \partial_\lambda} i_X^*c)$ for $c\in H^*_T(W)[z]$. 
\item[(3)] Let $c= \hjmath_*\hpi^*\gamma\in H^*_T(W)[z]$ with $\gamma \in H^*_T(Z)[z]$, where $\hjmath \colon \hD \to W$ is the inclusion and $\hpi \colon \hD \to Z$ is the projection. Then $y \bS(\theta)^{-1} c$ lies in $\QDM_T(W)$ and $\LT(y \bS(\theta)^{-1} c) = i_{Z\times \PP^1*} (\pr_1^*e^{z\partial_\lambda}\gamma) + i_{X*} \alpha$ for some $\alpha \in H^*_T(X)[z]$. 
\item[(4)] For $f\in \Ker(\kappa_\tX) \subset H^*_T(W)[z]$, there exist $c_1,c_2, c_3 \in H_T^*(W)[z]$ such that $y\bS(\theta)^{-1}c_3\in \QDM_T(W)$ and $f= \LT(\bS(\theta)c_1 + x\bS(\theta)^{-1}c_2 + y\bS(\theta)^{-1}c_3)$. 
\end{itemize} 
\end{lemma} 
\begin{proof} 
Using $M_W(\theta)|_{\QW=\btheta=0} = \id$ and Proposition \ref{prop:shift_fundsol}, we have for $c\in H^*_T(W)[z]$ that 
\[
\LT(\bS(\theta) c) = M_W(\theta) \bS(\theta) c|_{\QW= \btheta=0} = \cS (M_W(\theta) c) |_{\QW=\btheta=0} = \cS(c)|_{\QW=\btheta=0}.  
\]
Thus \eqref{eq:shiftop_W} implies $\LT(\bS(\theta)c) |_X = \LT(\bS(\theta)c)|_Z = 0$,  $\LT(\bS(\theta)c)|_\tX = (\lambda -[D]) e^{-z\partial_\lambda}i_{\tX}^*c$. Therefore $\LT(\bS(\theta)c) = i_{\tX*} (e^{-z\partial_\lambda} i_\tX^*c)$ as required. Part (1) follows. Part (2) follows by a similar computation using \eqref{eq:shiftop_W}. 

Take $\gamma \in H^*_T(Z)[z]$ and set $c=\hjmath_*\hpi^*\gamma$. 
Set $\bbf:=M_W(\theta)c$. 
Then the restrictions of  
\[
M_W(\theta) (y\bS(\theta)^{-1} c) = y\cS^{-1} \bbf 
\]
to the fixed loci $X$, $Z$, $\tX$ are given respectively by 
\begin{equation} 
\label{eq:restrictions_yS^{-1}f} 
-\frac{\lambda}{xy^{-1}} e^{z\partial_\lambda}\bbf_X, \quad \frac{e_{-\lambda}(\cN_{Z/X})}{\lambda+z} e^{z\partial_\lambda} \bbf_Z, \quad 
\frac{y}{\lambda-[D]+z} e^{z\partial_\lambda} \bbf_\tX 
\end{equation} 
by \eqref{eq:shiftop_W}. 
The equivariant class $c$ satisfies $c|_X=(\hjmath_*\hpi^*\gamma)|_X=0$. Therefore, applying the virtual localization formula to $\bbf_X = M_W(\theta)c|_X$, we find that non-zero contributions to $\bbf_X$ arise only from $T$-invariant curves that contain an irreducible component connecting $X$ with $Z$ or $X$ with $\tX$. 
This implies that $\bbf_X$ is divisible by $xy^{-1}$. 
Thus the restrictions \eqref{eq:restrictions_yS^{-1}f} of $M_W(\theta)(y \bS(\theta)^{-1} c)$ are all regular at $\QW=0$, and therefore $y \bS(\theta)^{-1} c \in \QDM_T(W)$. Moreover \eqref{eq:restrictions_yS^{-1}f} shows that $\LT(y \bS(\theta)^{-1} c)|_Z= e_{-\lambda}(\cN_{Z/X})e^{z\partial_\lambda} \gamma$ and $\LT(y\bS(\theta)^{-1}c)|_\tX =0$. By Lemma \ref{lem:GKZ_type}(4), we have $\LT(y\bS(\theta)^{-1} c) = i_{Z\times \PP^1*}(\pr_1^*e^{z\partial_\lambda}\gamma)+ i_{X*} \alpha$ for some $\alpha \in H^*_T(X)[z]$.  

Finally we prove part (4). By Lemma \ref{lem:Kirwan_kernel}, $f$ can be written as the sum $i_{\tX*}g_1 + i_{X*}g_2 + i_{Z\times \PP^1*} \pr_1^*g_3$ for $g_1\in H^*_T(\tX)[z]$, $g_2 \in H^*_T(X)[z]$ and $g_3 \in H^*_T(Z)[z]$. By part (3), setting $c_3 = \hjmath_*\hpi^*e^{-z\partial_\lambda}g_3$, we have  $y\bS(\theta)^{-1} c_3\in \QDM_T(W)$ and $f-\LT(y\bS(\theta)^{-1} c_3)$ is of the form $i_{\tX*}g_1 + i_{X*}g_2'$ for some $g_2'\in H^*_T(X)[z]$. By parts (1), (2), this can be written as $\LT(\bS(\theta) c_1 +  x \bS(\theta)^{-1} c_2)$ for some $c_1,c_2 \in H^*_T(W)$ since $i_X^*$ and $i_\tX^*$ are surjective. 
\end{proof} 

\subsection{Fourier projections} 
\label{subsec:Fourier_projections} 
We introduce the new variable $\frq:= S_Z^{-1} = yS^{-1}$ with $\deg \frq = 2(r-1)$. 
Let $\frs$ be the following number: 
\begin{equation} 
\label{eq:frs} 
\frs= \begin{cases} 
r-1 & \text{if $r$ is even;} \\
2(r-1) & \text{if $r$ is odd.} 
\end{cases} 
\end{equation} 
In this section, we construct projections from $\QDM_T(W)_\tX\sphat$ to the quantum $D$-modules of $X$ and $Z$ defined over formal Laurent series in $\frq^{-1/\frs}$, via the Fourier transformation. 
\subsubsection{Projection for $X$} 
\label{subsubsec:projection_X}
The Fourier transformation $\FT_X$ defines a map 
\[
\FT_X \colon \QDM_T(W)_\tX \to \tau^*\QDM(X)^{\rm ext}[\frq] 
\]
by Proposition \ref{prop:FT_QDM}(1), (2) and \eqref{eq:QDMW_tX_2nd}. In order to extend this map to the completion $\QDM_T(W)_\tX\sphat$, we consider the following extension of the ground ring for $\QDM(X)^{\rm ext}[\frq]$: 
\[
\C[z][\![C_{X,\N}^\vee]\!][\frq] \subset \C[z](\!(\frq^{-1/\frs})\!)[\![\QW]\!] 
\]
where the formal Laurent series ring $\C[z](\!(\frq^{-1/\frs})\!) = \C[\frq^{\pm1/\frs}][\![z]\!]$ should be understood in the graded sense (see \S\ref{subsec:formal_power_series}). 
This inclusion holds because $\C[C_{X,\N}^\vee] = \C[\QW][xy^{-1}\frq, \frq^{-1}]$ (see \eqref{eq:dualmonoids}) and $xy^{-1}\in \C[\QW]$ (see also Figure \ref{fig:Mori_cones}).  
We define formal Laurent versions of $\QDM(X)^{\rm ext}$, $\tau^*\QDM(X)^{\rm ext}$ by  
\begin{align*} 
\QDM(X)^{\rm La} & := H^*(X)[z](\!(\frq^{-1/\frs})\!)[\![\QW,\tau]\!], \\ 
\tau^*\QDM(X)^{\rm La} & := H^*(X)[z](\!(\frq^{-1/\frs})\!)[\![\QW,\btheta]\!].   
\end{align*} 
Note that $\tau^*\QDM(X)^{\rm ext}[\frq] \subset \tau^*\QDM(X)^{\rm La}$. Note also that $\C[z](\!(\frq^{-1/\frs})\!)[\![\QW,\btheta]\!]$ contains the ground ring $\C[z][\![C_{\tX,\N}^\vee,\btheta]\!]$ for $\QDM_T(W)_\tX\sphat$ since $\frq$ has positive degree. 
The flat connections $\nabla$ on these modules are defined similarly to \S\ref{subsec:extended_QDM}. 
\begin{proposition} 
\label{prop:FTX_completion} 
The map $\FT_X|_{\QDM_T(W)_\tX}$ in Proposition $\ref{prop:FT_QDM}$ extends to a homomorphism of $\C[z][\![C_{\tX,\N}^\vee,\btheta]\!]$-modules 
\begin{equation} 
\label{eq:FT_X_completed} 
\FT_X\sphat \colon \QDM_T(W)_\tX\sphat \to \tau^*\QDM(X)^{\rm La}    
\end{equation} 
on the completion. The map $\FT_X\sphat$ satisfies parts $(2)$-$(5)$ of Proposition $\ref{prop:FT_QDM}$.  
\end{proposition} 
\begin{proof} 
It suffices to show that $\FT_X|_{\QDM_T(W)_\tX}$ continuously extends to the completion. The point of the proof is that $\frq = y S^{-1}$ has positive degree. For $N=1,2,3,\dots$, let $\Omega_N$ be the set of monomials $\QW^\delta \btheta^I$ in the variables $(\QW,\btheta)$ such that $\omega_W \cdot \delta \ge N$, or $\btheta^I$ is divisible by $\theta^{i,k}$ for some $k\ge N$, or $\btheta^I$ is of the form $\theta^{i_1,k_1} \cdots \theta^{i_l, k_l}$ with $l\ge N$, where $\omega_W$ is a fixed ample class. 
For a graded module $K$, we write $F_N(K[\![\QW,\btheta]\!])\subset K[\![\QW,\btheta]\!]$ for the subspace consisting of power series $\sum a_{\delta,I} \QW^\delta\btheta^I$ with $a_{\delta,I}\in K$ being non-zero only when $\QW^\delta \btheta^I \in \Omega_N$.  Since $\FT_X(\QDM_T(W)) \subset \tau^*\QDM(X)^{\rm La}$, we have 
\begin{equation} 
\label{eq:FTX_filtration} 
\FT_X\left(I_N \cdot \C[\frq^\pm]\cdot \QDM_T(W) + F_N(\QDM_T(W))\right) 
\subset F_N(\tau^*\QDM(X)^{\rm La})
\end{equation} 
where we set $I_N := F_N(\C[z][\![\QW,\btheta]\!]) \subset \C[z][\![\QW,\btheta]\!]$. 
Because of this, it suffices to show that for any $N\ge 1$ and a fixed degree $n\in \Z$, there exists $k\ge 1$ such that every homogeneous element $f$ of degree $n$ in $\frakm^{2k} \QDM_T(W)_\tX$ satisfies 
\begin{equation} 
\label{eq:f_lies_in_the_deep_filter} 
f \in I_N \cdot \C[\frq^\pm]\cdot \QDM_T(W) + F_N(\QDM_T(W))
\end{equation} 
where $\frakm \subset \C[z][\![\QW,\btheta]\!][yS^{-1},S]$ is the ideal generated by $S$ and $y S^{-1}=\frq$ as in \S\ref{subsec:completion}.  
An element $f\in \frakm^{2k} \QDM_T(W)_\tX$ of degree $n$ can be written in the form 
\[
f = S^k g + \sum_{l\ge k} (y S^{-1})^l h_l  
\]
for some $g \in \QDM_T(W)_\tX=\C[\frq]\cdot \QDM_T(W)$ and $h_l\in \QDM_T(W)$. Here the sum over $l$ is finite and we may assume that $g$ and $h_l$ are homogeneous. In particular we have $\deg h_l = n - 2(r-1)l \le n- 2(r-1) k$. We can choose $k$ sufficiently large so that $y^k \in I_N$ and every elements of degree $\le n- 2 (r-1) k$ in $\QDM_T(W)$ lies in $F_N(\QDM_T(W))$. This is possible because $\deg z>0$ and $H^*_T(W)$ is non-negatively graded. Then $S^k g = y^k \frq^{-k} g \in I_N \cdot \C[\frq^\pm]\cdot \QDM_T(W)$ and $h_l \in F_N(\QDM_T(W))$; thus \eqref{eq:f_lies_in_the_deep_filter} holds. 
\end{proof} 

\begin{remark} 
We do not need to add the $\frs$th root $\frq^{-1/\frs}$ of $\frq^{-1}$ to define $\FT_X\sphat$, but we will need it when we consider the Fourier transformation associated with $Z$ below. 
\end{remark} 

\subsubsection{Projection for $Z$} 
\label{subsubsec:projection_Z}
We also rewrite the continuous Fourier transformations $\scrF_{Z,j}$, $j=0,\dots,r-2$ from \S\ref{subsec:continuous_Fourier} in terms of quantum $D$-modules. 
Let $\sigma$ denote the parameter of $H^*(Z)$. Let $\{\sigma^i\}$ be the coordinates of $\sigma$ corresponding to a homogeneous basis $\{\phi_{Z,i}\}$ of $H^*(Z)$. 
The quantum $D$-module of $Z$ is given by (see \S\ref{subsec:qcoh_qconn}) 
\[
\QDM(Z) = H^*(Z)[z][\![Q_Z, \sigma]\!] 
\]
equipped with the quantum connection $\nabla$ and the pairing $P_Z$, where $Q_Z$ denotes the Novikov variable of $Z$. 
In view of Corollary \ref{cor:Fourier_transform_JW}, we consider the (not necessarily injective but degree-preserving) extension of rings: 
\begin{equation} 
\label{eq:extension_Novikov_Z} 
\C[z][\![Q_Z,\sigma]\!] \to \C[z](\!(\frq^{-1/\frs})\!)[\![\QW,\sigma]\!], \ Q_Z^d \mapsto Q^{\imath_*d} \frq^{-\rho_Z \cdot d/(r-1)}
=\QW^{i_{Z*}d} S_Z^{-\rho_Z \cdot d/c_Z}, 
\end{equation} 
where $\imath\colon Z\to X$, $i_Z \colon Z \to W$ are the inclusions and $\rho_Z = c_1(\cN_{Z/W}) = c_1(\cN_{Z/X})$ (see Notation \ref{nota:normal_bundle}). 
We introduce a formal Laurent version of the quantum $D$-module of $Z$ as follows:  
\[
\QDM(Z)^{\rm La} := \QDM(Z)\otimes_{\C[z][\![Q_Z,\sigma]\!]} \C[z](\!(\frq^{-1/\frs})\!)[\![\QW,\sigma]\!] = H^*(Z)[z](\!(\frq^{-1/\frs})\!)[\![\QW,\sigma]\!]. 
\]
The quantum connection on $\QDM(Z)^{\rm La}$ is defined similarly to \S\ref{subsec:extended_QDM}; we define
\begin{align}
\label{eq:Z_qconn_La}  
\begin{split} 
\nabla_{\sigma^i} &=  \partial_{\sigma^i} + z^{-1} (\phi_{Z,i}\star_\sigma) \\
\nabla_{z\partial_z} & = z \partial_z - z^{-1} (E_Z\star_\sigma) + \mu_Z \\ 
\nabla_{\xi \hS \partial_\hS} & = \xi \hS\partial_\hS + z^{-1}  (\kappa_Z(\xi) \star_\sigma)
\end{split} 
\end{align} 
for $\xi\in H^2_T(W)$, where $\kappa_Z\colon H^2_T(W) \to H^2(Z)$ is the map sending $\xi$ to $i_Z^*\xi|_{\lambda = \rho_Z/(r-1)}$ and the quantum product $\star_\sigma$ is pushed forward along \eqref{eq:extension_Novikov_Z}. The map $\kappa_Z$ is dual to the homomorphism \eqref{eq:extension_Novikov_Z}; it is defined only on $H^2_T(W)$ and is \emph{not} the Kirwan map. 
\begin{remark} 
\label{rem:QDMZLa}
As in Remark \ref{rem:QDM}, the structure of $\QDM(Z)^{\rm La}$ can be reduced to a smaller ring, namely, the image $R$ of $\C[z][\![Q_Z e^{\sigma^{(2)}},\sigma']\!][\sigma^0]$ under \eqref{eq:extension_Novikov_Z}. Here we used notation analogous to Remark \ref{rem:QDM}. In other words, the connection \eqref{eq:Z_qconn_La} multiplied by $z$ preserves the submodule $H^*(Z)\otimes R$.  
\end{remark}

Let $j\in \{0,\dots,r-2\}$. By Corollary \ref{cor:Fourier_transform_JW}, there exist maps  $\sigma=\sigma_j(\theta) \in H^*(Z)(\!(\frq^{-\frac{1}{r-1}})\!)[\![\QW,\btheta]\!]$ and $u_j(\theta) \in q_{Z,j} H^*(Z)[z](\!(\frq^{-\frac{1}{r-1}})\!)[\![\QW,\btheta]\!]$ satisfying 
\begin{align}
\label{eq:initialcond_mirrormap}
\begin{split} 
\sigma_j(\theta)|_{\QW=0} &\in h_{Z,j} + \frakm' H^*(Z)[\![\frq^{-\frac{1}{r-1}},\btheta_Z \frq^{\frac{\bullet}{r-1}}]\!] \\ 
u_j(\theta)|_{\QW=0} & \in q_{Z,j} \left(1+\frakm'' H^*(Z)[z][\![\frq^{-\frac{1}{r-1}},\btheta_Z  \frq^{\frac{\bullet}{r-1}}]\!]\right)
\end{split} 
\end{align} 
(here we use notation analogous to Corollary \ref{cor:Fourier_transform_JW}) such that 
\begin{equation} 
\label{eq:Fourier_JW_for_Z}
\frq^{\rho_Z/((r-1)z)} \scrF_{Z,j}(J_W(\theta)) = M_Z(\sigma_j(\theta)) u_j(\theta)
\end{equation} 
where the fundamental solution $M_Z(\sigma)$ is pushed forward along the ring extension \eqref{eq:extension_Novikov_Z}. 
The quantities $q_{Z,j}$, $h_{Z,j}$ are given by (see \eqref{eq:qFj-hFj}) 
\begin{equation} 
\label{eq:qZj_hZj} 
q_{Z,j} = \frac{1}{\sqrt{r-1}} e^{\frac{\pi \iu}{r-1} (jr + \frac{1}{2})} \frq^{-\frac{r}{2(r-1)}}, \quad 
h_{Z,j} = \frac{2 \pi \iu}{r-1} (j+\frac{1}{2}) \rho_Z.  
\end{equation} 
The pullback $\sigma_j^*\QDM(Z)^{\rm La}$ is defined to be the module  $H^*(Z)[z](\!(\frq^{-1/\frs})\!)[\![\QW,\btheta]\!]$ equipped with the following pulled-back connection (cf.~\eqref{eq:pull-back_qconn}):  
\begin{align*} 
\nabla_{\theta^{i,k}} & =\partial_{\theta^{i,k}} + z^{-1} (\partial_{\theta^{i,k}} \sigma_j(\theta)) \star_{\sigma_j(\theta)}, \\ 
\nabla_{z\partial_z} & = z\partial_z - z^{-1} (E_Z\star_{\sigma_j(\theta)}) + \mu, \\ 
\nabla_{\xi \hS \partial_\hS} &=  \xi \hS\partial_\hS + z^{-1} (\kappa_Z(\xi)\star_{\sigma_j(\theta)}) + z^{-1} (\xi \hS\partial_\hS \sigma_j(\theta)) \star_{\sigma_j(\theta)}. 
\end{align*} 
Due to the constant term $h_{Z,j}$ in the change of variables $\sigma= \sigma_j(\theta)$, the pullback of functions $\sigma_j^* \colon \C[z](\!(\frq^{-1/\frs})\!)[\![\QW,\sigma]\!] \to \C[z](\!(\frq^{-1/\frs})\!)[\![\QW,\btheta]\!]$ is ill-defined. However, the pullback of connections is well-defined due to the Divisor Equation. In other words, the structure of $\QDM(Z)^{\rm La}$ can be reduced to the smaller ring $R$ as in Remark \ref{rem:QDMZLa} and the pullback $\sigma_j^* \colon R \to \C[z](\!(\frq^{-1/\frs})\!)[\![\QW,\btheta]\!]$ is well-defined. 


\begin{proposition} 
\label{prop:FT_Zj} 
For $j\in \{0,1,\dots,r-2\}$, we have a map 
\[
\FT_{Z,j} \colon \QDM_T(W)[\QW^{-1}] \to \sigma_j^*\QDM(Z)^{\rm La}[\QW^{-1}]  
\]
of $\C[z][\![\QW,\btheta]\!][\QW^{-1}]$-modules satisfying $M_Z(\sigma_j(\theta)) \circ \FT_{Z,j} = \frq^{\rho_Z/((r-1)z)} \scrF_{Z,j}\circ M_W(\theta)$,  i.e.~the following diagram commutes:  
\begin{equation} 
\label{eq:FT_scrF} 
\begin{aligned} 
\xymatrix{
\QDM_T(W)[\QW^{-1}] \ar[rr]^{\FT_{Z,j}} \ar[d]^{M_W(\theta)} 
&  & \sigma_j^*\QDM(Z)^{\rm La}[\QW^{-1}] \ar[d]^{M_Z(\sigma_j(\theta))} \\ 
\cH_W^{\rm rat}[\QW^{-1}] \ar[rr]^{\frq^{\rho_Z/((r-1)z)} \scrF_{Z,j}} &  & \cH_Z^{\rm La}[\QW^{-1}]  
}
\end{aligned}
\end{equation} 
where we set $\cH_Z^{\rm La} := H^*(Z)[z,z^{-1}](\!(\frq^{-1/\frs})\!)[\![\QW]\!]$. The map $\FT_{Z,j}$ satisfies the following: 
\begin{itemize} 
\item[(1)] $\FT_{Z,j}(\QDM_T(W)_\tX) \subset \sigma_j^*\QDM(Z)^{\rm La}$.  
\item[(2)] $\FT_{Z,j}$ intertwines $\hbS^\beta(\theta)$ with $\hS^\beta$ for $\beta \in N_1^T(W)$;  
\item[(3)] $\FT_{Z,j}$ intertwines $\lambda$ with $z\nabla_{S\partial_S} +\lambda_j = z\nabla_{\lambda \hS\partial_{\hS}}+\lambda_j = -z \nabla_{\frq\partial_\frq}+\lambda_j$; and more generally $z\nabla_{\xi  \QW\partial_\QW}$ with $z\nabla_{\xi\hS\partial_\hS} +i_{\pt}^*\xi|_{\lambda=\lambda_j}$ for $\xi \in H^2_T(W)$, where $\lambda_j =e^{-\frac{2\pi \iu}{r-1} (j+ \frac{r}{2})}\frq^{\frac{1}{r-1}}$ $($see  $\S\ref{subsubsec:formal_asymptotics}$$)$ and $i_{\pt}\colon \pt \to Z\hookrightarrow W$ is the inclusion of a point in $Z$; 
\item[(4)] $\FT_{Z,j}$ commutes with $\nabla_{\theta^{i,k}}$; 
\item[(5)] $\FT_{Z,j}$ is homogeneous of degree $-r$ and intertwines $\nabla_{z\partial_z}+\frac{1}{2}$ with $\nabla_{z\partial_z}+ z^{-1}(r-1)\lambda_j$.  
\end{itemize} 
Moreover, $\FT_{Z,j} \colon \QDM_T(W)_\tX \to \sigma_j^*\QDM(Z)^{\rm La}$ extends to a homomorphism of $\C[z][\![C_{\tX,\N}^\vee,\btheta]\!]$-modules  
\[
\FT_{Z,j}\sphat \colon \QDM_T(W)_\tX\sphat \to \sigma_j^*\QDM(Z)^{\rm La} 
\]
on the completion and $\FT_{Z,j}\sphat$ satisfies the same properties as $(2)$-$(5)$ above. 
\end{proposition} 
\begin{proof} 
By differentiating \eqref{eq:Fourier_JW_for_Z} by $\theta^{i,k}$, we obtain $\frq^{\rho_Z/((r-1)z)} \scrF_{Z,j} (M_W(\theta) \phi_i \lambda^k) = M_Z(\sigma_j(\theta)) \nabla_{\theta^{i,k}} u_j(\theta)$. 
We define the $\C[z][\![\QW,\btheta]\!]$-module map $\FT_{Z,j}\colon \QDM_T(W)\to \sigma_j^*\QDM(Z)^{\rm La}$ by sending a topological basis $\phi_i \lambda^k$ to $\nabla_{\theta^{i,k}}u_j(\theta)$. Then we extend it to the localization by $\QW$. It clearly satisfies the commutative diagram \eqref{eq:FT_scrF}. 

Parts (2)-(5) follow from the corresponding properties for $\scrF_{Z,j}$ via the commutative diagram \eqref{eq:FT_scrF}. Part (2) follows from Proposition \ref{prop:properties_conti_Fourier}(1) and Proposition \ref{prop:shift_fundsol}. The first half of part (3) follows from the computation: 
\begin{align*} 
& M_Z(\sigma_j(\theta)) \circ \FT_{Z,j} \circ \lambda = \frq^{\rho_Z/((r-1)z)}\scrF_{Z,j} \circ M_W(\theta) \circ \lambda 
&& \text{by \eqref{eq:FT_scrF}} \\ 
& = (z S\partial_S + \tfrac{\rho_Z}{r-1} + \lambda_j) \circ 
\frq^{\rho_Z/((r-1)z)}\scrF_{Z,j} \circ M_W(\theta)
&& \text{by Proposition \ref{prop:properties_conti_Fourier}(2)} \\
& = (z \lambda \hS\partial_\hS + \kappa_Z(\lambda) + \lambda_j) \circ M_Z(\sigma_j(\theta)) \circ \FT_{Z,j} && \text{by \eqref{eq:FT_scrF}} \\ 
& = M_Z(\sigma_j(\theta)) \circ (z \nabla_{\lambda \hS \partial_\hS} +\lambda_j) \circ \FT_{Z,j} 
&& \text{by \eqref{eq:fundsol_qconn}.} 
\end{align*} 
To show the latter half of part (3), it suffices to show that $\FT_{Z,j}$ intertwines $z\nabla_{\xi \QW\partial_\QW}$ with $z \nabla_{\xi \hS \partial_\hS}$ when $i_{\pt}^*\xi = 0$.  
Note that $i_{\pt}^*\xi=0$ implies $(\xi \hS \partial_\hS) \frq =0$. 
Therefore it suffices to show that $\scrF_{Z,j} \circ ( z \xi \QW\partial_{\QW}+\xi) =(z \xi \hS \partial_\hS +\kappa_Z(\xi) ) \circ \scrF_{Z,j}$. The condition $i_{\pt}^*\xi =0$ also implies that $\kappa_Z(\xi) = \xi|_Z\in H^2(Z)$ is independent of $\lambda$. The definition of $\scrF_{Z,j}$ in \S\ref{subsubsec:formal_asymptotics} together with $(\xi \hS \partial_\hS) \frq =0$ implies that $\scrF_{Z,j} \circ z \xi \QW\partial_\QW = z \xi \hS \partial_\hS \circ \scrF_{Z,j}$ and $\scrF_{Z,j} \circ \xi = \kappa_Z(\xi) \circ \scrF_{Z,j}$; part (3) follows. 
Part (4) follows from $\partial_{\theta^{i,k}} \circ \scrF_{Z,j} = \scrF_{Z,j} \circ \partial_{\theta^{i,k}}$ and \eqref{eq:fundsol_qconn}. 
By the definition of $\scrF_{Z,j}$, it is easy to see that $\frq^{\rho_Z/((r-1)z)} \scrF_{Z,j}$ is homogeneous of degree $-r$. 
Hence $\FT_{Z,j}$ is also homogeneous of degree $-r$. Part (5) follows from Proposition \ref{prop:properties_conti_Fourier}(3) and \eqref{eq:fundsol_qconn}. 

Part (1) follows from $\FT_{Z,j}(\QDM_T(W)) \subset \sigma_j^* \QDM(Z)^{\rm La}$, \eqref{eq:QDMW_tX_2nd} and part (2). 

Finally, we show that $\FT_{Z,j}|_{\QDM_T(W)_\tX}$ extends to the completion. But the discussion is completely parallel to Proposition \ref{prop:FTX_completion}: it is sufficient to observe that \eqref{eq:FTX_filtration} holds when we replace $\FT_X$ with $\FT_{Z,j}$ and $\tau^*\QDM(X)^{\rm La}$ with $\sigma_j^*\QDM(Z)^{\rm La}$.  
\end{proof} 

We introduce a shift of the map $\sigma=\sigma_j(\theta)$ in the identity direction so that $\FT_{Z,j}$ is more compatible with the connection. Set 
\begin{equation} 
\label{eq:varsigma} 
\varsigma_j(\theta) := \sigma_j(\theta)  - (r-1) \lambda_j.  
\end{equation} 
Because of the leading term $-(r-1)\lambda_j$, the pullback $M_Z(\varsigma_j(\theta)) = e^{-(r-1) \lambda_j/z} M_Z(\sigma_j(\theta))$ of the fundamental solution by $\varsigma_j$ may not be well-defined. However, the pullback $\varsigma_j^*\nabla$ of the quantum connection of $Z$ remains well-defined. 
Shifting the parameter $\sigma$ in the identity direction does not change the quantum product but affects the Euler vector field and the covariant derivative with respect to $\xi \hS \partial_\hS$ through $-(\xi \hS \partial_\hS) (r-1)\lambda_j =- \langle \xi \hS\partial_\hS, \frac{d\frq}{\frq}\rangle \lambda_j =   i_{\pt}^*\xi|_{\lambda \to \lambda_j}$. 
Therefore, we have 
\[
\varsigma_j^* \nabla = \sigma_j^* \nabla  + (r-1) \lambda_j \frac{dz}{z^2} - \frac{\lambda_j}{z}\frac{d\frq}{\frq}   
\]
where $\nabla$ represents the connection on $\QDM(Z)^{\rm La}$ and $\sigma_j^*\nabla$, $\varsigma_j^*\nabla$ are its pullbacks. We define $\varsigma_j^*\QDM(Z)^{\rm La}$ to be the module $H^*(Z)[z](\!(\frq^{-1/\frs})\!)[\![\QW,\btheta]\!]$ equipped with the pulled-back quantum connection $\varsigma_j^*\nabla$. The well-definedness can also be explained by the fact that the pullback $\varsigma_j^*\colon R\to \C[z](\!(\frq^{-1/\frs})\!)[\![\QW,\btheta]\!]$ is well-defined for the ring $R$ in Remark \ref{rem:QDMZLa}. 
Henceforth, we shall denote the connection on $\varsigma_j^*\QDM(Z)^{\rm La}$ simply by $\nabla$, instead of $\varsigma_j^*\nabla$. 
The preceding discussion implies the following corollary. 
\begin{corollary} 
\label{cor:FT_Zj} 
The map $\FT_{Z,j}\sphat \colon \QDM_T(W)_\tX\sphat \to \varsigma_j^*\QDM(Z)^{\rm La}$ of $\C[z][\![C_{\tX,\N}^\vee,\btheta]\!]$-modules, defined as the composition of $\FT_{Z,j}\sphat$ in Proposition $\ref{prop:FT_Zj}$ and the identity map $\sigma_j^*\QDM(Z)^{\rm La} \to \varsigma_j^*\QDM(Z)^{\rm La}$ of the underlying modules, intertwines the connection $\nabla+\frac{1}{2} \frac{dz}{z}$ with the connection $\nabla$. 
\end{corollary}

\subsection{Decomposition} 
\label{subsec:decomposition} 
The Fourier transformations $\FT_X\sphat$, $\FT_{Z,j}\sphat$ from Proposition $\ref{prop:FTX_completion}$ and Corollary $\ref{cor:FT_Zj}$ define a map $\Phi$ of $\C[z][\![C_{\tX,\N}^\vee,\btheta]\!]$-modules 
\begin{equation} 
\label{eq:decomposition} 
\Phi:=\FT_X\sphat \oplus \bigoplus_{j=0}^{r-2} \FT_{Z,j}\sphat \colon \QDM_T(W)_\tX\sphat  \to \tau^*\QDM(X)^{\rm La} \oplus \bigoplus_{j=0}^{r-2} \varsigma_j^*\QDM(Z)^{\rm La}. 
\end{equation} 
In this section, we show that this induces an isomorphism over $\C[z](\!(\frq^{-1/\frs})\!)[\![\QW,\btheta]\!]$.

\begin{theorem} 
\label{thm:decomposition} 
The map $\Phi$ in \eqref{eq:decomposition} induces an isomorphism after base change of  $\QDM_T(W)_\tX\sphat$ to $\C[z](\!(\frq^{-1/\frs})\!)[\![\QW,\btheta]\!]$. Composing $\Phi$ with  the inverse of the Fourier isomorphism $\FT_\tX\sphat$ in Theorem $\ref{thm:Fourier_isom}$, we obtain a map 
\[
\bPsi := \Phi \circ (\FT_\tX\sphat)^{-1} \colon \ttau^* \QDM(\tX)^{\rm ext} \to \tau^*\QDM(X)^{\rm La} \oplus \bigoplus_{j=0}^{r-2} \varsigma_j^*\QDM(Z)^{\rm La},    
\] 
inducing an isomorphism after base change to $\C[z](\!(\frq^{-1/\frs})\!)[\![\QW,\btheta]\!]$. This map $\bPsi$ commutes with the quantum connection and intertwines the pairing $P_\tX$ with $P_X \oplus \bigoplus_{j=0}^{r-2} P_Z$. 
\end{theorem} 

The rest of the section is devoted to the proof of Theorem \ref{thm:decomposition}. The logical structure of the proof is similar to \cite[\S 5.6]{Iritani-Koto:projective_bundle}. 
Let $\{\phi_{X,i}\}$, $\{\phi_{Z,m}\}$ be bases of $H^*(X)$, $H^*(Z)$ respectively as before. We introduce the following elements of $H^*_T(W)$ (see \S\ref{subsec:cohomology} for the notation): 
\begin{equation}
\label{eq:basis_QDMTW}  
c_i := \hvarphi^* \pr_1^* \phi_{X,i}, \qquad 
c_{l,m} := \hjmath_* (\hp^l \hpi^*\phi_{Z,m}) \qquad \text{with $0\le l \le r-2$.}  
\end{equation} 
The images of these elements under the Kirwan map $\kappa_\tX$ are 
\[
\kappa_\tX(c_i) = \varphi^* \phi_{X,i}, \qquad 
\kappa_\tX(c_{l,m}) = \jmath_*(p^l  \pi^*\phi_{Z,m}) 
\]
and they form a basis of $H^*(\tX)$ by \eqref{eq:cohomology_bu}. Thus $\{c_i, c_{l,m}\}$ gives a basis of $\QDM_T(W)_\tX\sphat$ over $\C[z][\![C_{\tX,\N}^\vee,\btheta]\!]$ by Theorem \ref{thm:Fourier_isom}. 
It is sometimes convenient to identify $\{c_i,c_{l,m}\}$ with a basis of $H^*(X) \oplus H^*(Z)^{\oplus (r-1)}$: $c_i$ corresponds to $\phi_{X,i}$ in the factor $H^*(X)$ and $c_{l,m}$ corresponds to $\phi_{Z,m}$ in the $l$th factor of $H^*(Z)^{\oplus (r-1)}$. 

\subsubsection{That $\Phi$ induces an isomorphism over $\C[z](\!(\frq^{-1/\frs})\!)[\![\QW,\btheta]\!]$}
\label{subsubsec:Phi_isom} 
It suffices to show that $\Phi(c_i), \Phi(c_{l,m})$ form a basis over $\C[z](\!(\frq^{-1/\frs})\!)[\![\QW,\btheta]\!]$. For this, it suffices to study their restrictions to $\QW=\btheta=0$, or the images under the homomorphism $\C[z](\!(\frq^{-1/\frs})\!)[\![\QW,\btheta]\!] \to \C[z](\!(\frq^{-1/\frs})\!)$ sending $\QW^\delta$, $\theta^{i,k}$ with $\delta \in \NEN(W)\setminus \{0\}$ to zero. 
\begin{remark} 
\label{rem:QW=0_subtlety}
In this section (\S\ref{subsec:decomposition}), we often consider the limit $\QW=0$ as power series in $\frq^{\pm 1/\frs}=S_Z^{\mp 1/\frs}$ and $\QW$. Note that the meaning of the limit $\QW=0$ is different from that in the proof of Corollary \ref{cor:Fourier_transform_JW_GIT} (see footnote \ref{footnote:meaning_of_QW=0}).  
We have extended the Novikov rings of $X$, $\tX$ and $Z$ to $\C(\!(\frq^{-1/\frs})\!)[\![\QW]\!]$ (see \S\ref{subsec:extended_QDM}, \S\ref{subsubsec:projection_X} and \eqref{eq:extension_Novikov_Z}). 
It is important to note that the Novikov rings for $X$ and $Z$ map to $\C$ (the constant term) under the limit $\QW=0$, but this is not the case for $\tX$ as the class of a line in a fiber of $D\to Z$ corresponds to $\frq =yS^{-1}$.   
\end{remark} 

We start by calculating the discrete Fourier transforms  $\sfF_X(M_W(\theta) c)$, $\sfF_\tX(M_W(\theta) c)$ along $\QW = 0$. 
\begin{lemma} 
\label{lem:sfF_MW_c} 
For $c\in H^*_T(W)$, we have 
\begin{align*} 
\sfF_X(M_W(\theta) c)|_{\QW=0} 
& = \kappa_X(c_\theta) + \sum_{k>0} \frq^{-k} 
\imath_*\left( \frac{\prod_{\nu=1}^{k-1} e_{-\nu z}(\cN_{Z/X})}{k!z^k}  \left[ i_Z^*c_\theta \right]_{\lambda=kz} \right) \\ 
\sfF_\tX(M_W(\theta)c)|_{\QW=0} 
& = \kappa_\tX(c_\theta) + \sum_{k>0} \frq^k 
\jmath_* \left(\frac{\prod_{\nu=1}^{k-1} ([D] -\nu z)}{\prod_{\nu=1}^k e_{-[D] + \nu z}(\cN_{Z/X})}
\left[i_Z^*c_\theta\right]_{\lambda=[D]-kz} \right) 
\end{align*} 
where $\imath \colon Z \to X$, $\jmath\colon D \to \tX$, $i_Z\colon Z\to W$ are the inclusions, $e_\lambda(\cdots)$ is the equivariant Euler class \eqref{eq:equiv_Euler} and $c_\theta := e^{\theta/z} c$. 
\end{lemma} 

\begin{proof} 
We have 
\begin{align}
\label{eq:sfF_X_MW_c} 
\begin{split}  
\sfF_X(M_W(\theta)c) |_{\QW=0} 
& = \sum_{k\in \Z} S^k \kappa_X(\cS^{-k} M_W(\theta)c) \Bigr|_{\QW=0} \\ 
& = \kappa_X(c_\theta) + \sum_{k>0} (y^{-1}S)^{k} \kappa_X([y^{-k}] \cS^{-k} M_W(\theta)c) 
\end{split} 
\end{align} 
where $[y^{-k}] \cS^{-k} M_W(\theta)c$ denotes the coefficient of $y^{-k} = \QW^{(0,0,-k,0)}$ of the $\QW$-power series $\cS^{-k}M_W(\theta) c$. Here note that, since the $\hS$-power series $\sfF_X(M_W(\theta)c)$ is supported on $C_{X,\N}^\vee$ (Proposition \ref{prop:support_conjecture}), terms with $k<0$ do not contribute to the limit $\QW=0$. 
For $k>0$, the restrictions of $\cS^{-k}M_W(\theta)c$ to the fixed loci $X,Z,\tX$ are given by the vector (see \eqref{eq:shiftop_W}): 
\[
\renewcommand\arraystretch{1.5}
\begin{pmatrix} 
x^{-k} \left( \prod_{\nu=0}^{k-1} (-\lambda -\nu z) \right) e^{kz \partial_\lambda} (c_\theta|_X + O(\QW)) \\ 
y^{-k} \frac{\prod_{\nu=0}^{k-1}e_{-\lambda-\nu z}(\cN_{Z/X})}{\prod_{\nu=1}^k (\lambda+ \nu z)} 
e^{k z\partial_\lambda} (c_\theta|_Z + O(\QW)) \\ 
\frac{1}{\prod_{\nu=1}^k (-D+\lambda+\nu z)} e^{kz \partial_\lambda} (c_\theta|_\tX + O(\QW))  
\end{pmatrix}.  
\] 
The coefficient of $y^{-k}$ of this vector is of the form 
\[
\begin{pmatrix} 
* \\ 
\frac{\prod_{\nu=0}^{k-1}e_{-\lambda-\nu z}(\cN_{Z/X})}{\prod_{\nu=1}^k (\lambda+ \nu z)} 
e^{k z\partial_\lambda} c_\theta|_Z \\
0 
\end{pmatrix}. 
\]
Since the shift operator preserves tangent spaces to the Givental cone  (see \S\ref{subsec:Givental_cone}), $[y^{-k}]\cS^{-k} M_W(\theta)c$ lies in $H^*_T(W)(\!(z^{-1})\!)[\![\btheta]\!]$ when expanded at $z=\infty$. 
By Lemma \ref{lem:GKZ_type}(4), we have 
\[
[y^{-k}] \cS^{-k} M_W(\theta) c = i_{Z\times \PP^1 *} \left( 
\frac{\prod_{\nu=1}^{k-1}e_{-\lambda-\nu z}(\cN_{Z/X})}{\prod_{\nu=1}^k (\lambda+ \nu z)} 
e^{k z\partial_\lambda} c_\theta|_Z \right) + i_{X*}g 
\] 
for some $g \in H^*_T(X)(\!(z^{-1})\!)[\![\btheta]\!]$. The first formula follows from this and \eqref{eq:sfF_X_MW_c}. 

The calculation of $\sfF_\tX(M_W(\theta)c)|_{\QW=0}$ is similar. Using the fact that the $\hS$-series $\sfF_\tX(M_W(\theta) c)$ is supported on $C_{\tX,\N}^\vee$, we have  
\begin{align} 
\label{eq:sfF_tX_MW_c} 
\begin{split} 
\sfF_\tX(M_W(\theta) c)|_{\QW=0} 
& = \sum_{k\in \Z} S^{-k} \kappa_\tX( \cS^k M_W(\theta) c) \Bigr|_{\QW=0} \\ 
& = \kappa_\tX(c_\theta) + \sum_{k> 0} (y S^{-1})^k \kappa_\tX( [y^k] \cS^{k} M_W(\theta) c) 
\end{split} 
\end{align}  
Using the fact that $xy^{-1}$ is the class of an effective curve, we can see that the restrictions of $[y^k] \cS^kM_W(\theta)c$ (with $k>0$) to the fixed loci $X, Z, \tX$ are given by a vector of the form: 
\[
\begin{pmatrix} 
0 \\ 
\frac{\prod_{\nu=0}^{k-1} (\lambda -\nu z)}{\prod_{\nu=1}^k e_{-\lambda + \nu z}(\cN_{Z/X})}
e^{-kz \partial_\lambda}c_\theta|_Z\\ 
*
\end{pmatrix}.  
\]
Since $[y^k] \cS^k M_W(\theta) c$ lies in $H^*_T(W)(\!(z^{-1})\!)[\![\btheta]\!]$, Lemma \ref{lem:GKZ_type}(3) implies that  
\[
[y^k] \cS^k M_W(\theta) c = \hjmath_* \left(\frac{\prod_{\nu=1}^{k-1} (\lambda -\nu z)}{\prod_{\nu=1}^k e_{-\lambda + \nu z}(\cN_{Z/X})}
e^{-kz \partial_\lambda}i_Z^*c_\theta \right) + i_{\tX*} g 
\]
for some $g\in H^*_T(\tX)(\!(z^{-1})\!)[\![\btheta]\!]$. The second formula follows from this and \eqref{eq:sfF_tX_MW_c}.  
\end{proof} 

By using the commutative diagram \eqref{eq:FT_sfF}, $\tau(\theta)|_{\QW=\btheta=0} = O(\frq^{-1})$ (see \eqref{eq:initialcond_tau_v}),  $M_X(\tau(\theta))|_{\QW=\btheta=0}=\id+O(\frq^{-1})$ and Lemma \ref{lem:sfF_MW_c}, we have 
\begin{equation} 
\label{eq:image_FTX_c}
\FT_X\sphat(c)|_{\QW=\btheta=0}= \kappa_X(c) + O(\frq^{-1}). 
\end{equation} 
By the definition of $\scrF_{Z,j}$ in \S\ref{subsubsec:formal_asymptotics}, we have (see \eqref{eq:leadingterm_scrF}) 
\[
\frq^{\rho_Z/((r-1)z)}\scrF_{Z,j} (M_W(\theta) c) |_{\QW=\btheta=0} = q_{Z,j} e^{h_{Z,j}/z} 
\lambda_j^n (b  + O(\frq^{-\frac{1}{r-1}}))
\]
where we set $i_Z^*c = \lambda^n b + O(\lambda^{n-1})$ with $b \in H^*(Z)$. Therefore the commutative diagram \eqref{eq:FT_scrF}, $\sigma_j(\theta)|_{\QW=\btheta=0} = h_{Z,j}+O(\frq^{-\frac{1}{r-1}})$ (see \eqref{eq:initialcond_mirrormap}) and $M_Z(\sigma_j(\theta))|_{\QW=\btheta=0} = e^{h_{Z,j}/z}(\id + O(\frq^{-\frac{1}{r-1}}))$ imply that 
\begin{equation} 
\label{eq:image_FTZ_c}
\FT_{Z,j}\sphat(c)|_{\QW=\btheta=0} = q_{Z,j} \lambda_j^n (b + O(\frq^{-\frac{1}{r-1}})).  
\end{equation} 
Using $\kappa_X (c_i )= \phi_{X,i}$, $i_Z^*c_i = \phi_{X,i}|_Z$, $\kappa_X (c_{l,m}) = 0$, $i_Z^*c_{l,m} = (-1)^l \lambda^{l+1}\phi_{Z,m}$, \eqref{eq:image_FTX_c} and \eqref{eq:image_FTZ_c}, we obtain 
\begin{align*} 
\Phi(c_i)|_{\QW=\btheta=0} & = \left(\phi_{X,i} + O(\frq^{-1}), 
\left\{ q_{Z,j} (\phi_{X,i}|_Z + O(\frq^{-\frac{1}{r-1}})) \right\}_{0\le j\le r-2}\right) \\ 
a_l \frq^{-\frac{l+1}{r-1}}\Phi(c_{l,m})|_{\QW=\btheta=0} 
& = \left( O(\frq^{-\frac{r+l}{r-1}}), \left\{ q_{Z,j}  (\zeta^{-j(l+1)}\phi_{Z,m}+ O(\frq^{-\frac{1}{r-1}})) \right\}_{0\le j\le r-2}\right) 
\end{align*} 
where $a_l := (-1)^l e^{\frac{\pi \iu r}{r-1}(l+1)}$ and $\zeta = e^{2\pi\iu/(r-1)}$. 
Since $(\zeta^{-j(l+1)})_{0\le j,l\le r-2}$ is an invertible matrix, this shows that $\Phi(c_i), a_l \frq^{-\frac{l+1}{r-1}} \Phi(c_{l,m})$ form a basis of $\tau^*\QDM(X)^{\rm La}\oplus \bigoplus_{j=0}^{r-2} \varsigma_j^*\QDM(Z)^{\rm La}$ 
and thus $\Phi(c_i), \Phi(c_{l,m})$ also. 

\begin{lemma} 
\label{lem:inversion_Phi} 
Using the basis $\{c_i,c_{l,m}\}$, we identify $\Phi|_{\QW=\btheta=0}$ as an endomorphism of $H^*(X)\oplus H^*(Z)^{\oplus r-1}$ over the ring $\C[z](\!(\frq^{-1/\frs})\!)$. 
For $0\le l \le r-2$, we write $e_l \phi_{Z,m}$ for the basis element $\phi_{Z,m}$ in the $l$th factor of $H^*(Z)^{\oplus (r-1)}$.  
Under this notation, the inverse of $\Phi|_{\QW=\btheta = 0}$ is given as follows: 
\begin{align*} 
\Phi^{-1}(\phi_{X,i}) |_{\QW=\btheta=0} 
&= \left(\phi_{X,i}+O(\frq^{-1}),  
\left\{ - a_j \frq^{-1} \phi_{X,i}|_Z \delta_{j,r-2} + O(\frq^{-\frac{j+2}{r-1}})\right\}_{0\le j\le r-2}
\right) \\ 
\Phi^{-1}(e_l \phi_{Z,m}) |_{\QW= \btheta = 0} 
& = \left( O(\frq^{- \frac{r+2l}{2(r-1)}}), \left\{\frac{a_j\zeta^{l(j+1)}}{(r-1) q_{Z,l}} \frq^{-\frac{j+1}{r-1}}(\phi_{Z,m} + O(\frq^{-\frac{1}{r-1}})) \right\}_{0\le j\le r-2} \right) 
\end{align*} 
\end{lemma}
\begin{proof} 
With respect to the bases $\{c_i, a_l \frq^{-\frac{l+1}{r-1}} c_{l,m}\}$ and $\{\phi_{X,i}, q_{Z,j}e_j \phi_{Z,m}\}$, $\Phi|_{\QW=\btheta=0}$ is represented by the following block matrix: 
\begin{equation} 
\label{eq:Phi_at_QW=btheta=0}
\left( 
\begin{array}{c|ccc} 
\id_{H^*(X)} & & 0& \\ \hline 
\imath^* & & & \\ 
\vdots & & \left(\zeta^{-j(l+1)}\right)_{j,l} & \\ 
\imath^* & & &   
\end{array}
\right) 
+ 
\left( 
\begin{array}{c|ccc} 
O(\frq^{-1}) & \cdots & O(\frq^{-\frac{r+l}{r-1}}) & \cdots \\ \hline 
& & & \\ 
O(\frq^{-\frac{1}{r-1}}) & & O(\frq^{-\frac{1}{r-1}}) & \\ 
 & & & 
\end{array} 
\right) 
\end{equation} 
where $\left(\zeta^{-j(l+1)}\right)_{j,l}$ means the matrix whose $(j,l)$-block is $\zeta^{-j(l+1)} \id_{H^*(Z)}$, $\imath^* \colon H^*(X)\to H^*(Z)$ is the restriction and $O(\frq^{-\frac{r+l}{r-1}})$ at  the upper-right corner lies in the $l$th block of columns. If we denote this sum by $A+B$, we have 
\[
(A+B)^{-1} = A^{-1} - A^{-1} B A^{-1} + A^{-1} B A^{-1} BA^{-1} - \cdots 
\]
The leading term $A^{-1}$ can be computed as follows: 
\[
A^{-1} = \left( 
\begin{array}{c|ccc} 
\id_{H^*(X)} & & 0& \\ \hline 
0 & & & \\ 
\vdots & &\smash{\raisebox{-.5\normalbaselineskip}{$\left(\frac{1}{r-1}\zeta^{l(j+1)}\right)_{j,l} $}}& \\ 
0 & & & \\
-\imath^* & & &   
\end{array}
\right).  
\]
It follows that the remaining terms $A^{-1} B A^{-1}$, $A^{-1} B A^{-1}B A^{-1}, \dots$ are all of the form 
\[
\left( 
\begin{array}{c|ccc} 
O(\frq^{-1}) & \cdots & O(\frq^{-\frac{r+l}{r-1}}) & \cdots \\ \hline 
& & & \\ 
O(\frq^{-\frac{1}{r-1}}) & & O(\frq^{-\frac{1}{r-1}}) & \\ 
 & & & 
\end{array} 
\right). 
\]
The inverse $(A+B)^{-1}$ is the sum of the above two. 
The lemma follows easily from this. 
\end{proof}

\subsubsection{That $\bPsi=\Phi \circ (\FT_\tX\sphat)^{-1}$ preserves the pairing} 
It is clear from Theorem \ref{thm:Fourier_isom}, Proposition \ref{prop:FTX_completion} and Corollary \ref{cor:FT_Zj} that $\bPsi$ commutes with the connection. It remains to show that $\bPsi$ preserves the pairing. 
First we compute the images of $c_i, c_{l,m}$ in $\ttau^*\QDM(\tX)^{\rm ext}$ under $\FT_\tX\sphat$ along $\QW=\btheta=0$. 

\begin{lemma} 
\label{lem:image_under_FTtX_limit} 
Suppose that $c\in H^*_T(W)$ is such that $i_Z^*c\in H^*(Z)[\lambda]$ is of degree $\le r-1$ as a polynomial of $\lambda$.  Then $\FT_\tX\sphat(c)|_{\QW=\btheta=0} = \kappa_\tX(c)$. 
In particular, $\FT_\tX\sphat(c_i)|_{\QW=\btheta=0} = \varphi^*\phi_{X,i}$ and $\FT_\tX\sphat(c_{l,m})|_{\QW=\btheta=0} = \jmath_*(p^l \pi^*\phi_{Z,m})$ for $0\le l\le r-2$. 
\end{lemma} 
\begin{proof} 
Recall the formula for $\sfF_\tX(M_W(\theta) c)|_{\QW=\btheta=0}$ in Lemma \ref{lem:sfF_MW_c}.  
Since the degree of $i_Z^*c$ as a polynomial of $\lambda$ is less than or equal to $r-1$, we see that the quantity
\[
\frac{\prod_{\nu=1}^{k-1} ([D] -\nu z)}{\prod_{\nu=1}^k e_{-[D] + \nu z}(\cN_{Z/X})}
\left[e^{-kz \partial_\lambda}i_Z^*c\right]_{\lambda=[D]} 
\] 
appearing there is of order $O(z^{k-1-rk+r-1})$ as $z\to \infty$. Since $k-1-rk+r-1 = r-2 - (r-1) k \le -1$, we have $\sfF_\tX(M_W(\theta)c)|_{\QW=\btheta=0}= \kappa_\tX(c) + O(z^{-1})$. 
Because $M_\tX(\ttau(\theta)) = \id + O(z^{-1})$, it follows from the commutative diagram \eqref{eq:FT_sfF} that $\FT_\tX\sphat(c)|_{\QW=\btheta=0} = \kappa_\tX(c)$. 
\end{proof} 

By Lemma \ref{lem:image_under_FTtX_limit}, $\bPsi^{-1}(\phi_{X,i})|_{\QW = \btheta=0}, \bPsi^{-1}(e_l \phi_{Z,m})|_{\QW=\btheta=0}$ are presented by the same vectors as $\Phi^{-1}(\phi_{X,i})|_{\QW=\btheta=0}, \Phi^{-1}(e_l \phi_{Z,m})|_{\QW=\btheta=0}$ in Lemma \ref{lem:inversion_Phi}, under the isomorphism $H^*(\tX) \cong H^*(X) \oplus H^*(Z)^{\oplus (r-1)}$ in the first line of \eqref{eq:cohomology_bu}. 
Using this fact and the fact that the two summands $H^*(X)$, $H^*(Z)^{\oplus (r-1)}$ of $H^*(\tX)$ are orthogonal to each other (for the Poincar\'e pairing) and that 
\begin{align*} 
\int_\tX \jmath_*(p^j \pi^*\gamma) \cup \jmath_*(p^{j'} \pi^*\gamma') 
& = \int_D -p^{j+j'+1} \pi^*(\gamma \cup \gamma') \\
& = \begin{cases} 
0 & \text{if $j+j'< r-2$;} \\ 
-\int_Z \gamma\cup \gamma' & \text{if $j+j' =r-2$} 
\end{cases} 
\end{align*} 
for $\gamma, \gamma' \in H^*(Z)$, 
we compute 
\begin{align*} 
P_\tX(\bPsi^{-1}(\phi_{X,i}), \bPsi^{-1}(\phi_{X,i'})) \bigr|_{\QW=\btheta=0} 
& = P_X (\phi_{X,i}, \phi_{X,i'}) + O(\frq^{-1}), \\ 
P_\tX(\bPsi^{-1}(\phi_{X,i}), \bPsi^{-1}(e_l \phi_{Z,m})) \bigr|_{\QW=\btheta=0} 
& = O(\frq^{-\frac{r}{2(r-1)}}), \\ 
P_\tX(\bPsi^{-1}(e_l \phi_{Z,m}), \bPsi^{-1}(e_{l'} \phi_{Z,m'}))\bigr|_{\QW=\btheta=0} 
& = \delta_{l,l'} P_Z(\phi_{Z,m}, \phi_{Z,m'}) + O(\frq^{-\frac{1}{r-1}}). 
\end{align*} 
This shows that $\bPsi^{-1}|_{\QW=\btheta=0}$ preserves the pairing at the limit $\frq = \infty$. 
We consider the pullback $P_1 := (\bPsi^{-1})^* P_\tX$ of the pairing $P_\tX$ via $\bPsi^{-1}$ and compare it with $P_2 := P_X \oplus P_Z^{\oplus (r-2)}$. Both $P_1$ and $P_2$ are flat pairings on 
\[
\cE:=\tau^*\QDM(X)^{\rm La} \oplus \bigoplus_{j=0}^{r-2} \varsigma_j^* \QDM(Z)^{\rm La},
\] 
regular and non-degenerate at the limit $\frq=\infty, \QW=\btheta=0$ (with respect to a standard basis) and $P_1|_{\frq=\infty,\QW=\btheta=0} = P_2|_{\frq=\infty,\QW=\btheta=0}$ by the above computation. 
Then $P_2^{-1} P_1$ defines a flat endomorphism of $\cE$. We claim that a flat endomorphism of $\cE|_{\QW=\btheta=0}$ which is the identity at $\frq=\infty$ is the identity. 
The quantum connection on $\cE|_{\QW=\btheta=0}$ along the $\frq$-direction is given by (see \eqref{eq:pull-back_qconn}, \eqref{eq:Z_qconn_La}, \eqref{eq:varsigma}) 
\begin{align*} 
\tau^* \nabla_{\frq\partial_\frq}|_{\QW=\btheta=0} & = \frq\partial_\frq + z^{-1} \kappa_X(-\lambda) +  z^{-1} \left[\frq\partial_\frq \tau(0)\star_{\tau(0)}\right]_{\QW=0} \\ 
& = \frq\partial_\frq + z^{-1} O(\frq^{-1}) \\ 
\varsigma_j^* \nabla_{\frq \partial_\frq}|_{\QW=\btheta=0}& =  \frq\partial_\frq + z^{-1} \kappa_Z(-\lambda) + z^{-1} \left(-\lambda_j + \left[\frq\partial_\frq \sigma_j(0)\star_{\sigma_j(0)}\right]_{\QW=0}\right) \\
& = \frq\partial_\frq - z^{-1} \left(\lambda_j + \frac{\rho_Z}{r+1} + O(\frq^{-\frac{1}{r-1}}) \right).
\end{align*} 
Here we used $\frq \partial_\frq = - \lambda \hS\partial_\hS$ and the properties  \eqref{eq:initialcond_tau_v}, \eqref{eq:initialcond_mirrormap} for $\tau$ and $\varsigma_j=-(r-1)\lambda_j+\sigma_j$; note that the quantities $\rho_Z$, $O(\frq^{-1})\in H^*(X)[\frq^{-1}]$, $O(\frq^{-\frac{1}{r+1}})\in H^*(Z)[\frq^{-\frac{1}{r-1}}]$ act by the cup product (see Remark \ref{rem:QW=0_subtlety}). 
Write $\frr = \frq^{-\frac{1}{r-1}}$. A flat endomorphism $F(\frr)$ of $\cE$ satisfies a differential equation of the form: 
\[
(\frr \partial_\frr + \frr^{-1}  \ad(D) + \ad(N_0) + \ad(N_1) \frr+ \ad(N_2) \frr^2 + \cdots) F(\frr) = 0  
\]
where $D,N_i \in \End(H^*(X)\oplus H^*(Z)^{\oplus (r-1)})$ are mutually commuting operators such that $D$ is semisimple and $N_0$ is nilpotent. 
Suppose that $F(0)= \id$. Expanding $F(\frr) = F_0 + F_1 \frr + F_2 \frr^2 + \cdots$ with $F_0=\id$, we find 
\begin{equation} 
\label{eq:recursive_equation}
\tag{$*_n$} 
\ad(D) F_{n+1} + (n+\ad(N_0))F_n + \ad(N_1) F_{n-1} + \ad(N_2) F_{n-2} + \cdots + \ad(N_n) F_0 = 0 
\end{equation} 
for $n\ge 0$. Since $\ad(D)F_0 = \ad(N_i) F_0= 0$, we have $\ad(D) F_1 =0$ by (\hyperref[eq:recursive_equation]{$*_0$}). Then (\hyperref[eq:recursive_equation]{$*_1$}) implies $\ad(D)^2 F_2 =0$. Since $D$ is semisimple, $\ad(D)$ is also semisimple and $\ad(D)F_2= 0$. 
By (\hyperref[eq:recursive_equation]{$*_1$}) again, we have $(1+\ad(N_0)) F_1=0$. The nilpotence of $N_0$ implies that $F_1=0$.  
In this way we can show inductively that $F_1=F_2= \cdots =0$. The claim follows and $P_1|_{\QW= \btheta=0} = P_2|_{\QW=\btheta=0}$. 
The quantum connection of $\cE$ along the $(\QW,\btheta)$-direction has at worst logarithmic singularities with nilpotent residues at $\QW=\btheta=0$ (see \eqref{eq:pull-back_qconn}, \eqref{eq:Z_qconn_La}, \eqref{eq:initialcond_tau_v}, \eqref{eq:initialcond_mirrormap}). By expanding $P_2^{-1} P_1$ in power series in $\QW, \btheta$, we can similarly show that $P_2^{-1} P_1= \id$. This completes the proof of Theorem \ref{thm:decomposition}.   

\begin{remark} 
\label{rem:global}
We describe a geography of the global ``K\"ahler moduli space'' without giving precise definitions on formal geometry. 
The quantum $D$-module $\ttau^*\QDM(\tX)^{\rm ext}$ of $\tX$ can be viewed as a connection over the ``graded formal superscheme''  $\Spf(\C[z][\![C_{\tX,\N}^\vee,\btheta]\!])$ and that  $\tau^*\QDM(X)^{\rm ext}$ of $X$ can be viewed as a connection over $\Spf(\C[z][\![C_{X,\N}^\vee,\btheta]\!])$. 
These base spaces are glued to the formal ``toric'' scheme 
\[
\frM = \Spf(\C[z][\![C_{\tX,\N}^\vee,\btheta]\!]) \cup \Spf(\C[z][\![C_{X,\N}^\vee,\btheta]\!])
\] 
associated with the fan structure on the $T$-ample cone $\overline{C_T(W)}$ (see \S\ref{subsec:T-ample_cone} and Figure \ref{fig:global}). We compare these connections over an (\'etale) open subset $\Spf (\C[z](\!(\frq^{-1/\frs})\!)[\![\QW,\btheta]\!])$ of $\frM$ and find that the difference equals  $\bigoplus_{j=0}^{r-2} \varsigma_j^* \QDM(Z)^{\rm La}$. All these connections arise from the equivariant quantum $D$-module $\QDM_T(W)$ of $W$, which originally lives on the affinization $\frM_0 = \Spf (\C[z][\![\NEN^T(W),\btheta]\!])$ of $\frM$ by Proposition \ref{prop:equiv_QDM_module_over} . 
\end{remark}

\subsection{Restriction to a finite-dimensional slice} 
\label{subsec:restriction_slice} 
In this section, we restrict the parameter $\theta$ to lie in the finite-dimensional slice $\sfH \subset H^*_T(W)$ spanned by the classes $c_i, c_{l,m}$ in \eqref{eq:basis_QDMTW} with $0\le l \le r-2$. By restricting the decomposition $\bPsi$ in Theorem \ref{thm:decomposition} to this subspace $\sfH$, we obtain a desired decomposition for $\QDM(\tX)$. We also show that $x$ and $y$ factor out from the final result. 

We write $\vartheta=\{\vartheta^\alpha\}$ for the coordinates on $\sfH$ dual to the basis $\{c_\alpha\} = \{c_i,c_{l,m}\}$: we may regard $\{\vartheta^\alpha\}$ as a subset of the infinitely many coordinates $\btheta = \{\theta^{i,k}\}$ on $H^*_T(W)$. We also denote by $\vartheta$ the corresponding point $\sum_\alpha \vartheta^\alpha c_\alpha$ on $\sfH$. 
The ``mirror maps'' $\ttau(\theta)$, $\tau(\theta)$ in \S\ref{subsec:FT_QDM} (arising from Corollary \ref{cor:Fourier_transform_JW_GIT}) restricted to $\sfH$ satisfies 
\begin{equation}
\label{eq:ttau_tau_fin}
\begin{aligned}  
\ttau(\vartheta) &\in H^*(\tX)[\![C_{\tX,\N}^\vee,\vartheta]\!], &  
\ttau(0)|_{\QW= 0} &\in \frq H^*(\tX)[\frq], \\
\tau(\vartheta) &\in H^*(X)[\![C_{X,\N}^\vee, \vartheta]\!], &  
\tau(0)|_{\QW=0} &\in \frq^{-1} H^*(X)[\frq^{-1}],
\end{aligned} 
\end{equation} 
where we take the restriction to $\QW=0$ as power series in $\frq^\pm, \QW$. 
The ``mirror map'' $\varsigma_j(\theta) = \sigma_j (\theta)- (r-1)\lambda_j$ in \S\ref{subsubsec:projection_Z} and \eqref{eq:varsigma} (arising from Corollary \ref{cor:Fourier_transform_JW}) restricted to $\sfH$ satisfies 
\begin{align}
\label{eq:varsigma_fin} 
\begin{split}  
\varsigma_j(\vartheta) & \in H^*(Z)(\!(\frq^{-\frac{1}{r-1}})\!)[\![\QW,\vartheta]\!] \\ 
\varsigma_j(0)|_{\QW=0} & \in -(r-1) \lambda_j +  h_{Z,j} + \frq^{-\frac{1}{r-1}}H^*(Z)[\frq^{-\frac{1}{r-1}}] 
\end{split} 
\end{align} 
where recall that $\lambda_j = e^{-\frac{2\pi \iu}{r-1} (j+ \frac{r}{2})}\frq^{\frac{1}{r-1}}$.  
\begin{lemma} 
\label{lem:asymptotics_mirrormaps} 
We write $f|_{\QW=0}$ for the restriction of $f\in \C(\!(\frq^{-1/\frs})\!)[\![\QW]\!]$ to $\QW=0$ as power series in $\frq^{\pm 1/\frs}$ and $\QW$. We have 
\begin{align}
\label{eq:asymptotics_mirrormaps}
\begin{split}   
\ttau(\vartheta)|_{\QW=0} & \in \kappa_\tX (\vartheta) + (\vartheta)^2 H^*(\tX)[\frq][\![\vartheta]\!] \\
\tau(\vartheta)|_{\QW=0} & \in \kappa_X(\vartheta) + \frq^{-1}[Z] + (\frq^{-1}\vartheta, \frq^{-2}) H^*(X)[\frq^{-1}][\![\vartheta]\!] \\  
\partial_{\vartheta^\alpha}\varsigma_j(\vartheta) |_{\QW=\vartheta=0} & 
\in 
\begin{cases}  
\imath^*\phi_{X,i} +\frq^{-\frac{1}{r-1}} H^*(Z)[\frq^{-\frac{1}{r-1}}] & \text{if $c_\alpha = c_i$}; \\  
(-1)^l \lambda_j^{l+1} \left(\phi_{Z,m} + \frq^{-\frac{1}{r-1}} H^*(Z)[\frq^{-\frac{1}{r-1}}]\right) & \text{if $c_\alpha = c_{l,m}$}
\end{cases} 
\end{split} 
\end{align} 
where $(\vartheta) \subset \C[\frq][\![\vartheta]\!]$ denotes the ideal generated by $\vartheta^\alpha$, $ (\frq^{-1}\vartheta, \frq^{-2}) \subset \C[\frq^{-1}][\![\vartheta]\!]$ denotes the ideal generated by $\frq^{-1}\vartheta^\alpha, \frq^{-2}$ and $\imath\colon Z\to X$ is the inclusion. Note that these conditions refine \eqref{eq:ttau_tau_fin}, \eqref{eq:varsigma_fin}.  
\end{lemma} 
\begin{proof} 
We showed in the proof of Lemma \ref{lem:image_under_FTtX_limit} that $\sfF_\tX(J_W(\vartheta))|_{\QW=\vartheta=0} = 1 + O(z^{-2})$ (set $c=1$ in the second formula of Lemma \ref{lem:sfF_MW_c}). Therefore $\tv(\vartheta)|_{\QW=\vartheta=0} = M_\tX(\ttau(\vartheta))^{-1}(1+ O(z^{-2}))|_{\QW=\vartheta=0}$ by \eqref{eq:Fourier_JW_for_X_tX}. Since $M_\tX(\ttau(\vartheta))^{-1}1 = 1 - \ttau(\vartheta)/z + O(z^{-2})$, we know that $\ttau(\vartheta)|_{\QW=\vartheta=0} = 0$ by comparing the coefficients of $z^{-1}$. 
Also, we have 
\begin{equation} 
\label{eq:FT_c_cov_der} 
\FT_\tX\sphat(c_\alpha) = \FT_\tX\sphat(z \nabla_{\vartheta^\alpha} 1 ) = z\nabla_{\vartheta^\alpha} (\FT_\tX\sphat(1)) 
=\left( z \partial_{\vartheta^\alpha} + (\partial_{\vartheta^\alpha}\ttau(\vartheta))\star_{\ttau(\vartheta)} \right) \FT_\tX\sphat(1) 
\end{equation} 
where we restrict $\theta$ to lie in $\sfH$. 
Setting $\QW=\vartheta=z=0$ and using Lemma \ref{lem:image_under_FTtX_limit}, we have $\partial_{\vartheta^\alpha}\ttau(\vartheta)|_{\QW=\vartheta=0} = \kappa(c_\alpha)$. These computations show the first line of \eqref{eq:asymptotics_mirrormaps}.

Setting $c=1$ in Lemma \ref{lem:sfF_MW_c}, we find 
\[
\sfF_X(J_W(\vartheta))|_{\QW=0} = e^{\kappa_X(\vartheta)/z} + z^{-1}\frq^{-1} ( [Z] + O(\vartheta)) + O(\frq^{-2}). 
\]
Since this equals $M_X(\tau(\vartheta)) v(\vartheta)|_{\QW=0}=e^{\tau(\vartheta)/z} v(\vartheta)|_{\QW=0}$ by \eqref{eq:Fourier_JW_for_X_tX}, we must have  $\tau(\vartheta)|_{\QW=0}=\kappa_X(\vartheta)+ \frq^{-1}[Z] + \frq^{-1} O(\vartheta) + O(\frq^{-2})$. The second line of \eqref{eq:asymptotics_mirrormaps} follows. 

Similarly to the above computation \eqref{eq:FT_c_cov_der}, we have 
\[
\FT_{Z,j}\sphat(c_\alpha) = 
\left( z \partial_{\vartheta^\alpha} + (\partial_{\vartheta^\alpha}\sigma_j(\vartheta)) \star_{\sigma_j(\vartheta)} \right) \FT_{Z,j}\sphat(1).  
\]
Using \eqref{eq:image_FTZ_c} and setting $\QW=\vartheta=z=0$, we arrive at the third line of \eqref{eq:asymptotics_mirrormaps}. 
\end{proof} 

Lemma \ref{lem:asymptotics_mirrormaps} implies that the change of variables $\sfH \to H^*(\tX)$, $\vartheta \mapsto \ttau(\vartheta)$ is (formally) invertible over $\C(\!(\frq^{-1})\!)[\![\QW]\!]$. 
We can also see that the map $\sfH \to H^*(X) \oplus H^*(Z)^{\oplus (r-1)}$, $\vartheta \mapsto (\tau(\vartheta), \{\varsigma_j(\vartheta)\}_{0\le j\le r-2})$ is invertible over $\C(\!(\frq^{-\frac{1}{r-1}})\!)[\![\QW]\!]$ as its Jacobian matrix at $\QW=\vartheta=0$ is invertible (it is equivalent to a matrix of the form \eqref{eq:Phi_at_QW=btheta=0}). 
Writing $\ttau \mapsto \vartheta= \vartheta(\ttau)$ for the inverse of the map $\vartheta \mapsto \ttau(\vartheta)$, we consider the composition  
\begin{equation} 
\label{eq:composed_change_of_variables} 
\tau(\ttau) := \tau(\vartheta(\ttau)), \qquad 
\varsigma_j(\ttau) := \varsigma_j(\vartheta(\ttau)) 
\end{equation} 
We have $\tau(\ttau) \in H^*(X)(\!(\frq^{-1})\!)[\![\QW,\ttau]\!]$, $\varsigma_j(\ttau) \in H^*(Z)(\!(\frq^{-\frac{1}{r-1}})\!)[\![\QW,\ttau]\!]$ and 
\begin{align} 
\label{eq:leadingterms_composedmirrormaps} 
\begin{split} 
\tau(\ttau)|_{\QW=\ttau=0} & = \frq^{-1} [Z] + O(\frq^{-2}) \\ 
\varsigma_j(\ttau)|_{\QW= \ttau=0} & = -(r-1) \lambda_j + h_{Z,j} + O(\frq^{-\frac{1}{r-1}})
\end{split}  
\end{align} 
by \eqref{eq:varsigma_fin}, \eqref{eq:asymptotics_mirrormaps}.  
Using $\{\kappa_\tX(c_\alpha)\}$ as a basis of $H^*(\tX)$ and writing $\ttau^\alpha$ for the dual coordinates, Lemma \ref{lem:asymptotics_mirrormaps} implies that the Jacobian matrix is given by 
\begin{align}
\label{eq:Jacobianmatrix} 
\begin{split} 
\left.\parfrac{\tau}{\ttau^\alpha}(\ttau)\right|_{\QW=\ttau=0} 
& = \kappa_X(c_\alpha)+O(\frq^{-1})
= \begin{cases} 
\phi_{X,i} + O(\frq^{-1}) & \text{if $c_\alpha=c_i$;} \\ 
O(\frq^{-1}) & \text{if $c_\alpha = c_{l,m}$;} 
\end{cases} 
\\
\left. \parfrac{\varsigma_j}{\ttau^\alpha}(\ttau) \right|_{\QW=\ttau=0} 
& = \begin{cases} 
\imath^*\phi_{X,i} + O(\frq^{-\frac{1}{r-1}}) & \text{if $c_\alpha= c_i$;}\\ 
(-1)^l \lambda_j^{l+1} (\phi_{Z,m} + O(\frq^{-\frac{1}{r-1}})) & \text{if $c_\alpha = c_{l,m}$.}
\end{cases}
\end{split}  
\end{align}

Let $\QDM(\tX)^{\rm La} = H^*(X)[z](\!(\frq^{-1/\frs})\!)[\![\QW,\ttau]\!]$ denote the formal Laurent version of the quantum $D$-module of $\tX$, induced from $\QDM(\tX)^{\rm ext}$ in \S\ref{subsec:extended_QDM} by the further extension $\C[z][\![C_{\tX,\N}^\vee,\ttau]\!] \subset \C[z](\!(\frq^{-1/\frs})\!)[\![\QW,\ttau]\!]$ of rings.  
Restricting the decomposition $\bPsi$ in Theorem \ref{thm:decomposition} to the finite-dimensional slice $\sfH$, we obtain an isomorphism $\Psi$ of $\C[z](\!(\frq^{-1/\frs})\!)[\![\QW,\ttau]\!]$-modules with connections:   
\begin{equation} 
\label{eq:Psi_restricted_to_H}
\Psi \colon \QDM(\tX)^{\rm La} \to \tau^*\QDM(X)^{\rm La} \oplus \bigoplus_{j=0}^{r-2} \varsigma_j^* \QDM(Z)^{\rm La}  
\end{equation} 
where $\tau^*\QDM(X)^{\rm La}$, $\varsigma_j^*\QDM(Z)^{\rm La}$ are defined to be the modules $H^*(X)[z](\!(\frq^{-1/\frs})\!)[\![\QW,\ttau]\!]$, $H^*(Z)[z](\!(\frq^{-1/\frs})\!)[\![\QW,\ttau]\!]$ equipped with the pullbacks of the quantum connection by the maps $\ttau \mapsto \tau(\ttau)$, $\ttau \mapsto \varsigma_j(\ttau)$ in \eqref{eq:composed_change_of_variables} respectively. 
These pullbacks are well-defined due to the String and Divisor equations. 
The well-definedness can also be explained by the fact that the structures of $\QDM(X)^{\rm La}$ and $\QDM(Z)^{\rm La}$ are reduced to $\C[z][\![\QW,\tau]\!]$ and the ring $R$ in Remark \ref{rem:QDMZLa}, respectively, and that the pullbacks $\tau^* \colon \C[z][\![\QW,\tau]\!] \to \C[z](\!(\frq^{-1/\frs})\!)[\![\QW,\ttau]\!]$, $\varsigma_j^* \colon R\to \C[z](\!(\frq^{-1/\frs})\!)[\![\QW,\ttau]\!]$ are well-defined. 


Recall from \S\ref{subsec:N1_W} that the Novikov variable $\QW$ of $W$ is composed of the variables $Q,x,y$. We observe that $x$ and $y$ split off from the decomposition \eqref{eq:Psi_restricted_to_H}. 

\begin{proposition}
\label{prop:xy} 
The maps $\ttau\mapsto \tau(\ttau)$, $\ttau \mapsto \varsigma_j(\ttau)$ in \eqref{eq:composed_change_of_variables} are independent of $x$ and $y$ as elements of $H^*(X)(\!(\frq^{-1})\!)[\![\QW,\ttau]\!]$, $H^*(Z)(\!(\frq^{-\frac{1}{r-1}})\!)[\![\QW,\ttau]\!]$ respectively. 
The matrix entries of the isomorphism $\Psi$ in \eqref{eq:Psi_restricted_to_H} with respect to bases of $H^*(\tX)$, $H^*(X)$, $H^*(Z)$ are also independent of $x$ and $y$ as elements of  $\C[z](\!(\frq^{-1/\frs})\!)[\![\QW,\ttau]\!]$. 
\end{proposition} 
\begin{proof} 
Let $x\partial_x$, $y\partial_y$ denote the partial derivatives with respect to the coordinate system $(Q,x,y, \frq = yS^{-1})$. They are given by $x\partial_x = (0,1,0,0) \hS \partial_{\hS} = [X] \hS \partial_{\hS}$ and $y\partial_y = (0,0,1,1) \hS \partial_{\hS} = (\lambda- [\hD]) \hS \partial_{\hS}$ in the notation of \S\ref{subsec:N1_W} and \S\ref{subsec:extended_QDM}. Let $\nabla^{\tX}$, $\nabla^{X}$, $\nabla^{Z,j}$ denote the quantum connections on $\QDM(\tX)^{\rm La}$,  $\tau^*\QDM(X)^{\rm La}$, $\varsigma_j^* \QDM(Z)^{\rm La}$ respectively.  Using the fact that the images of $[X], \lambda -[\hD]\in H^2_T(W)$ under the maps $\kappa_X$, $\kappa_\tX$, $\kappa_Z$ are all zero (see \S\ref{subsec:Kirwan}, \eqref{eq:pull-back_qconn} and \S\ref{subsubsec:projection_Z}), we have 
\begin{align*} 
\nabla^\tX_{x\partial_x} & = x\partial_x, &  \nabla^X_{x\partial_x} & = x\partial_x + z^{-1}(x \partial_x \tau(\ttau))\star_{\tau(\ttau)}, & 
\nabla^{Z,j}_{x\partial_x} & = x\partial_x + z^{-1} (x\partial_x \varsigma_j(\ttau))\star_{\varsigma_j(\ttau)} \\ 
\nabla^\tX_{y\partial_y} & = y\partial_y, &  \nabla^X_{y\partial_y} & = y\partial_y + z^{-1}(y \partial_y \tau(\ttau))\star_{\tau(\ttau)}, & 
\nabla^{Z,j}_{y\partial_y} & = y\partial_y + z^{-1} (y\partial_y \varsigma_j(\ttau))\star_{\varsigma_j(\ttau)} 
\end{align*} 
with respect to the standard trivialization. Therefore $z \nabla^{\tX}_{x\partial_x}$, $z \nabla^\tX_{y\partial_y}$ vanish along $z=0$. Since the isomorphism $\Psi$ intertwines the connections and $\tau(\ttau)$, $\varsigma_j(\tau)$ are independent of $z$, it follows that $\tau(\ttau)$, $\varsigma_j(\tau)$ are independent of $x$ and $y$. Therefore, the isomorphism $\Psi$ commutes with $x\partial_x$ and $y\partial_y$. This implies the conclusion. 
\end{proof} 

The following lemma is immediate from the formula $\QW^{\delta} = Q^{\pr_{1*} \hvarphi_*\delta} x^{[X] \cdot \delta} y^{-[\hD] \cdot \delta}$ for $\delta \in \NEN(W)$ (see \S\ref{subsec:N1_W}). 

\begin{lemma} 
\label{lem:xy} 
Let $K$ be a graded ring. The elements of $K[\![\QW]\!]=K[\![\NEN(W)]\!]$ that are independent of $x,y$ are precisely the elements of the subring $K[\![Q]\!] = K[\![\NEN(X)]\!] \subset K[\![\QW]\!]$, where $Q^d$ is identified with $\QW^{i_{X*}d}$. 
\end{lemma} 

By Proposition \ref{prop:xy} and Lemma \ref{lem:xy}, we have that $\tau(\ttau) \in H^*(X)(\!(\frq^{-1})\!)[\![Q,\ttau]\!]$ and $\varsigma_j(\ttau) \in H^*(Z) (\!(\frq^{-\frac{1}{r-1}})\!)[\![Q,\ttau]\!]$ and that the decomposition $\Psi$ in \eqref{eq:Psi_restricted_to_H} is reduced to the smaller ring $\C[z](\!(\frq^{-1/\frs})\!)[\![Q,\ttau]\!]$.  In light of this, we define ``smaller'' formal Laurent versions of the quantum $D$-modules as follows: 
\begin{align}
\label{eq:QDMla}
\begin{split}  
\QDM(\tX)^{\rm la} & := H^*(\tX)[z](\!(\frq^{-1/\frs})\!)[\![Q,\ttau]\!], \\
\QDM(X)^{\rm la} & := H^*(X)[z](\!(\frq^{-1/\frs})\!)[\![Q,\tau]\!] \\ 
\QDM(Z)^{\rm la} & := H^*(Z)[z](\!(\frq^{-1/\frs})\!)[\![Q,\sigma]\!] \\ 
\tau^*\QDM(X)^{\rm la} & := H^*(X)[z](\!(\frq^{-1/\frs})\!)[\![Q,\ttau]\!], \\ 
\varsigma_j^* \QDM(Z)^{\rm la} & :=  H^*(Z)[z](\!(\frq^{-1/\frs})\!)[\![Q,\ttau]\!]. 
\end{split} 
\end{align} 
Here, $\QDM(\tX)^{\rm la}$ is  induced from $\QDM(\tX)$ by the ring extension 
\begin{equation} 
\label{eq:extension_QDMtXla}
\C[z][\![\tQ,\ttau]\!] \hookrightarrow \C[z](\!(\frq^{-1/\frs})\!)[\![Q,\ttau]\!], \quad 
\tQ^\td \mapsto  Q^{\varphi_*\td} \frq^{-[D]\cdot \td}
\end{equation} 
that fits into the diagram\footnote{Recall that $\QDM(\tX)^{\rm ext}$ was induced from $\QDM(\tX)$ by the ring extension in the left arrow (\S\ref{subsec:extended_QDM}) and $\QDM(\tX)^{\rm La}$ was induced from $\QDM(\tX)^{\rm ext}$ by the ring extension in the bottom arrow.}   
\[
\xymatrix{
\C[z][\![\tQ,\ttau]\!] \ar@{^{(}->}[r] \ar@{^{(}->}[d]^{\kappa_\tX^*} & \C[z](\!(\frq^{-1/\frs})\!)[\![Q,\ttau]\!] \ar@{^{(}->}[d]  \\ 
\C[z][\![C_{\tX,\N}^\vee,\ttau]\!] \ar@{^{(}->}[r] &  \C[z](\!(\frq^{-1/\frs})\!)[\![\QW,\ttau]\!]. 
}
\]
Similarly, $\QDM(X)^{\rm la}$ is induced from $\QDM(X)$ by the ring extension 
\begin{equation}
\label{eq:extension_QDMXla}  
\C[z][\![Q,\tau]\!]\hookrightarrow \C[z](\!(\frq^{-1/\frs})\!)[\![Q,\tau]\!], \quad Q^d \mapsto Q^d 
\end{equation} 
that fits into the diagram 
\[
\xymatrix{
\C[z][\![Q,\tau]\!] \ar@{^{(}->}[r] \ar@{^{(}->}[d]^{\kappa_X^*} & \C[z](\!(\frq^{-1/\frs})\!)[\![Q,\tau]\!] \ar@{^{(}->}[d]  \\ 
\C[z][\![C_{X,\N}^\vee,\tau]\!] \ar@{^{(}->}[r] &  \C[z](\!(\frq^{-1/\frs})\!)[\![\QW,\tau]\!], 
}  
\]
and $\tau^*\QDM(X)^{\rm la}$ is its pullback by $\tau=\tau(\ttau)$. Likewise, $\QDM(Z)^{\rm la}$ is induced from $\QDM(Z)$ by the ring extension 
\begin{equation} 
\label{eq:extension_QDMZla} 
\C[z][\![Q_Z,\sigma]\!] \to \C[z](\!(\frq^{-1/\frs})\!)[\![Q,\sigma]\!], \quad 
Q_Z^d\mapsto Q^{\imath_*d} \frq^{-\rho_Z \cdot d/(r-1)}, 
\end{equation} 
through which the homomorphism \eqref{eq:extension_Novikov_Z} factors, and $\varsigma_j^*\QDM(Z)^{\rm la}$ is its pullback by $\sigma = \varsigma_j(\ttau)$. The quantum connections on these modules are naturally induced by these ring extensions and pullbacks, where $\rho_Z = c_1(\cN_{Z/X})$. 
Explicitly, the quantum connection on $\QDM(\tX)^{\rm la}$ is given by 
\begin{align}
\label{eq:qconn_QDMtXla}  
\begin{split} 
\nabla_{\ttau^\alpha} &= \partial_{\ttau^\alpha} + z^{-1} (\phi_{\tX,\alpha} \star_\ttau), \\ 
\nabla_{z\partial_z} & = z\partial_z - z^{-1} (E_\tX \star_{\ttau}) + \mu_\tX,\\ 
\nabla_{\xi Q\partial_Q} & = \xi Q\partial_Q + z^{-1} (\varphi^*\xi)\star_\ttau, \\ 
\nabla_{\frq \partial_\frq} & = \frq \partial_\frq  - z^{-1}[D]\star_\ttau 
\end{split} 
\intertext{where $\phi_{\tX,\alpha}$ is a basis of $H^*(\tX)$ and $\ttau^\alpha$ is the coordinate dual to $\phi_{\tX,\alpha}$, the quantum connection on $\tau^*\QDM(X)^{\rm la}$ is given by}
\label{eq:qconn_QDMXla}  
\begin{split} 
\nabla_{\ttau^\alpha} &= \partial_{\ttau^\alpha} + z^{-1} (\partial_{\ttau^\alpha}\tau(\ttau)) \star_{\tau(\ttau)} \\ 
\nabla_{z\partial_z} & = z\partial_z - z^{-1} (E_X \star_{\tau(\ttau)}) + \mu_X\\ 
\nabla_{\xi Q\partial_Q} & = \xi Q\partial_Q + z^{-1} (\xi\star_{\tau(\ttau)}) + z^{-1} (\xi Q\partial_Q \tau(\ttau))\star_{\tau(\ttau)},  \\
\nabla_{\frq \partial_\frq} & = \frq \partial_\frq + z^{-1} (\frq \partial_\frq \tau(\ttau))\star_{\tau(\ttau)} 
\end{split} 
\intertext{and the quantum connection on $\varsigma_j^*\QDM(Z)^{\rm la}$ is given by}
\label{eq:qconn_QDMZla} 
\begin{split} 
\nabla_{\ttau^\alpha} & = \partial_{\ttau^\alpha} + z^{-1} (\partial_{\ttau^\alpha}\varsigma_j(\ttau)) \star_{\tau(\ttau)} \\ 
\nabla_{z\partial_z} & = z\partial_z - z^{-1} (E_Z \star_{\varsigma_j(\ttau)}) + \mu_Z\\ 
\nabla_{\xi Q\partial_Q} & = \xi Q\partial_Q + z^{-1} (\imath^*\xi\star_{\varsigma_j(\ttau)}) + z^{-1} (\xi Q\partial_Q \varsigma_j(\ttau))\star_{\varsigma_j(\ttau)},  \\ 
\nabla_{\frq \partial_\frq} & = \frq \partial_\frq - z^{-1}(r-1)^{-1} (\rho_Z\star_{\varsigma_j(\ttau)}) + z^{-1} (\frq \partial_\frq \varsigma_j(\ttau))\star_{\varsigma_j(\ttau)} 
\end{split} 
\end{align} 
where $\xi \in H^2(X)$, $\imath \colon Z\to X$ is the inclusion and the quantum products $\star$ in \eqref{eq:qconn_QDMtXla}, \eqref{eq:qconn_QDMXla}, \eqref{eq:qconn_QDMZla} are the push-forwards of those for $\tX$, $X$, $Z$ along the ring extensions \eqref{eq:extension_QDMtXla},  \eqref{eq:extension_QDMXla}, \eqref{eq:extension_QDMZla} respectively. 
The pullbacks of the quantum connection by $\tau(\ttau)$, $\varsigma_j(\ttau)$ are well-defined, as explained previously (see the discussion after \eqref{eq:Psi_restricted_to_H}). 

Summarizing the above discussions, we arrive at the following conclusion. 

\begin{theorem} 
\label{thm:decomposition_fd} 
There exist maps $H^*(\tX) \to H^*(X)$, $\ttau\mapsto \tau(\ttau)\in H^*(X)(\!(\frq^{-1})\!)[\![Q,\ttau]\!]$ and $H^*(\tX) \to H^*(Z)$, $\ttau \mapsto \varsigma_j(\ttau) \in H^*(Z)(\!(\frq^{-\frac{1}{r-1}})\!)[\![Q,\ttau]\!]$, $0\le j\le r-2$, and an isomorphism $\Psi$ of $\C[z](\!(\frq^{-1/\frs})\!)[\![Q,\ttau]\!]$-modules 
\[
\Psi \colon \QDM(\tX)^{\rm la} \to \tau^*\QDM(X)^{\rm la} \oplus \bigoplus_{j=0}^{r-2} \varsigma_j^* \QDM(Z)^{\rm la}  
\]
satisfying the following properties: 
\begin{itemize} 
\item[(1)] $\Psi$ commutes with the quantum connection; 
\item[(2)] $\Psi$ intertwines the pairing $P_\tX$ with $P_X \oplus P_Z^{\oplus (r-1)}$; 
\item[(3)] the first component of $\Psi$ $($mapping to $\tau^*\QDM(X)^{\rm la}$$)$ is homogeneous of degree zero and the second $($mapping to $\bigoplus_{j=0}^{r-2} \varsigma_j^* \QDM(Z)^{\rm la}$$)$ is homogeneous of degree $-r$; 
\item[(4)] $\Psi$ has the following asymptotics $($see $\S\ref{subsec:cohomology}$ for the notation$)$: 
\begin{align*}
\Psi(\varphi^*\phi_{X,i})|_{Q=\ttau=0} & = 
\left(\phi_{X,i} + O(\frq^{-1}), 
\left\{ q_{Z,j} (\phi_{X,i}|_Z + O(\frq^{-\frac{1}{r-1}})) \right\}_{0\le j\le r-2}\right) \\ 
\Psi(\jmath_*(p^l \pi^*\phi_{Z,m}))|_{Q=\ttau=0} 
& = \left( O(\frq^{-1}), \left\{ q_{Z,j}  (-1)^l \lambda_j^{l+1} (\phi_{Z,m}+ O(\frq^{-\frac{1}{r-1}})) \right\}_{0\le j\le r-2}\right) 
\end{align*} 
\item[(5)] the maps $\tau(\ttau)$, $\varsigma_j(\ttau)$ are homogeneous of degree $2$; 

\item[(6)] $\tau(\ttau)|_{Q=\ttau=0} =  \frq^{-1}[Z] + O(\frq^{-2})$ and $\varsigma_j(\ttau)|_{Q=\ttau=0} =  -(r-1) \lambda_j + h_{Z,j} + O(\frq^{-\frac{1}{r-1}})$; 

\item[(7)]  the Jacobian matrix of $(\tau(\ttau), \varsigma_j(\ttau))$ at $Q=\ttau=0$ is given in \eqref{eq:Jacobianmatrix} and is invertible, where $\{\ttau^\alpha\}$ are the coordinates dual to the basis $\{\kappa_\tX(c_\alpha)\}=\{\varphi^*\phi_{X,i}, \jmath_*(p^l \pi^*\phi_{Z,m})\}$ of $H^*(\tX)$ with $c_\alpha=\{c_i,c_{l,m}\}$ given in \eqref{eq:basis_QDMTW}.  
\end{itemize} 
Here the formal Laurent forms $\QDM(\tX)^{\rm la}$, $\tau^*\QDM(X)^{\rm la}$, $\varsigma_j^*\QDM(Z)^{\rm la}$ are described in \eqref{eq:QDMla} and \eqref{eq:qconn_QDMtXla}--\eqref{eq:qconn_QDMZla}, 
$\frs$ is given in \eqref{eq:frs},  
$q_{Z,j}, h_{Z,j}$ are given in \eqref{eq:qZj_hZj}, $\lambda_j = e^{-\frac{2\pi \iu}{r-1} (j+ \frac{r}{2})}\frq^{\frac{1}{r-1}}$, and $\{\phi_{X,i}\}, \{\phi_{Z,m}\}$ are homogeneous bases of $H^*(X)$, $H^*(Z)$ respectively. 
\end{theorem}

\subsection{Initial conditions for the decomposition} 
\label{subsec:initial_conditions} 
In this section, we discuss the initial conditions for the decomposition $\Psi$ and the maps $\tau(\ttau), \varsigma_j(\ttau)$ in Theorem \ref{thm:decomposition_fd} along $Q=\ttau=0$. Borrowing the idea of Katzarkov-Kontsevich-Pantev-Yu \cite{Kontsevich:Miami2020, Kontsevich:Simons2021} and Hinault-Yu-Zhang-Zhang \cite{HYZZ:framing}, we show that $\Psi$, $\tau(\ttau)$, $\varsigma_j(\ttau)$ can be uniquely reconstructed from these initial conditions and genus-zero Gromov-Witten invariants of $X$ and $Z$. This, in turn, provides an algorithm to reconstruct genus-zero Gromov-Witten invariants of $\tX$ from those of $X$ and $Z$. 

\begin{remark} 
We note that the restriction of $\QDM(\tX)$ to $Q= \ttau=0$, where we consider the initial condition for the decomposition, can be computed by counting only exceptional rational curves that are contracted to points under $\varphi \colon \tX \to X$. The remaining variable $\frq$ represents the class of the exceptional line. The function $\sfF_\tX(M_W(\theta) 1)|_{\QW=0}$ in Lemma \ref{lem:sfF_MW_c} serves as an $I$-function for this ``exceptional'' quantum $D$-module $\QDM(\tX)|_{Q=0}$.  
\end{remark} 

\subsubsection{Initial conditions} 
Recall that the decomposition $\Psi$ in Theorem \ref{thm:decomposition_fd} arises from the restriction to the finite-dimensional slice $\sfH\subset H^*_T(W)$ of the map $\bPsi$ in Theorem \ref{thm:decomposition}, which was defined as the composition $\Phi\circ (\FT_\tX\sphat)^{-1}$. The initial value $\Psi|_{Q=\ttau =0}$ equals $\bPsi|_{\QW=\vartheta = 0}$, as the point $\vartheta =0$ in $\sfH$ corresponds to the point $\ttau=0$ in $H^*(\tX)$ under the invertible change of variables $\ttau = \ttau(\vartheta)$ in \eqref{eq:ttau_tau_fin} (see \eqref{eq:asymptotics_mirrormaps}).  Since $(\FT_\tX\sphat)^{-1}|_{\QW=\vartheta=0}$ is a constant endomorphism sending the basis $\{\varphi^*\phi_{X,i}, \jmath_*(p^l \pi^*\phi_{Z,m})\}$ to the basis $\{c_\alpha\} = \{c_i,c_{l,m}\}$ in \eqref{eq:basis_QDMTW} by Lemma \ref{lem:image_under_FTtX_limit}, it suffices to determine the action of $\Phi|_{\QW=\vartheta=0}$ on these basis. 
Recall that $\Phi = \FT_X\sphat \oplus \bigoplus_{j=0}^{r-2} \FT_{Z,j}\sphat$. Using the commutative diagrams \eqref{eq:FT_sfF},
\eqref{eq:FT_scrF}, Lemma \ref{lem:sfF_MW_c} and the fact that $M_W(\vartheta) | _{\QW=\vartheta=0} = \id$, $M_X(\tau)|_{\QW=0} = e^{\tau/z}$, $M_Z(\sigma)|_{\QW=0} = e^{\sigma/z}$  (see also Remark \ref{rem:QW=0_subtlety}), we have 
\begin{align}
\label{eq:explicit_initial_condition}  
\begin{split} 
e^{\tau(\vartheta)/z} \FT_X\sphat(c_\alpha) |_{\QW=\vartheta=0} & = \kappa_X(c_\alpha) + \sum_{k>0} \frq^{-k}
\imath_*\left( \frac{\prod_{\nu=1}^{k-1} e_{-\nu z}(\cN_{Z/X})}{k!z^k}  \left[i_Z^*c_\alpha \right]_{\lambda=kz} \right), \\ 
e^{\sigma_j(\vartheta)/z} \FT_{Z,j}\sphat(c_\alpha) |_{\QW=\vartheta=0} & = \frq^{\rho_Z/((r-1)z)} \scrF_{Z,j} (c_\alpha). 
\end{split} 
\end{align} 
These power series belong to $H^*(X)[z,z^{-1}][\![\frq^{-1}]\!]$, $\frq^{-\frac{r}{2(r-1)}} \frq^{\frac{n_\alpha}{r-1}}  H^*(Z)[z,z^{-1}][\![\frq^{-\frac{1}{r-1}}]\!]$ respectively, where $n_\alpha = 0$ if $c_\alpha = c_i$ and $n_\alpha = l+1$ if $c_\alpha = c_{l,m}$ (see \eqref{eq:leadingterm_scrF}). Note that the right-hand sides of \eqref{eq:explicit_initial_condition} can be explicitly calculated by \emph{purely topological data}, such as the Chern classes of $\cN_{Z/X}$, without requiring curve counts (see \S\ref{subsubsec:formal_asymptotics} for $\scrF_{Z,j}$). 

When $c_\alpha=1$, the right-hand sides of \eqref{eq:explicit_initial_condition} are of the forms $1+O(\frq^{-1})$, $q_{Z,j} e^{h_{Z,j}/z}(1+O(\frq^{-\frac{1}{r-1}}))$ respectively and we can take their logarithms with respect to the cup product structure on $H^*(X)$, $H^*(Z)$. Because $\FT_X\sphat(1)$, $\FT_{Z,j}\sphat(1)$ are non-negative power series in $z$, the initial conditions $\tau(\ttau)|_{Q=\ttau=0} = \tau(\vartheta)|_{\QW= \vartheta=0}$, $\varsigma_j(\ttau)|_{Q=\ttau=0} = -(r-1)\lambda_j + \sigma_j(\vartheta)|_{\QW= \vartheta = 0}$ for the mirror maps are given by 
\begin{align}
\label{eq:initial_condition_mirrormaps} 
\begin{split} 
\tau^\circ:=\tau(\ttau)|_{Q=\ttau=0} & = [z^{-1}] \log \left( 1 + \sum_{k>0} \frq^{-k}
\frac{\imath_*(\prod_{\nu=1}^{k-1} e_{-\nu z}(\cN_{Z/X}))}{k!z^k}  \right), \\
\varsigma_j^\circ :=\varsigma_j(\ttau)|_{Q=\ttau=0} & = -(r-1) \lambda_j + [z^{-1}] \log  \left( \frq^{\rho_Z/((r-1)z)} \scrF_{Z,j}(1) \right) 
\end{split} 
\end{align} 
where $[z^{-1}](\cdots)$ means the coefficient of $z^{-1}$.  
The initial conditions for $\Psi$: 
\begin{align*} 
\Psi(\varphi^*\phi_{X,i})|_{Q=\ttau=0} & = \left(\FT_X\sphat(c_i), \{\FT_{Z,j}\sphat(c_i)\}_{0\le j\le r-2} \right)|_{\QW= \vartheta=0} \\ 
\Psi(\jmath_*(p^l \pi^*\phi_{Z,m}))|_{Q=\ttau=0} & = \left(\FT_X\sphat(c_{l,m}), \{\FT_{Z,j}\sphat(c_{l,m})\}_{0\le j\le r-2}\right)|_{\QW= \vartheta = 0}
\end{align*}  
are then determined by \eqref{eq:explicit_initial_condition}.  We denote it by 
\begin{equation} 
\label{eq:Psi_circ} 
\Psi^\circ = \Psi|_{Q=\ttau=0} \in \Hom(H^*(\tX), H^*(X)\oplus H^*(Z)^{\oplus (r-1)})[z](\!(\frq^{-1/\frs})\!). 
\end{equation} 

\subsubsection{Reconstruction} 
We can reconstruct the decomposition $\Psi$ from the above initial conditions $\tau^\circ,\varsigma_j^\circ,\Psi^\circ$, as follows. Using the initial condition \eqref{eq:initial_condition_mirrormaps}, we write 
\[
\tau(\ttau) = \tau^\circ + t (\ttau), \quad \varsigma_j(\ttau) = \varsigma_j^\circ + s_j(\ttau) 
\]
with $t(\ttau) \in H^*(X)(\!(\frq^{-1})\!)[\![Q,\ttau]\!], s_j(\ttau) \in H^*(Z)(\!(\frq^{-\frac{1}{r-1}})\!)[\![Q,\ttau]\!]$. The map $\ttau \mapsto (t(\ttau), s_0(\ttau),\dots,s_{r-2}(\ttau))$ gives a formal invertible change of variables over $\C(\!(\frq^{-\frac{1}{r-1}})\!)[\![Q]\!]$. Let $(t,s)=(t,s_0,\dots,s_{r-2}) \mapsto \ttau(t,s)$ denote the inverse map. Consider the external direct sum of quantum $D$-modules of $X$ and $Z$: 
\begin{align*} 
\cM& := (H^*(X)\oplus H^*(Z)^{\oplus (r-1)})\otimes \C[z](\!(\frq^{-1/\frs})\!)[\![Q,t,s_0,\dots,s_{r-2}]\!] \\
& \cong \QDM(X)^{\rm la} \boxplus \bigboxplus_{j=0}^{r-2} \QDM(Z)^{\rm la}   
\end{align*} 
where we identify the parameter $\tau\in H^*(X)$ for $\QDM(X)^{\rm la}$ and the parameter $\varsigma_j\in H^*(Z)$ for the $j$th summand $\QDM(Z)^{\rm la}$ as follows: 
\begin{equation} 
\label{eq:shift_of_coordinates}
\tau = \tau^\circ + t, \quad \varsigma_j = \varsigma_j^\circ + s_j. 
\end{equation} 
This shift of coordinates is well-defined as the structures of $\QDM(X)^{\rm la}$, $\QDM(Z)^{\rm la}$ can be reduced to smaller rings, as we discussed after \eqref{eq:Psi_restricted_to_H}. 
By Theorem \ref{thm:decomposition_fd}, we know that $\cM$ is isomorphic to the pullback $\ttau^*\QDM(\tX)^{\rm la}$ of $\QDM(\tX)$ by the map $(t,s) \mapsto \ttau(t,s)$: 
\begin{equation} 
\label{eq:isom_Psi} 
\Psi \colon \ttau^*\QDM(\tX)^{\rm la} \overset{\cong}{\longrightarrow} \cM. 
\end{equation} 
The isomorphism $\Psi|_{t=s=Q=0}$ is given by the initial condition $\Psi^\circ$ in \eqref{eq:Psi_circ}.  
We claim that the isomorphism $\Psi$ and the map $\ttau(t,s)$ can be reconstructed from the initial conditions $\Psi^\circ$, $\tau^\circ$, $\varsigma_j^\circ$ and the data $\QDM(X)$, $\QDM(Z)$. In the coordinate system $(t,s,Q)$ and with respect to the standard trivialization, the connection on $\cM$ is of the form: 
\[
z\nabla = z d + \begin{pmatrix} 
dt\star_\tau &  & & \\ 
& ds_0 \star_{\varsigma_0} & & \\
& & \ddots & \\ 
& & & ds_{r-2} \star_{\varsigma_{r-2}} 
\end{pmatrix} 
+ \sum_{i} 
\begin{pmatrix} 
\xi_i \star_\tau &  & & \\ 
& \imath^*\xi_i \star_{\varsigma_0} & & \\ 
& & \ddots & \\ 
& && \imath^*\xi_i \star_{\varsigma_{r-2}} 
\end{pmatrix} 
\frac{d Q^{d_i}}{Q^{d_i}} 
\]
where we expand $t = \sum_i t^i \phi_{X,i}$, $s_j = \sum_j s^i_j \phi_{Z,i}$ and write $dt \star_\tau = \sum_i dt^i (\phi_{X,i}\star_\tau)$, $ds_j\star_{\varsigma_j} = \sum_i ds^i_j (\phi_{Z,i}\star_{\varsigma_j})$ with the coordinate shift \eqref{eq:shift_of_coordinates} understood; $\{d_i\} \subset N_1(X)$ and $\{\xi_i\} \subset N^1(X)$ are mutually dual basis modulo torsions. 
Here we disregard the covariant derivatives with respect to the other variables $\frq$ and $z$. 
The above connection has nilpotent residues along $Q=0$ and hence admits a unique fundamental solution $M \in \End(H^*(X)\oplus H^*(Z)^{\oplus (r-1)})\otimes \C[z^{-1}](\!(\frq^{-1/\frs})\!)[\![Q,t,s]\!]$ satisfying 
\[
M|_{t=s=Q=0} = \id, \quad M \circ z \nabla = (zd + \xi dQ/Q) \circ M  
\]
where 
\[
\xi dQ/Q := \sum_i \begin{pmatrix} 
\xi_i \cup &  & & \\ 
& \imath^*\xi_i \cup & & \\ 
& & \ddots & \\ 
& && \imath^*\xi_i \cup 
\end{pmatrix} 
\frac{d Q^{d_i}}{Q^{d_i}}.  
\]
More explicitly, such a fundamental solution is given by\footnote{The term $(r-1) \lambda_j$ is added to $\varsigma_j^\circ$ and $\varsigma_j = \varsigma_j^\circ + s_j$ to ensure that the factors $e^{(\varsigma_j^\circ+(r-1)\lambda_j)/z}$ and $M_Z(\varsigma^j + (r-1)\lambda_j)$ are both well-defined over the ring $\C[z^{-1}](\!(\frq^{-1/\frs})\!)[\![Q,s]\!]$. See \eqref{eq:leadingterms_composedmirrormaps}. }
\[
M = 
\begin{pmatrix} 
e^{-\tau^\circ/z} M_X(\tau) &  & &  & \\
& \ddots & & & \\ 
& &  e^{-(\varsigma_j^\circ+(r-1)\lambda_j)/z}M_Z(\varsigma_j+(r-1)\lambda_j)  & \\
& & & \ddots 
\end{pmatrix} 
\]
where the coordinate shift \eqref{eq:shift_of_coordinates} is used and the fundamental solutions $M_X$, $M_Z$ of $X$ and $Z$ are pushed-forward along the ring extensions \eqref{eq:extension_QDMXla}, \eqref{eq:extension_QDMZla} respectively. 

The map $\Psi$ in \eqref{eq:isom_Psi} induces a fundamental solution $M\circ \Psi$ for $\ttau^*\QDM(\tX)^{\rm la}$. Consider $M' = (\Psi^\circ)^{-1} \circ M \circ \Psi \in \End(H^*(\tX))\otimes \C[z,z^{-1}](\!(\frq^{-1/\frs})\!)[\![Q,t,s]\!]$. Since $\Psi^\circ$ does not depend on $(t,s,Q)$, this is again a fundamental solution for $\ttau^*\QDM(\tX)^{\rm la}$ satisfying 
\[
M'|_{t=s=Q=0} = \id, \quad M' \circ z \ttau^*\nabla = (zd + \xi dQ/Q) \circ M' 
\]
where $\ttau^*\nabla$ denotes the quantum connection on $\ttau^*\QDM(\tX)^{\rm la}$ in the $(t,s,Q)$-direction (again ignoring the directions of $\frq$ and $z$). Since the connection matrices for $\ttau^*\nabla$ contain only negative powers of $z$ and $\ttau^*\nabla|_{z=\infty} = d$, it follows that $M'$ satisfies $M' = \id + O(z^{-1})$. We find that $M'$ and $\Psi^{-1}$ can be uniquely determined from the equation 
\[
M' \circ \Psi^{-1} = (\Psi^\circ)^{-1} \circ M  
\]
together with the condition that $M' = \id + O(z^{-1})$ and $\Psi^{-1}$ contains no negative powers of $z$. Viewing $z$ as a parameter of loops, we can interpret this equation as a \emph{Birkhoff factorization} of the loop group element $(\Psi^\circ)^{-1} \circ M$; $\Psi^{-1}$ is the positive Birkhoff factor of $(\Psi^\circ)^{-1} \circ M$ (see e.g.~\cite[Lemma 5.23]{CIJ} for a Birkhoff factorization of formal loops). Therefore the decomposition $\Psi$ can be reconstructed from the initial conditions. In particular, we can reconstruct the connection on $\ttau^*\QDM(\tX)^{\rm la}$ as the gauge transform of that on $\cM$ by $\Psi$. Finally, the relations 
\[
z\ttau^*\nabla_{t^i} 1 = \partial_{t^i} \ttau, \quad z \ttau^*\nabla_{s_j^i} 1 = \partial_{s_j^i} \ttau, \quad 
z \ttau^*\nabla_{\xi Q\partial_Q} 1 =\xi +  (\xi Q\partial_Q \ttau) 
\]
with $\xi \in N^1(X)$ give differential equations for $\ttau(t,s)$ with respect to $(t,s,Q)$ and they determine $\ttau(t,s)$ uniquely from the initial condition $\ttau(t,s) |_{t=s=Q=0} = 0$. This completes the reconstruction. 

\begin{remark} 
Katzarkov-Kontsevich-Pantev-Yu \cite{Kontsevich:Miami2020, Kontsevich:Simons2021} and Hinault-Yu-Zhang-Zhang \cite{HYZZ:framing} have demonstrated the unique reconstructibility of the decomposition using an ``extension of framing'' result. This section provides the initial condition for the decomposition, as stated in Theorem \ref{thm:decomposition_fd}, and, building on the ideas from these works, offers a reconstruction algorithm using Birkhoff factorization. The use of Birkhoff factorization in this context dates back to Coates-Givental \cite{Coates-Givental} and Guest \cite{Guest:D-mod} (see also \cite{Iritani:efc, Iritani:genmir}). 
\end{remark}

\subsection{Decomposition of $F$-manifolds} 
Theorem \ref{thm:decomposition_fd} implies a decomposition of quantum cohomology $F$-manifolds in the sense of Hertling-Manin \cite{Hertling-Manin:weak, Manin:F-manifolds}. 

In general, quantum cohomology of $X$ defines the structure of an $F$-manifold on the cohomology group $H^*(X)$ viewed as a formal supermanifold over the Novikov ring \cite{Manin:Frobenius_book}. The product structure $\star$ on the tangent sheaf can be uniquely determined from the quantum connection by the following properties: 
\[
z \nabla_v \circ z \nabla_w = z \nabla_{v\star w} + O(z)
\]
where $v,w$ are tangent vector fields on $H^*(X)$ (vector fields on the $\tau$-space). 
The Euler vector field 
\[
\vec{E}_X = \partial_{c_1(X)} + \sum_{i} \left(1-\frac{\deg \phi_i}{2}\right) \tau^i \partial_{\tau^i},
\]
is determined from the quantum connection by the property that $\nabla_{z\partial_z} + \nabla_{\vec{E}_X}$ is regular at $z=0$, where we used the notation $\vec{E}_X$ in place of $E_X$ to distinguish it from the Euler vector field \eqref{eq:Euler_grading} viewed as a section of the quantum $D$-module. 
Therefore the decomposition $\Psi$ of quantum connection induces a decomposition of the underlying $F$-manifolds with Euler vector fields. This proves Corollary \ref{cor:F-manifolds}.  

\begin{figure}[t] 
\centering 
\begin{tikzpicture}[x=0.8cm, y=0.8cm]
\fill[blue, opacity=0.15] (0,0) circle (1); 
\fill[brown, opacity=0.1] (1.44,2.5) circle (0.8); 
\fill[brown, opacity=0.1] (1.44,-2.5) circle (0.8); 
\fill[brown, opacity=0.1] (-2.88,0) circle (0.8); 

\fill (0.5,0) circle (0.07); 
\fill (0.25,0.43) circle (0.07); 
\fill (-0.25, 0.43) circle (0.07); 
\fill (-0.5,0) circle (0.07); 
\fill (-0.25,-0.43) circle (0.07); 
\fill (0.25,-0.43) circle (0.07); 

\draw (2.2,0) node {$\QH(X)$}; 
\draw (3.3,2.5) node {$\QH(Z)$}; 

\fill (1.54,2.67) circle (0.07); 
\fill (1.34,2.33) circle (0.07); 

\fill (-2.68,0) circle (0.07); 
\fill (-3.08,0) circle (0.07); 

\fill (1.54,-2.67) circle (0.07); 
\fill (1.34,-2.33) circle (0.07); 
\end{tikzpicture} 
\caption{Euler eigenvalues near $Q=\ttau=0$ for $r=4$. The eigenvalues of $E_Z\star_{\varsigma_j(\ttau)}$, $j=0,\dots,r-2$ surround those of $E_X\star_{\tau(\ttau)}$ as ``satellites''. Here we assume the convergence of quantum cohomology and the maps $\tau(\ttau), \varsigma_j(\ttau)$. We also choose $\frq$ to be positive real.}
\label{fig:Euler_eigenvalues} 
\end{figure}

\begin{remark}
We can see from Theorem \ref{thm:decomposition_fd} how the eigenvalues of the quantum multiplication by the Euler vector field $E_\tX \star_\ttau$ decompose along the locus $Q=\ttau=0$. The quantum product for $X$ or $Z$ is the cup product along $Q=\ttau=0$ and therefore the eigenvalues of the Euler multiplication are given by the $H^0$ component of the Euler vector field. By the asymptotics of $\tau(\ttau)$, $\varsigma_j(\ttau)$ in Theorem \ref{thm:decomposition_fd}(6), we see that the eigenvalues of $E_X\star_{\tau(\ttau)}$ are all zero and those of $E_Z \star_{\varsigma_j(\ttau)}$ are all $-(r-1) \lambda_j$ along this locus. See Figure  \ref{fig:Euler_eigenvalues}. Note here that $(-\lambda_j)^{r-1} = - \frq$. 
The eigenvalues of the Euler multiplication correspond to critical values of the mirror Landau-Ginzburg potential and the same has been observed in \cite[Figure 10]{Iritani:discrepant} for mirrors of toric blowups. 
Gyenge-Szab\'o \cite{Gyenge-Szabo} studied asymptotics of the Euler eigenvalues associated with blowdowns to the minimal models of projective surfaces. 
\end{remark} 

\appendix 

\section{Proof of Theorem \ref{thm:divisor_reduction}}
\label{append:divisor_reduction}
In Corollary \ref{cor:Fourier_transform_JW_GIT}, we proved that Conjecture \ref{conj:reduction} holds for the GIT quotients $X, \tX$ of $W$, which are contained in $W$ as $\C^\times$-fixed divisors. Theorem \ref{thm:divisor_reduction} can be shown by essentially the same argument. Here we add a few arguments to complete the proof in the general case. 

Let $X$, $Y$ be as in Theorem \ref{thm:divisor_reduction}. Inverting the action if necessary, we may assume that the normal bundle of $Y$ has $\C^\times$-weight 1. We let $T=\C^\times$. 
We take a splitting of the exact sequence $0\to N_1(X) \to N_1^T(X) \to \Hom(\C^\times,T) = \Z \to 0$ given by the section class $\sigma_Y(1) \in N_1^{\rm sec}(E_1)\subset N_1^T(X)$ associated with the $T$-fixed divisor $Y$. This induces a dual splitting $N^1_T(X) \cong N^1(X) \oplus \Z$. It sends $\homega\in N^1_T(X)$ to $(\omega, a)$ where $\omega$ is the non-equivariant limit of $\homega$ and $a\in \Z$ is such that $\sigma_Y(1) \cdot \homega = \homega|_p = a \lambda$ for a point $p\in Y$. 
Let $\cS=\hcS^{\sigma_Y(1)}$ denote the shift operator on the Givental space (Definition \ref{def:shift_Givental}) corresponding to $(0,1) \in N_1(X)\oplus \Z \cong N_1^T(X)$. 
We have 
\begin{equation} 
\label{eq:shiftop_divisor} 
(\cS^k \bbf)_Y = \frac{\prod_{c=-\infty}^0 (\lambda + [Y] +cz)}{\prod_{c=-\infty}^{-k} (\lambda+[Y]+cz)} e^{-k z\partial_\lambda} \bbf_Y 
\end{equation} 
for $\bbf \in \cH^{\rm rat}_X$. We also have $\kappa_Y(\lambda) = -[Y]|_Y = - c_1(\cN_{Y/X})$. We write $\hS = (Q,S)$ for the formal variable of the group ring $\C[N_1^T(X)] = \C[N_1(X) \oplus \Z]$. 

We claim that the power series 
\begin{equation} 
\label{eq:I-function_append} 
I = z \sum_{k\in \Z} \kappa_Y(\cS^{-k} J_X(\tau)) S^k 
\end{equation} 
is supported on $C_Y^\vee \cap N_1^T(X)$. Note that the set of homology classes of non-constant $\C^\times$-orbit closures is finite. In particular, there are finitely many classes $\delta_1,\dots,\delta_l \in N_1(X)$ representing $\C^\times$-invariant irreducible curves meeting $Y$ but not contained in $Y$. We set $y_i = Q^{\delta_i}$ for $1\le i\le l$. 
Arguing as in the proof of Proposition \ref{prop:support_conjecture} and using \eqref{eq:shiftop_divisor}, we find that $\kappa_Y(\cS^{k}J_X(\tau))$ with $k\ge  0$ is non-vanishing only when it is divisible by $y_i^k$, i.e.~belongs to $y_i^k H^*(Y)[z,z^{-1}][\![Q]\!]$ for some $1\le i\le l$. Thus $I$ can be written as power series of $y_1 S^{-1},\dots,y_l S^{-1}$ and $S$ over the Novikov ring $\C[\![Q]\!]$. 
It now suffices to show that $(\delta_i,-1), (0,1)\in N_1(X)\oplus \Z \cong N_1^T(X)$ lies in the dual cone of $C_Y$. Let $\homega = (\omega,a)$ be an equivariant ample class in the cone $C_Y$. Take a $\C^\times$-invariant irreducible curve $C_i\cong \PP^1$ of class $\delta_i$ connecting a point $p_1 \in Y$ and another $\C^\times$-fixed point $p_2 \notin Y$. Then we have $\omega \cdot \delta_i = \int_{C_i} \homega = (\homega|_{p_1} - \homega|_{p_2})/\lambda = a - \homega|_{p_2}/\lambda$ by the localization theorem. The Hilbert-Mumford criterion says that $-\homega|_{p_1}/\lambda =-a<0$ and $-\homega|_{p_2}/\lambda >0$; hence $a>0$ and $\omega\cdot \delta_i -a >0$. This implies $(\delta_i,-1), (0,1) \in C_Y^\vee$ and the claim follows. 

The above argument shows that the power series $I$ is supported on the monoid 
\[
C_{Y,\N}^\vee := \NEN(X) + \left\langle (0,1), (\delta_i, -1) : i=1,\dots,l\right\rangle_\N \subset N_1^T(X)\]
contained in $C_Y^\vee$. 
A similar argument to that in \S\ref{subsec:cont_vs_discrete} establishes the following result, which implies Theorem \ref{thm:divisor_reduction}. 
\begin{proposition} 
The $I$-function \eqref{eq:I-function_append} has the form: 
\[
I = z M_Y(\sigma) v 
\]
for some $\sigma\in H^*(Y)[\![C^\vee_{Y,\N}]\!][\![\btau]\!]$ and $v \in H^*(Y)[z][\![C^\vee_{Y,\N}]\!][\![\btau]\!]$ satisfying $\sigma|_{\btau=0}\equiv S$, $v|_{\btau=0}\equiv 1$ modulo the closed ideal of $\C[z][\![C_{Y,\N}^\vee]\!]$ generated by $Q^d$, $d \in \NEN(X)\setminus \{0\}$ and $y_i S^{-1} = Q^{\delta_i} S^{-1}$, $i=1,\dots,l$. Here, the fundamental solution $M_Y$ for $Y$ is pushed-forward along the map $\C[\![\NEN(Y)]\!] \to \C[\![C_{Y,\N}^\vee]\!]$, $Q_Y^d \mapsto Q^{i_*d} S^{-c_1(\cN_{Y/X})\cdot d}$, induced by the dual Kirwan map, where $Q_Y$, $Q$ denote the Novikov variables of $Y$ and $X$, respectively. 
\end{proposition} 

\section{Proof of the asymptotics \eqref{eq:expansion_ISz}}
\label{append:asymptotics}

We work in the setting of the proof of Proposition \ref{prop:conti_vs_discrete}. Throughout the argument, we fix a positive real number $S>0$. Recall from \S\ref{subsubsec:formal_asymptotics} that the coefficient of $\QW^\delta\bx^m$ in $\scrF_{\tX,0}(\bbf)$ arises from the \emph{formal} asymptotic expansion of the integral 
\[
e^{-S/z} I_f(S,z) = 
\frac{e^{-S/z}}{\sqrt{2\pi z}}
\int_{S-\iu \infty}^{S+\iu \infty} \cI(\lambda,z) d\lambda 
\]
with
\[
\cI(\lambda,z) :=S^{\lambda/z}  G_\tX f(\lambda,z) = \frac{S^{\lambda/z}}{\sqrt{-2\pi z}} (-z)^{(\lambda -[D])/(-z)} \Gamma\left(\frac{\lambda-[D]}{-z}\right)  f(\lambda,z)
\]
via the Stirling approximation for the $\Gamma$-function. 
We need to show that the procedure in \S\ref{subsubsec:formal_asymptotics} computes the \emph{actual}  asymptotic expansion as $z$ approaches zero from the negative real axis. 
We may assume that $f(\lambda,z)\in H^*(\tX)\otimes \C(\lambda,z)_{\rm hom}$ is homogeneous of degree $\ell$ with respect to the cohomology grading and $\deg \lambda = \deg z=2$. 
The Stirling approximation gives the asymptotic expansion of the form 
\begin{equation} 
\label{eq:Stirling_append}
\cI(\lambda,z) 
 \sim e^{\eta(\lambda)/z} \sum_{i=-M}^\infty c_i(\lambda) z^i 
\end{equation} 
with $\eta(\lambda) = \lambda - \lambda \log \lambda + \lambda \log S$ and $c_i(\lambda)\in \lambda^{-1/2} H^*(\tX)[\lambda,\lambda^{-1},\log \lambda]$. 
The asymptotic series $\sum_{i=-M}^\infty c_i(\lambda) z^i$ is the product of the inverse of the quantum Riemann-Roch operator $\Delta= \Delta_{(\cN_{\tX/W},e_\lambda^{-1})}$ (see \eqref{eq:QRR_operator} and Remark \ref{rem:Stirling}) and the Laurent expansion of $f(\lambda,z)$ at $z=0$; 
note that $\Delta^{-1}$ is a Laurent series of $z$ bounded in the negative direction. 
The asymptotic expansion \eqref{eq:Stirling_append} is valid when $z$ approaches zero in the angular sector $\arg(-\lambda/z) \in (-\pi,\pi)$. 
Write 
\[
\cI(\lambda,z) =e^{\eta(\lambda)/z} \left(\sum_{i=-M}^N c_i(\lambda) z^i +R_N(\lambda,z)\right). 
\]
The standard stationary phase method (or Laplace's method) shows that the asymptotic expansion (as $z\nearrow 0$) of the integral 
\[
\frac{e^{-S/z}}{\sqrt{2\pi z}}\int_{S-\iu \infty}^{S+\iu \infty} 
e^{\eta(\lambda)/z} \sum_{i=-M}^N c_i(\lambda) z^i d\lambda
\]
coincides with the coefficient of $Q^\delta \bx^m$ in $\scrF_{\tX,0}(\bbf)$ up to order $z^N$. Hence it suffices to show that 
\begin{equation} 
\label{eq:error_term} 
\frac{e^{-S/z}}{\sqrt{2\pi z}}\int_{S-\iu \infty}^{S+\iu \infty} 
e^{\eta(\lambda)/z} R_N(\lambda,z) d\lambda = O(z^{N+1}) 
\end{equation} 
as $z\nearrow 0$. To show this, we use the following homogeneity property 
\[
e^{-\eta(\lambda)/z}\cI(\lambda,z) = \lambda^{\frac{[D]}{z}+\frac{\ell-1}{2}}  
\lambda^{\frac{\deg}{2}} (e^{-\eta(1) \lambda/z} \cI(1,z/\lambda))  
\]
where $\deg \in \End(H^*(\tX))$ is the ordinary degree operator on $H^*(\tX)$. 
Using the asymptotic expansion \eqref{eq:Stirling_append} for $e^{-\eta(1)/u} \cI(1,u)$ together with the above homogeneity, we find that the remainder $R_N(\lambda,z)$ satisfies the estimate 
\[
|R_N(\lambda,z)|\le C (1+|\Im \lambda|^n) |z|^{N+1}
\]
for some $C, n>0$, whenever $z\in [-1,0)$ and $\lambda$ is on the integration contour $\Re(\lambda) = S$. Note that $\arg(-\lambda/z)$ lies in the range $[-\frac{\pi}{2},\frac{\pi}{2}]$ where the Stirling approximation is valid. 
The rest of the argument is standard. We can find constants $\alpha,\beta>0$ such that for all $t\in \R$, 
\[
\Re(\eta(S+\iu t)) -S = t \arctan(t/S) \ge 
\begin{cases} 
\alpha t^2 & \text{if $|t|\le 1$;} \\ 
\beta |t| & \text{if $|t|\ge 1$.} 
\end{cases} 
\]
Then we have, for $-1\le z<0$, 
\begin{align*} 
\frac{\left|\text{LHS of \eqref{eq:error_term}} \right|}{|z|^{N+1}}  
& \le \frac{1}{\sqrt{2 \pi (-z)}}\left( \int_{-1}^1 e^{\alpha t^2/z} 2C dt + 2 
\int_1^\infty e^{\beta t/z} C (1+t^n) dt \right)  \\ 
& \le \frac{2C}{\sqrt{2\alpha}} + 2 C \frac{e^{\beta/z}}{\sqrt{2\pi(-z)}} \int_0^\infty e^{\beta s/z} (1+(1+s)^n) ds 
\end{align*} 
where the second term in the last line goes to zero as $z\nearrow 0$. The conclusion follows.

\bibliographystyle{amsplain}
\bibliography{qcoh_monoidal}

\end{document}